\documentclass[11pt,oneside,reqno]{amsart}
\usepackage{geometry}
\usepackage[utf8]{inputenc}
\usepackage{amstext,latexsym,amsbsy,amsmath,amssymb,amsthm,mathtools,relsize,geometry,enumerate}
\usepackage{hyperref}
\usepackage[shortlabels]{enumitem}

\hypersetup{
  colorlinks   = true, 
  urlcolor     = blue, 
  linkcolor    = red, 
  citecolor   = red 
}
\usepackage{mathtools,enumerate,enumitem}

\usepackage[capitalize,nameinlink,noabbrev]{cleveref}

\usepackage{graphicx}
\usepackage{epstopdf}
\setkeys{Gin}{width=\linewidth,totalheight=\textheight,keepaspectratio}
\graphicspath{{./images/}}

\usepackage{multicol}
\usepackage{booktabs}
\usepackage{mathrsfs}
\usepackage{color}

\newcommand{\nbb}{\mathbb{N}}
\newcommand{\rbb}{\mathbb{R}}
\newcommand{\zbb}{\mathbb{Z}}

\renewcommand{\L}{\mathcal{L}}
\newcommand{\Lhat}{\widehat{\L}}

\renewcommand{\H}{\mathcal{H}}
\newcommand{\Pcal}{\mathcal{P}}
\newcommand{\B}{\mathcal{B}}
\newcommand{\la}{\langle}
\newcommand{\ra}{\rangle}

\newcommand{\grad}{\nabla}

\newcommand{\Hzeromu}{\mathcal{H}^0_\mu}
\newcommand{\Honemu}{\mathcal{H}^1_\mu}
\newcommand{\Htwomu}{\mathcal{H}^2_\mu}
\newcommand{\Hthreemu}{\mathcal{H}^3_\mu}
\newcommand{\Hmmu}{\mathcal{H}^m_\mu}

\newcommand{\Mzeromu}{M^0_\mu}
\newcommand{\Monemu}{M^1_\mu}
\newcommand{\Mtwomu}{M^2_\mu}
\newcommand{\Mthreemu}{M^3_\mu}

\newcommand{\Mmmu}{M^m_\mu}

\newcommand{\Mdelta}{\mathcal{M}_\delta}

\newcommand{\numu}{\nu^\mu}

\newcommand{\etatilde}{\tilde{\eta}}

\newcommand{\uhat}{\widehat{u}}
\newcommand{\etahat}{\widehat{\eta}}

\newcommand{\De}{\Delta}
\newcommand{\lap}{\De}

\newcommand{\x}{U_0}

\newcommand{\bdy}{\partial}

\newcommand{\mi}{\wedge}

\newcommand{\Tr}{\text{Tr}}

\renewcommand{\d}{\text{d}}

\newcommand{\domain}{\mathcal{O}}

\newcommand{\f}{\varphi}

\newcommand{\E}{\mathbb{E}}
\newcommand{\Ecal}{E}
\renewcommand{\P}{\mathbb{P}}

\newcommand{\textred}[1]{\textcolor{red}{#1}}

\newcommand{\Z}{\mathcal{Z}}
\newcommand{\T}{\mathbb{T}}
\newcommand{\Tcal}{\mathcal{T}}
\newcommand{\dom}{\text{Dom}}

\newcommand{\nhat}{\hat{n}}

\newcommand{\close}{\!\!\!}

\theoremstyle{plain}
\newtheorem{theorem}{Theorem}[section]

\newtheorem{lemma}[theorem]{Lemma}

\newtheorem{definition}[theorem]{Definition}

\theoremstyle{definition}

\newtheorem{remark}[theorem]{Remark}

\usepackage[pagewise]{lineno}

\numberwithin{equation}{section}

\title{Existence and higher regularity of statistically steady states for the stochastic Coleman-Gurtin equation}

\author{Nathan E.~ Glatt-Holtz$^1$, Vincent R. Martinez$^{2}$ and Hung D.~Nguyen$^3$}

\address{$^1$ Department of Statistics, Indiana University, Bloomington, IN, USA}
\address{$^2$ Department of Mathematics \& Statistics, CUNY Hunter College, Department of Mathematics, CUNY Graduate Center, New York, NY, USA}
\address{$^3$ Department of Mathematics, University of Tennessee, Knoxville, TN, USA}

\begin{document}

\maketitle

\begin{abstract}
We study a class of semi-linear differential Volterra equations with polynomial-type potentials that incorporates the effects of memory while being subjected to random perturbations via an additive Gaussian noise. We show that for a broad class of non-linear potentials, the system always admits invariant probability measures. However, the presence of memory effects precludes access to compactness in a typical fashion. In this paper, this obstacle is overcome by introducing functional spaces adapted to the memory kernels, thereby allowing one to recover compactness. Under the assumption of sufficiently smooth noise, it is then shown that the statistically stationary states possess higher-order regularity properties dictated by the structure of the nonlinearity. This is established through a control argument that asymptotically transfers regularity onto the solution by exploiting the underlying Lyapunov structure of the system in a novel way.
\end{abstract}


\section{Introduction}

\label{sec:intro}
Let $\domain\subset \rbb^d$ be a bounded open domain with smooth boundary, {where $d\geq1$}. We are interested in the following semilinear stochastic Volterra equation, which is written in non-dimensional variables:
\begin{equation} \label{eqn:react-diff:K}
\begin{aligned}
&\d u(t)=\kappa\Delta u(t)\d t +(1-\kappa)\int_{0}^{\infty}\close K(s)\Delta u(t-s)\d s\d t+\f(u(t))\d t+Q\d w(t),\\
&u(t)\big|_{\partial\domain}=0,\qquad u|_{(-\infty,0]}=u_0.
\end{aligned}
\end{equation}
This equation was introduced to describe the {evolution} of a {scalar} field $u(t)=u(t,x):[0,\infty)\times \domain\to\rbb$, such as heat, in a viscoelastic medium. On the right--hand side of~\eqref{eqn:react-diff:K}, $\f:\rbb\to\rbb$ {represents the potential, typically given as a polynomial nonlinearity that satisfies} certain dissipative conditions, $Q\d w(t)$ is a Gaussian process which is delta correlated (white) in time and {whose spatial correlation is characterized by the operator $Q$},
$\kappa\in(0,1)$ defines the relative contribution of memory terms, and $K:\rbb^+\to\rbb^+$, where $\rbb^+=[0,\infty)$, denotes a memory kernel, which regulates the extent to which the past can affect the present; it is assumed to be a smooth function of exponential-type that satisfies
\begin{equation} \label{cond:K:exponential}
K'+\delta K\leq 0,
\end{equation}
for some constant $\delta>0$.

In the absence of stochastic forcing, that is, when $Q\equiv 0$, there are many works investigating the large-time behavior of \eqref{eqn:react-diff:K} by means of deterministic global attractors  \cite{conti2005singular,conti2006singular,
gatti2005memory,
gatti2004exponential}. On the other hand, when noise is present, the theory of random attractors was investigated in the works \cite{caraballo2007existence,caraballo2008pullback}. Notably, the results in \cite{conti2005singular,conti2006singular} establish that the Hausdorff dimension of the global attractor is dictated by the smoothness of the nonlinearity. In contrast, the investigation of the statistically steady states of the stochastic system \eqref{eqn:react-diff:K}, {as represented by the invariant probability measures corresponding to the Markovian dynamics of its history-augmented system (see \eqref{eqn:react-diff:mu}),} seemed to receive less attention. {Recently, in \cite{glatt2022short}, under the assumption that the system is directly excited by random forcing on a sufficiently large subspace of the component of phase space corresponding only to the temperature of the history-augmented system,} stationary solutions of \eqref{eqn:react-diff:K} are shown to be good approximations of the {stationary solutions of the} corresponding stochastic nonlinear heat equation, that is, without memory. However, in the absence of a condition on the number of directions which are directly forced, the issue of the existence of statistically steady states nevertheless remains a challenge, owing to the hyperbolic nature of system's capacity to retain past information.

The main goal of the work is thus two-fold. First, within a broad class of nonlinear potentials, we demonstrate that the system always admits invariant probability measures. Secondly, we seek to establish a relation between the support of the invariant measures and the regularities of the noise and nonlinear potentials. In what follows, we provide a detailed overview of the main mathematical results of the article.

\subsection{Overview of the main results} \label{sec:intro:results}
 We note that the process $u(t)$ governed by \eqref{eqn:react-diff:K} is not Markovian owing to the presence of memory effects. It is therefore not clear what is meant by an invariant probability measure for the system \eqref{eqn:react-diff:K} at this stage. In order to make sense of statistically steady states, we augment the original process by an auxiliary variable such that Markovianity is recovered. This framework was originally introduced in \cite{dafermos1970asymptotic} in the study of long-time behavior of a viscoelastic system and later exploited in \cite{conti2006singular} for studying singular limits of systems with memory. The main idea is to introduce a ``history variable", $\eta(t,s)$ that is subsequently appended to the original process, $u(t)$, to form a joint process $(u(t),\eta(t,s))$ that is Markovian in an extended phase space. More precisely, we introduce the \emph{integrated past history} of $u(t)$ defined by
\begin{equation} \label{form:eta}
\eta(t,s):=\int_0^s u(t-r)\d r,\quad s,t {\ \geq\ } 0.
\end{equation}
Observe that $\eta$ satisfies the following inhomogeneous transport equation
\begin{align*}
\partial_t\eta(t,s)=-\partial_s\eta(t,s)+u(t).
\end{align*}
To see the role of $\eta$ in \eqref{eqn:react-diff:K}, set
\begin{align}\label{def:mu:K}
\mu(s):=-K'(s).
\end{align}
Upon integrating by parts with respect to $s$ in the memory term appearing in \eqref{eqn:react-diff:K}, then invoking \eqref{form:eta}, one obtains
\begin{equation} \label{eqn:integration-by-part}
\int_{0}^{\infty}\close K (s)\Delta u(t-s)\d s=\int_{0}^{\infty}\close \mu(s)\Delta\eta(t,s)\d s.
\end{equation}
This allows one to rewrite \eqref{eqn:react-diff:K} as the following extended system, whose dynamics are now Markovian:
\begin{equation} \label{eqn:react-diff:mu:original}
\begin{aligned}
{\d} u(t)&=\kappa\Delta u(t){\d t}+(1-\kappa)\int_0^\infty\close \mu(s)\Delta\eta(t,s)\d s {\d t} +\f(u(t))\d t+Q{\d} w(t),\\
\quad
\partial_t \eta(t,s)&=-\partial_s\eta(t,s)+u(t),\quad  \\
u(t)\big|_{\partial\domain}&=0,\quad\eta(t,s)|_{\partial \domain}=0,\\
u(0)&=u_0 \text{ in }\domain,\quad \eta(0;s)=\eta_0(s).
\end{aligned}
\end{equation}
The precise details of the phase space for the extended variable $(u,\eta)$ are provided in Section \ref{sect:notation}.

The main result of the paper is the existence of invariant probability measures associated to the extended system \eqref{eqn:react-diff:mu:original} and the property that their supports belong to spaces of regularity stronger than that of the natural phase space of the system.

\begin{theorem} \label{thm:regularity:pseudo}
{Suppose that $K$ satisfies~\eqref{cond:K:exponential}, that the noise has a sufficiently smooth spatial correlation structure, and that $\f$ grows at most algebraically, that is, $\f$ satisfies}
\begin{align*}
    |\f(x)|\le c|x|^{p}+C,\quad \text{for all}\ x\in \rbb,
\end{align*}
for some positive constants $c,C,p$. Then under certain dissipativity conditions on $\f$, the extended system \eqref{eqn:react-diff:mu:original} possesses at least one invariant probability measure $\nu$. Furthermore, {when $d\le 3$ and $p<5$}, the support of any invariant probability measure $\nu$ consists of functions as smooth as $\f$ and $Q$ allow.
\end{theorem}

We emphasize that invariance refers to the invariance of $\nu$ with respect to the Markov semigroup associated to extended system \eqref{eqn:react-diff:mu:original}. We refer the reader to Theorem \ref{thm:existence} and Theorem \ref{thm:regularity} for the more precise version of Theorem \ref{thm:regularity:pseudo}. We note that the conditions imposed on the potential, $\f$, in Theorem \ref{thm:regularity:pseudo} apply to a broad class of polynomial nonlinearities, one of which is the well-known Allen-Cahn cubic potential, $\f(x)=x-x^3$.

\subsection{Previous related literature and methodology of proofs}

Stochastic differential equations with memory were studied as early as the seminal work \cite{ito1964stationary}. Since then, there have been many works concerning the theory of well-posedness {for} infinite-dimensional systems with memory {such as} \cite{barbu1975nonlinear,barbu52nonlinear,barbu1979existence,
bonaccorsi2004large,bonaccorsi2006infinite,clement1996some,clement1997white}. The existence of statistically steady states in {this} context {was addressed} for several {systems}, such as {the} stochastic Volterra equations \cite{clement1998white}, {the} Navier-Stokes equation \cite{weinan2001gibbsian}, {and the} Ginzburg–Landau equation \cite{weinan2002gibbsian}, while the {issue of uniqueness was studied in} \cite{bakhtin2005stationary,bonaccorsi2012asymptotic,
weinan2002gibbsian,weinan2001gibbsian,
glatt2020generalized,hairer2011asymptotic}. Particularly, in the work of \cite{bonaccorsi2012asymptotic} and \cite{nguyen2022ergodicity}, ergodicity of a reaction-diffusion equation with memory effects similar to \eqref{eqn:react-diff:K} was investigated. The notable difference between the system considered there (see \cite[Equation (1.1)]{bonaccorsi2012asymptotic}) from the one considered in the current article is the presence, {in \cite[Equation (1.1)]{bonaccorsi2012asymptotic}}, of a \textit{negative} convolution integral and a smallness assumption that imposes that the maximum amplitude of $K$ be dominated by the viscous drag coefficient; {in our notation, this corresponds to replacing $+(1-\kappa)K(s)$ with $-(1-\kappa)K(s)$ and assuming $K(t)<\frac{\kappa}{1-\kappa}$, respectively.} By employing an approach from \cite{engel2001one,miller1974linear}, the system considered in \cite{bonaccorsi2012asymptotic, nguyen2022ergodicity} could then be rewritten as an abstract Cauchy system similar to~\eqref{eqn:react-diff:mu:original} on product spaces. It was further shown that the system possessed unique ergodicity and that it was strongly mixing. In a companion paper \cite{glatt2022short} to the present article, the large-time behavior and singular limit of the system \eqref{eqn:react-diff:mu:original} was investigated. It was shown there that under the assumption that the noise excited sufficiently many directions in phase space \cite[Condition Q3]{glatt2022short},  \eqref{eqn:react-diff:mu:original} admitted exactly one invariant probability measure and that the system was exponentially mixing. It is important to point out that such a hypothesis is not assumed here for the purpose of establishing existence of invariant probability measures. Hence, the results of this article are complementary to those obtained in \cite{glatt2022short}. In fact, for a very general set of conditions on the noise, we will demonstrate that \eqref{eqn:react-diff:mu:original} always admits invariant probability measures whose supports consist of functions with smoothness intimately related to that of the stochastic forcing and the potential present in the system.

Turning to~\eqref{eqn:react-diff:K}, the extended phase space approach for studying~\eqref{eqn:react-diff:K} as manifested in the system ~\eqref{eqn:react-diff:mu:original} was introduced in the work \cite{dafermos1970asymptotic} and later popularized for many partial differential equations (PDEs) \cite{conti2005singular,conti2006singular,gatti2004exponential,
gatti2005memory} as well as stochastic PDEs (SPDEs) \cite{caraballo2007existence,caraballo2008pullback,li2019asymptotic,
liu2017well,liu2019asymptotic,
shangguan2024geometric,shu2020asymptotic}. Such a Markovian approach can also be found in SDEs, e.g., in Langevin dynamics with memory \cite{glatt2020generalized,
ottobre2011asymptotic,pavliotis2014stochastic}. As mentioned in \cite{glatt2022short}, the advantage of this method is that one may rewrite the noiseless counterpart of \eqref{eqn:react-diff:mu:original} as an autonomous system of evolution equations on product spaces (see \eqref{eqn:react-diff:mu} below). In turn, this provides one access to the classical Markovian framework to study statistically invariant structures of~\eqref{eqn:react-diff:mu:original}. On the other hand, the analytic trade-off in introducing the history variable, $\eta$, can be found in the hyperbolic structure of its governing evolution equation, which subsequently precludes access to compactness in the traditional way. We show that we can nevertheless overcome this difficulty by establishing sufficient regularity of the solution in the proper functional setting.

A crucial property that is leveraged in this article, as well as in many others \cite{bonaccorsi2012asymptotic,clement1997white,
glatt2022short, ottobre2011asymptotic,pavliotis2014stochastic}, is the assumption that the memory kernel decays exponentially into the past (see \eqref{cond:K:exponential} and \nameref{cond:mu}). For stochastic equations with memory decaying non-exponentially, e.g., sub--exponential or power-law, we refer the reader to \cite{baeumer2015existence, desch2011p,
 glatt2020generalized, nguyen2018small}. We emphasize that in the regime of exponential-type memory kernels that are considered in this article, the results we establish are able to accommodate a wider class of nonlinear potentials $\f$ with minimal growth conditions, and are compatible with results regarding their deterministic counterpart.

 Recall that {the} main result of the paper is to establish the existence and regularity of {invariant probability measures} of \eqref{eqn:react-diff:mu:original} (see Theorem \ref{thm:regularity:pseudo}). In particular, we show that the system \eqref{eqn:react-diff:mu:original} admits at least {one} invariant {probability} measure $\nu$. Furthermore, in spatial dimensions {physical dimensions} $d\le 3$, {we show these invariant probabilities are in fact supported in spaces of higher regularity relative to that of the phase space}. The strategy that we employ for establishing existence is the classical Krylov-Bogoliubov procedure, which invokes tightness for a sequence of time-average measures. {However, as mentioned in the preceding paragraph, due to the lack of compactness from the memory variables, we resort to the techniques developed \cite{conti2005singular,conti2006singular} to effectively recover it.}

 On the other hand, the regularity of the invariant probabilities is {established} through a series of bootstrap {arguments that carefully exploits the smoothing property} of solutions inherent to the system, coupled with a control strategy, originally introduced in \cite{glatt2021long} and developed further in \cite{nguyen2022ergodicity}, which asymptotically guides solutions towards a region of phase space whose spatial regularity is regulated by the spatial regularity possessed by the external driving forces. {The bootstrap argument furnishes higher-order moment bounds by leveraging control over lower-order moment bounds. This manifests explicitly in our analysis by establishing a recursive relation on Lyapunov-type functionals corresponding to higher Sobolev-type norms in terms of those corresponding to lower-order Sobolev-type norms.} The rough idea {behind the control argument, on the other hand,} is as follows: denote an invariant probability measure of \eqref{eqn:react-diff:mu:original} by $\nu^\mu$ and the ``controlled process" by $(\uhat,\etahat)$. Let $|\cdotp|$ denote the norm of the extended phase space and $\|\cdotp\|$ denote a stronger norm of a subspace of the extended phase space. Suppose that the controlled process $(\uhat,\etahat)$, with \textit{smooth initial value} $(\uhat_0,\etahat_0)$, possesses time-uniform bounds in expectation with respect to the stronger norm $\|\cdotp\|$, and moreover, has the property that $(\uhat,\etahat)$ asymptotically converges to $(u,\eta)$ with respect to the weaker norm $|\cdotp|$, in expectation, for \textit{any initial data}. Lastly, let $P_N$ denote the $N$-dimensional Galerkin projection. Then, proceeding heuristically, by invariance of $\nu^\mu$ one may argue
    \begin{align*}
        \int\|(P_Nu_0,P_N\eta_0)\|d\nu^\mu&= \int\E\|(P_Nu(t;u_0),P_N\eta(t;\eta_0)\|d\nu^\mu\notag
        \\
        &\leq C_{N}\int\E|(u(t;u_0)-\uhat(t;\uhat_0),\eta(t,\eta_0)-\etahat(t;\etahat_0)|d\nu^\mu\\
        &\quad\ +\int\E\|(\uhat(t;\uhat_0),\etahat(t;\etahat_0))\|d\nu^\mu.\notag
    \end{align*}
The first integral converges to $0$ as $t\rightarrow\infty$, while the second integral is bounded uniformly in time. Since one may always choose $t$ accordingly to $N$, one obtains
    \begin{align}
        \sup_N\int\|(P_Nu_0,P_N\eta_0)\|d\nu^\mu<\infty,\notag
    \end{align}
from which one may then deduce that $\nu^\mu$ is supported in the smoother subspace. We explicitly construct a control with the above properties in Section \ref{sec:regularity} below. The key idea is introduce a control which enforces convergence in the phase space, but also retains the Lyapunov structure of the original system. Ultimately, the control we introduce is an affine perturbation of the original system that enforces convergence only on a sufficiently large finite-dimensional subspace. Thus, our control strategy is designed in such a way that it effectively transfers the regularity of the controlled process to the solution in the time-asymptotic limit. {From this perspective, the regularity result of the current article} can also be considered as a stochastic analogue of results in deterministic settings regarding the regularity of the global attractor \cite{bonaccorsi2006infinite}. These results are stated precisely in Theorem \ref{thm:existence} and Theorem \ref{thm:regularity} below. Their proofs are provided in Section \ref{sec:existence} and Section \ref{sec:regularity}, respectively.

\subsection{Organization of the paper}

The rest of the paper is organized as follows: in Section \ref{sec:main-result}, we {review} the {precise} functional setting {that we work in}. Particularly, we will {formulate} \eqref{eqn:react-diff:mu:original} as an abstract Cauchy {problem} \eqref{eqn:react-diff:mu} on an appropriate product space. We also identify the main assumptions that we make on the memory, the nonlinear potentials and the noise structure. Then, we {state} our main results in this section, including Theorem \ref{thm:existence} on the existence and Theorem \ref{thm:regularity} on {the} regularity of invariant probability measures. In Section \ref{sec:apriori-moment-estimate}, we perform a priori moment bounds on the solutions {that ultimately ensures the existence and regularity of the invariant probability measures}. In Section \ref{sec:existence} and Section \ref{sec:regularity}, respectively, we prove the main results of the paper, regarding the existence and regularity of an invariant probability. The paper concludes with two appendices, Section \ref{sec:well-posed} and Section \ref{sec:auxiliary-result}. In Section \ref{sec:well-posed}, we supply the details for the construction of unique pathwise solutions via the standard Galerkin approximation to \eqref{eqn:react-diff:mu:original} while in Section \ref{sec:auxiliary-result}, we collect several {technical} auxiliary results that are invoked in proving the regularity of invariant probability measures.

\section{Assumptions and statements of main results}\label{sec:main-result}

In this section, we state our main results and detail the various structural assumptions we impose on \eqref{eqn:react-diff:mu} for them. In Section \ref{sect:notation}, we review the functional settings and the phase spaces for \eqref{eqn:react-diff:mu:original} \cite{conti2005singular,conti2006singular,
glatt2022short}. In Section \ref{sect:well-posed}, we state the well-posedness of the system \eqref{eqn:react-diff:mu} through Theorem \ref{thm:well-posed}. In Section \ref{sect:exist}, we state the existence of invariant probabilitym measures through Theorem \ref{thm:existence} and the regularity of the support of these invariant probability measures through Theorem \ref{thm:regularity}.

\subsection{Functional setting}\label{sect:notation}

Given a bounded open set $\domain$ in $\rbb^d$ with smooth boundary, we let $L^p(\domain)$, for $1\leq p\leq\infty$ denote the usual Lebesgue spaces.  In the particular case $p=2$, we let $H=L^2(\domain)$. We denote the {corresponding} inner product and norm in $H$ by $\langle\,\cdot\,,\,\cdot\,\rangle_H$ and $\|\cdot\|_H$, {respectively}.

Let $A$ denote the (negative) Dirichlet Laplacian operator, {$-\De_D$}. It is well-known that there exists a complete orthonormal basis $\{e_k\}_{k\ge 1}$ in $H$ that diagonalizes $A$, i.e., there exists a positive sequence $0<\alpha_1<\alpha_2<\dots$ diverging to infinity such that
\begin{equation} \label{eqn:Ae_k=-alpha.e_k}
Ae_k=\alpha_k e_k, \quad k\ge 1.
\end{equation}

More generally, for each $r\in\rbb$, we denote by $H^r$, the domain of $A^{r/2}$ endowed with the inner product, i.e., $H^r:=D(A^{r/2})$, (see \cite{cerrai2020convergence,conti2006singular}). Then the corresponding inner product is defined by
$$\langle u,v\rangle_{H^r}=\sum_{k\ge 1}\alpha_k^{r}\la u,e_k\ra_H\la v,e_k\ra_H,$$
so that the corresponding induced norm is given by
$$\|u\|_{H^r}^2=\sum_{k\geq 1}\alpha_k^{r}\langle u,e_k\rangle^2_H.$$

Next, we {introduce} the notion of extended phase spaces {that was developed in} \cite{
conti2005singular,conti2006singular,
dafermos1970asymptotic}. {This will establish the formal framework in which we construct solutions to} \eqref{eqn:react-diff:mu:original}. First, given a memory kernel $\mu:[0,\infty)\to[0,\infty)$, we define the following weighted Hilbert spaces
\begin{equation}
M^\beta_{\mu}=L^2_{\mu}([0,\infty);H^{\beta+1}), \qquad \beta\in\rbb,
\end{equation}
endowed with the inner product
\begin{align}\label{def:Mnorm}
\langle \eta_1,\eta_2\rangle_{M^\beta_\mu}=\int_0^\infty\close \mu(s)\langle A^{(1+\beta)/2}\eta_1(s),A^{(1+\beta)/2}\eta_2(s)\rangle_H\d s.
\end{align}
It is important to note that while the usual embedding $H^{\beta_1}\subset H^{\beta_2}$, $\beta_1>\beta_2$, is compact, the embedding $M^{\beta_1}_\mu\subset M^{\beta_2}_\mu$ is only continuous \cite{pata2001attractors}. In order to establish existence of invariant probability measures, this defect will require us to introduce additional spaces, which we develop now.

Next, let $\Tcal_\mu$ be the operator on $\Mzeromu$ defined by
\begin{align} \label{form:Tcal}
\Tcal_\mu\eta := -\partial_s\eta,\qquad\dom(\Tcal_\mu) =\{\eta\in\Mzeromu:\partial_s\eta\in \Mzeromu,\eta(0)=0\},
\end{align}
where $\partial_s$ is the derivative in the distribution sense. In other words, $\Tcal_\mu$ is the infinitesimal generator of the right-translation semigroup acting on $\Mzeromu$ \cite[Theorem 3.1]{grasselli2002uniform}. Furthermore, if $u\in L^1_{\text{loc}}([0,\infty);H^1)$ then the following functional formulation of the Cauchy initial-value problem
\begin{equation} \label{eqn:eta:Cauchy-problem}
\left\{
\begin{aligned}
\frac{\d}{\d t}\eta(t)&=\Tcal_\mu\eta(t)+u(t),\\
\eta(0)&=\eta_0\in\Mzeromu,
\end{aligned}
\right.
\end{equation}
has a unique solution $\eta\in C([0,\infty);\Mzeromu)$ with the following representation: \cite{conti2006singular,grasselli2002uniform}
\begin{equation} \label{form:eta(t)-representation}
\eta(t,s;\eta_0)=\begin{cases} \int_0^s u(t-r)\d r,& 0<s\le t,\\
\eta_0(s-t)+\int_0^t u(t-r)\d r,& s>t.
\end{cases}
\end{equation}
The above explicit formula will be helpful when we show the existence of invariant probability measures in Theorem \ref{thm:existence} below. Another significance of $\Tcal_\mu$ is the following useful estimate \cite[Theorem 3.1]{grasselli2002uniform} that will be employed throughout {our analysis below}: for $\eta\in\dom(\Tcal_\mu)$, observe that integrating by parts gives
\begin{align}
\la \Tcal_\mu\eta,\eta\ra_{\Mzeromu}&=-\frac{1}{2}\int_0^\infty\close \mu(s)\partial_s\|A^{1/2}\eta(s)\|^2_H\d s\notag \\
&=\frac{1}{2}\int_0^\infty\close \mu'(s)\|A^{1/2}\eta(s)\|^2_H\d s\notag \\
&\le -\frac{1}{2}\delta\|\eta\|_{\Mzeromu}^2. \label{ineq:<T.eta,eta>}
\end{align}
Similarly, we may derive the following bound in $M^\beta_\mu$ for any $\beta\in\rbb$
\begin{equation} \label{ineq:<T.eta,eta>_(M^n)}
\begin{aligned}
\la \Tcal_\mu\eta,\eta\ra_{M^\beta_\mu}\le  -\frac{1}{2}\delta\|\eta\|_{M^\beta_\mu}^2.
\end{aligned}
\end{equation}

Furthermore, given $\eta\in \Mzeromu$, we introduce the {\it tail function}
\begin{equation}\label{form:tailfunction}
\T_\eta^\mu(r) = \int_{(0,\frac{1}{r})\cup(r,\infty)}\close\close\close\|A^{1/2}\eta(s)\|_H^2\mu(s)\d s,\qquad r\geq 1.
\end{equation}
Then we define the Banach space for $\beta\in\rbb$
\begin{equation} \label{form:space:L}
\mathcal{E}^\beta_\mu =\{\eta\in M^\beta_\mu:\eta\in\dom(\Tcal_\mu),\quad\sup_{r\geq 1}r\T_\eta^\mu(r)<\infty\},
\end{equation}
with the norm {defined by}
\begin{equation} \label{form:space:L:norm}
\|\eta\|_{\Ecal^{\beta}_\mu}^2 = \|\eta\|_{M^\beta_\mu}^2+\|\Tcal_\mu\eta\|^2_{\Mzeromu}+\sup_{r\geq 1}r\T_\eta^\mu(r).
\end{equation}

Having introduced these ``memory spaces", we define for $\beta\in\rbb$ the Banach spaces
\begin{equation} \label{form:space:H_epsilon}
\H^\beta_\mu=H^\beta\times M^\beta_\mu,\qquad \Z^\beta_\mu=H^\beta\times \Ecal^\beta_\mu.
\end{equation}
It was shown in \cite{conti2006singular, gatti2004exponential,joseph1989heat,joseph1990heat, pata2001attractors} that $\Ecal^\beta_\mu$ is compactly embedded into $M^0_\mu$ for $\beta>0$, (see \cite[Lemma 3.1]{conti2006singular} and \cite[Lemma 5.5]{pata2001attractors}) so that any bounded set in $\Z^\beta_\mu$ is totally bounded in $\H^\beta_\mu$. Moreover, for $(u,\eta)\in \H^\beta_\mu$ (or $\Z^\beta_\mu$), we denote by $\pi_i,\,i=1,2,$ the projection on marginal spaces, namely
$$\pi_1(u,\eta)=u,\qquad \pi_2(u,\eta)=\eta.$$
Also, for $n\ge 1$, we denote by $P_n$ the projection of $(u,\eta)$ onto the {subspace spanned by the first $n$ eigenfunctions $e_k$}:
\begin{align} \label{form:P_Nu}
P_nu=\sum_{k=1}^n\la u,e_k\ra_He_k,\quad\text{and}\quad P_n\eta(s) =\sum_{k=1}^n\la \eta(s),e_k\ra_H e_k.
\end{align}

Finally, we may now recast \eqref{eqn:react-diff:mu:original} as follows:
\begin{align}
\d u(t)&=-\kappa A u(t)\d t-(1-\kappa)\int_0^\infty\close \mu(s)A\eta(t,s)\d s \d t+\f(u(t))\d t+Q\d  w(t),\notag \\
\frac{\d}{\d t}\eta(t)&= \Tcal_\mu\eta(t)+u(t), \notag \\
( u(0),\eta(0))&=(u_0,\eta_0)\in\Hzeromu. \label{eqn:react-diff:mu}
\end{align}

\subsection{Well-posedness}\label{sect:well-posed}
In this subsection, we discuss the well-posedness of \eqref{eqn:react-diff:mu}. We first start with the condition on the memory kernel $\mu$.
\subsection*{(M1)}\label{cond:mu} Let $\mu\in C^1([0,\infty))$ be a positive function such that
\begin{equation} \label{ineq:mu}
\mu'+\delta\mu\le 0,
\end{equation}
for some $\delta>0$.

{Regarding the noise, we assume that} $w(t)$ is a cylindrical Wiener process on $H$, {whose decomposition is given by}
$$w(t)=\sum_{k\ge 1}e_kB_k(t),$$
where $\{e_k\}_{k\ge 1}$ is the orthonormal basis of $H$ as in \eqref{eqn:Ae_k=-alpha.e_k} and $\{B_k(t)\}_{k\ge 1}$ is a sequence of independent {standard} one-dimensional Brownian motions, each defined on the same stochastic basis $\mathcal{S}=(\Omega, \mathcal{F},\{\mathcal{F}_t\}_{t\ge 0},\P)$ \cite{karatzas2012brownian}. Concerning the linear operator $Q$, we impose the following assumption \cite{bonaccorsi2012asymptotic,cerrai2020convergence,da2014stochastic,glatt2017unique}:
\subsection*{(Q1)}\label{cond:Q}
$Q:H\to H$ is a symmetric, non-negative, bounded linear map {such that}
\begin{align*}
\Tr(QAQ)<\infty,\quad \text{and}\quad \sup_{x\in\domain}\sum_{k\ge 1}|Qe_k(x)|^2<\infty.
\end{align*}
{In the above, we recall that $\Tr(QAQ)=\sum_{k\ge 1}\la QAQe_k,e_k\ra_H =\sum_{k\ge 1}\|Qe_k\|^2_{H^1}$.}

\begin{remark} We note that condition $ \sup_{x\in\domain}\sum_{k\ge 1}|Qe_k(x)|^2<\infty$ is required for the well--posedness \cite{bonaccorsi2012asymptotic,da2014stochastic}. We do not explicitly make use of this condition for the large--time asymptotic analysis of~\eqref{eqn:react-diff:mu}.

\end{remark}

Finally, concerning the {potential, $\f:\rbb\to\rbb$}, we impose the following  conditions:

\subsection*{(P0)}\label{cond:phi} $\f\in C^1$ satisfies $\f(0)=0$.

\subsection*{(P1)}\label{cond:phi:1} There exist positive constants $a_1$ and $p_0>1$ such that for all $x\in\rbb$,
$$ |\f(x)|\le a_1(1+|x|^{p_0}).$$

\subsection*{(P2)}\label{cond:phi:2} There exist positive constants $a_2,a_3$ such that for all $x\in\rbb$,
$$x\f(x)\le -a_2|x|^{p_0+1}+a_3,$$
where $p_0$ is the same constant from \textbf{(P1)}.

\subsection*{(P3)}\label{cond:phi:3} The derivative $\f'$ satisfies
$$\sup_{x\in\rbb}\f'(x)=:a_\f<\infty.$$

Fixing a stochastic basis $\mathcal{S}=(\Omega, \mathcal{F},\{\mathcal{F}_t\}_{t\ge 0},\P)$, let us now state what we mean by a ``weak solution" of \eqref{eqn:react-diff:mu}  (see \cite{glatt2008stochastic}).
\begin{definition} \label{defn:mild-soln}
Given initial condition $(u_0,\eta_0)\in \Hzeromu$, a process $U(\cdot)=\big(u(\cdot),\eta(\cdot)\big)$ is called a weak solution of \eqref{eqn:react-diff:mu} if $u(\cdotp)$ is $\mathcal{F}_t$-adapted and there exists $q\geq1$ such that $\P$--a.s. one has
\begin{align*}
u\in C_{w}([0,\infty);H)\cap L^2_{\emph{loc}}([0,\infty);H^1), \quad \eta\in C({[0,\infty)};\Mzeromu),
\end{align*}
and
\begin{align*}
\f(u)\in L^q_{\emph{loc}}([0,\infty);L^q(\domain)),
\end{align*}
{where $C_{w}([0,\infty);H)$ denotes the space of functions that are weakly continuous in $H$ with respect to $t$. Moreover,} for $\P$--a.s.
\begin{align*}
\la u(t),v\ra_H&=\la u_0,v\ra_H  -\kappa\int_0^t\la u(r),v\ra_{H^1} \emph{d} r-(1-\kappa)\int_0^t\la \eta(r),v\ra_{\Mzeromu}\\
&\quad+\int_0^t\la \f(u(r)),v\ra_H \emph{d} r+\int_0^t\la v(r),Q\emph{d} w(r)\ra_H,\\
\la \eta(t),\etatilde\ra_{\Mzeromu}&=\la\eta_0,\etatilde\ra_{\Mzeromu}+  \int_0^t\la \Tcal_\mu\eta(r),\etatilde\ra_{\Mzeromu} \emph{d}r+\int_0^t\la u(r),\etatilde\ra_{\Mzeromu}\emph{d} r,
\end{align*}
{holds for all $v\in H^1\cap L^{q'}(\domain)$, where $q'\geq1$ is the H\"older conjugate of $q$, and for all $\etatilde\in \Mzeromu$.}
\end{definition}

The first result of this note is the following well-posedness result ensuring the existence and uniqueness of weak solutions.
\begin{theorem} \label{thm:well-posed}
{Assume that \nameref{cond:mu}, \nameref{cond:Q} and \nameref{cond:phi}--\nameref{cond:phi:3} hold}. Then for all $U_0\in\Hzeromu$, \eqref{eqn:react-diff:mu} admits a unique weak solution $U(\,\cdot\,;\x)$ in the sense of Definition \ref{defn:mild-soln}. Furthermore, the solution $U(t;U_0)$ is continuous in $H$ with respect to $U_0$, for all fixed $t\ge 0$, i.e.,
\begin{align*}
\E\|U(t;U^{n}_0 )-U(t;U_0)\|^2_{\Hzeromu}\to 0,\quad\text{as  }n\to \infty,
\end{align*}
whenever $\|U_0^n-U_0\|_{\Hzeromu}\to 0$ as $n\to\infty$.
\end{theorem}
The method that we employ to construct solutions is the well-known Faedo-Galerkin approximation and can be found in many previous works for SPDE. {We refer the reader to \cite{albeverio2008spde,
caraballo2007existence,caraballo2008pullback,
glatt2008stochastic}, for instance}. For the sake of completeness, we supply the relevant details for the proof of Theorem \ref{thm:well-posed} in~Section \ref{sec:well-posed}.

\subsection{Existence of invariant probability measures}\label{sect:exist}

Under the assumptions of the well-posedness result Theorem \ref{thm:well-posed}, we may define Markov transition probabilities corresponding to the process $U(t;U_0)$ satisfying \eqref{eqn:react-diff:mu} given by
\begin{align*}
P_t^{\mu}(U_0,A):=\P(U(t;\x)\in A),
\end{align*}
for each $t\geq 0$, $U_0\in\Hzeromu$, and Borel sets $A\subseteq \Hzeromu$. We let $\B_b(\Hzeromu)$ denote the set of bounded Borel measurable functions defined on $\Hzeromu$. The Markov semigroup associated to \eqref{eqn:react-diff:mu} is the operator $P_t^{\mu}:\B_b(\Hzeromu)\to\B_b(\Hzeromu)$ defined by
\begin{align}\label{form:P_t^epsilon}
P_t^{\mu} f(U_0)=\E[f(U(t;\x))], \quad f\in \B_b(\Hzeromu).
\end{align}
Recall that a probability measure $\nu\in \Pcal r(\Hzeromu)$ is said to be \emph{invariant} for the semigroup $P_t^{\mu}$ if for every $f\in \B_b(\Hzeromu)$
\begin{align*}
\int_{\Hzeromu} f(U_0) (P_t^{\mu})^*\nu(\d U_0)=\int_{\Hzeromu} f(U_0)\nu(\d U_0),
\end{align*}
where $(P_t^{\mu})^*\nu$ denotes the push-forward measure of $\nu$ by $P_t^{\mu}$, i.e.,
$$\int_{\Hzeromu}f(\x)(P^\mu_t)^*\nu(\d \x)=\int_{\Hzeromu}P^\mu_t f(\x)\nu(\d \x).$$
Next, we denote by, $\L^\mu$, the generator associated to system \eqref{eqn:react-diff:mu}. One defines $\L^\mu$ for any $g\in C^2(\Hzeromu)$ with $g=g(U)$, $U=(u,\eta)$, satisfying
\begin{align*}
\Tr(D_{uu}gQQ^*)<\infty,
\end{align*}
by
\begin{equation} \label{form:L^epsilon}
\begin{aligned}
\L^\mu g(u,\eta)&:=-\kappa\la Au,D_u g\ra_H-(1-\kappa)\int_0^\infty\close \mu(s) \la  A\eta(s),D_u g\ra_H\d s+\la \f(u),D_u g\ra_H \\
&\quad\ +\la \Tcal_\mu\eta,D_\eta g\ra_{\Mzeromu}+\la u,{D_\eta  g}\ra_{\Mzeromu}+\frac{1}{2}\Tr(D_{uu}gQQ^*).
\end{aligned}
\end{equation}

In light of~\nameref{cond:mu}, we introduce the class, $\Mdelta$, of memory kernels defined by
\begin{equation} \label{form:M_delta}
\Mdelta = \{\mu:C^1([0,\infty);(0,\infty)):\mu'+\delta\mu\le 0\}.
\end{equation}
We now state our second main result of the paper concerning the existence and moment bounds in $\Hzeromu$ of an invariant probability measure.
\begin{theorem} \label{thm:existence}
Under the same hypothesis of Theorem \ref{thm:well-posed}, the system \eqref{eqn:react-diff:mu} admits at least one invariant probability measure $\numu$. Furthermore, for all $\beta>0$ sufficiently small independent of $\mu\in \Mdelta$ as in~\eqref{form:M_delta}, any invariant probability $\numu$ of~\eqref{eqn:react-diff:mu} satisfies
\begin{equation} \label{ineq:exponential-bound:nu^epsilon}
\sup_{\mu\in\Mdelta}\int_{\Hzeromu} \exp\left({\beta\|(u,\eta)\|^2_{\Hzeromu}}\right)\numu(\emph{d} u,\emph{d}\eta)<\infty.
\end{equation}
\end{theorem}
The existence of $\numu$ will follow the classical Krylov-Bogoliubov argument where it is sufficient to establish the tightness for a sequence of time-averaged probability measures. In order to do that, as mentioned above in Section \ref{sect:notation}, we will establish suitable moment bounds in $\Z^1_\mu$, (see \eqref{form:space:H_epsilon}), which is compactly embedded into $\Hzeromu$. The proof of Theorem \ref{thm:existence} will be presented in Section \ref{sec:existence}.

\subsection{Regularity of invariant probability measures}\label{sect:regularity} Finally, we turn to the main topic of regularity of invariant probability measures. For this purpose, we state the following conditions, which will be employed to study the support of $\numu$

\subsection*{(P4)}\label{cond:phi:d=3} Let $\f$ satisfy \nameref{cond:phi}--\nameref{cond:phi:3}. There exists $m\ge 2$ such that $\f\in C^{m-1}(\rbb)$ and that the following hold:
    \begin{enumerate}[label=\alph*.]

    \item For $i=2,\dots, 2[\frac{m}{2}]-2$, $\f^{(i)}(0)=0$.

    \item For $i=1,\dots, m-1$, there exists $p_i>0$ such that $|\f^{(i)}(x)|\le c(1+|x|^{p_i})$. Furthermore, $p_1<4$.

    \end{enumerate}

\subsection*{(Q2)}\label{cond:Q:higher_regularity}
Let $Q$ be as in \nameref{cond:Q} and $m$ be as in \nameref{cond:phi:d=3}. We assume that
\begin{align*}
\Tr(QA^mQ)<\infty,
\end{align*}
where $\Tr(QAQ)=\sum_{k\ge 1}\la QA^mQe_k,e_k\ra_H=\sum_{k\ge 1}\|Qe_k\|^2_{H^m}$.

Having introduced assumptions \nameref{cond:phi:d=3} and \nameref{cond:Q:higher_regularity} about the regularity of the non-linear potentials and noise structures, respectively, we now state the next main result asserting that any invariant probability $\numu$ in { the setting of dimension} $d\le 3$ must concentrate in spaces of higher regularity.
\begin{theorem} \label{thm:regularity}
{Let $d\leq 3$ and assume the same conditions from  Theorem \ref{thm:well-posed}. Additionally assume that \nameref{cond:Q:higher_regularity}, \nameref{cond:phi:d=3} hold} for $m\ge 2$. Then, any invariant probability $\numu$ of \eqref{eqn:react-diff:mu} satisfies
\begin{equation}\label{ineq:exponential-bound:nu(H^m)}
\sup_{\mu\in \Mdelta}\int_{\Hmmu} \exp\Big\{\beta_m\|(u,\eta)\|^{q_m}_{\Hmmu} \Big\}\numu(\emph{d}u,\emph{d}\eta)<\infty,
\end{equation}
for some positive constants $\beta_m$, $q_m\in(0,1)$. Here, $\Mdelta$ is as in~\eqref{form:M_delta}. Furthermore, for all $p\ge 1$,
\begin{equation} \label{ineq:moment-bound:nu(H^m)}
\sup_{\mu\in \Mdelta}\int_{\Hmmu} \|u\|^p_{H^m}\|u\|^2_{H^{m+1}}\numu(\emph{d} u,\emph{d}\eta)<\infty,
\end{equation}
In particular, $\numu(H^{m+1}\times \Mmmu)=1$.
\end{theorem}
We note that Theorem \ref{thm:regularity} may be considered as a stochastic analogue to \cite[Theorem 8.1]{conti2006singular} for global attractors in deterministic settings. In order to prove Theorem \ref{thm:regularity}, {we} will make use of a series of bootstrap arguments where we will inductively establish moment bounds in {higher-order} spaces $\H^k$ for $k=2,\dots,m$. The restriction $d\leq3$ arises from the need to control $L^\infty$--norms, while the restriction $p_1<4$ from \nameref{cond:phi:d=3} arises from the need to control $|\f'(x)|$ by $|x|^{p_1}$. Our approach draws upon the one taken in \cite{conti2006singular}, which dealt with analogous regularity issues. The proof of Theorem \ref{thm:regularity} will be {carried out} in Section \ref{sec:regularity}.

\section{A priori moment estimates}\label{sec:apriori-moment-estimate}

Throughout the rest of the paper, $c$ and $C$ denote generic positive constants that may change from line to line. The main parameters that they depend on will appear between parenthesis, e.g., $c(T,q)$ is a function of $T$ and $q$. First we state an exponential moment bound in $\Hzeromu$ for $U(t)$.

We introduce the function $\Psi_0$ defined as
\begin{equation} \label{form:Psi_0}
\Psi_0(u,\eta)=\frac{1}{2}\|u\|^2_{H}+\frac{1}{2}(1-\kappa)\|\eta\|^2_{\Mzeromu}.
\end{equation}

In Lemma \ref{lem:moment-bound:H^0_epsilon} below, we assert an energy estimate in $\Hzeromu$ through function $\Psi_0$ defined in \eqref{form:Psi_0}.

\begin{lemma} \label{lem:moment-bound:H^0_epsilon}
Assume the hypotheses of~Theorem \ref{thm:well-posed} and let $U_0=(u_0,\eta_0)\in \Hzeromu$. Then

\begin{enumerate}[noitemsep,topsep=0pt,wide=0pt,label=\arabic*.,ref=\theassumption.\arabic*]
\item For all $\beta$ sufficiently small independent of $\mu\in\Mdelta$,
\begin{equation}\label{ineq:exponential-bound:H^0_epsilon}
\E \exp\left({\beta\Psi_0(U(t))}\right)\le e^{-c_0t}\exp\left({\beta\Psi_0(U_0)}\right)+C_0,\quad t\ge 0,
\end{equation}
where $\Psi_0$ is as in \eqref{form:Psi_0} and $c_0=c_0(\beta)>0$, $C_0=C_0(\beta)>0$ do not depend on $U_0\in\Hzeromu$, $t\ge 0$ and $\mu\in\Mdelta$.

\item For all $n\ge 1$, there exist positive constants $c_{0,n},\tilde{c}_{0,n}$ and $C_{0,n}$ independent of $U_0\in\Hzeromu$, $t\ge 0$ and $\mu\in\Mdelta$ such that
\begin{equation}  \label{ineq:d.Psi_0^n}
\emph{d} \Psi_0(U(t))^n\le -c_{0,n} \Psi_0(U(t))^n\emph{d} t+\tilde{c}_{0,n}\emph{d} t+\emph{d} M_{0,n}(t),
\end{equation}
where $M_{0,n}(t)$ is the semi-martingale given by
\begin{equation} \label{form:M_(0,n)}
M_{0,n}(t)=\int_0^t n\Psi_0(U(r))^{n-1}\la u(t),Q\emph{d} w(r)\ra_H.
\end{equation}
Furthermore,
\begin{equation} \label{ineq:Psi_0^n}
\E\Psi_0(U(t))^n \le e^{-c_{0,n}t}\Psi_0(U_0)^n+C_{0,n},\quad t\ge 0.
\end{equation}
\end{enumerate}
\end{lemma}

\begin{proof}

We first start {with}~\eqref{ineq:exponential-bound:H^0_epsilon} and compute partial derivatives of $\Psi_0$
\begin{align*}
D_u\Psi_0 = u,\quad D_\eta \Psi_0={(1-\kappa)}\eta,\quad \text{and}\quad D_{uu}\Psi_0=Id.
\end{align*}
Recalling $\L^\mu$ as in \eqref{form:L^epsilon}, we have
\begin{align*}
\L^\mu\Psi_0(u,v)& {\ =}-\kappa\|A^{1/2}u\|^2_H-(1-\kappa)\la \eta,u\ra_{\Mzeromu}+\la \f(u),u\ra_H+\frac{1}{2}\Tr(QQ^*)\\
&\quad\ +(1-\kappa)\la \Tcal_\mu\eta,\eta\ra_{\Mzeromu} +(1-\kappa)\la u,\eta\ra_{\Mzeromu}\\
&= -\kappa\|A^{1/2}u\|^2_H+(1-\kappa)\la \Tcal_\mu\eta,\eta\ra_{\Mzeromu}+\la \f(u),u\ra_H+\frac{1}{2}\Tr(QQ^*).
\end{align*}
Recalling \nameref{cond:Q}, we readily have $\Tr(QQ^*)<\infty$. In light of \eqref{ineq:<T.eta,eta>},
\begin{align*}
\la \Tcal_\mu\eta,\eta\ra_{\Mzeromu}\le -\frac{1}{2}\delta\|\eta\|^2_{\Mzeromu}.
\end{align*}
 Using \nameref{cond:phi:2}, it holds that
\begin{align*}
\la \f(u),u\ra_H\le a_3|\domain|,
\end{align*}
where $|\domain|$ denotes the {Lebesgue measure} of $\domain$ in $\rbb^d$.
Combining the above estimates, we arrive at
\begin{align} \label{ineq:L^epsilon.Psi_0}
\L^\mu\Psi_0(u,v)&\le -\kappa\|A^{1/2}u\|^2_H-\frac{1}{2}(1-\kappa)\delta\|\eta\|^2_{\Mzeromu}+a_3|\domain|+\frac{1}{2}\Tr(QQ^*).
\end{align}
Turning to~\eqref{ineq:exponential-bound:H^0_epsilon}, we consider $g(u,\eta)=\exp\left({\beta\Psi_0(u,\eta)}\right)$. For $\xi\in\Hzeromu$, the Frechet derivatives of $g$ along the direction of $\xi$ are given by
\begin{align*}
\la D_ug(u,\eta),\pi_1\xi\ra_{H}&=\beta \exp\left({\beta\Psi_0(u,\eta)}\right)\la u,\pi_1\xi\ra_{H},\\
\la D_\eta g(u,\eta),\pi_2\xi\ra_{\Mzeromu}&=\beta(1-\kappa) \exp\left({\beta\Psi_0(u,\eta)}\right)\la\eta,\pi_2\xi\ra_{\Mzeromu},\\
\text{and}\quad D_{uu}g(u,\eta)(\xi)&=\beta \exp\left({\beta\Psi_0(u,\eta)}\right)\pi_1\xi+\beta^2\exp\left({\beta\Psi_0(u,\eta)}\right)\la u,\pi_1\xi\ra_{H} u.
\end{align*}
{Applying $\L^\mu$ to $g$ gives
\begin{align*}
\L^\mu g(u,\eta)&=\beta e^{\beta{\Psi_0}(u,\eta)}\Big(-\kappa\|A^{1/2}u\|^2_H+(1-\kappa)\la \Tcal_\mu\eta,\eta\ra_{\Mzeromu}+\la \f(u),u\ra_H \notag \\
&\quad\ +\frac{1}{2}\Tr(QQ^*)+\frac{1}{2}\beta\sum_{k\ge 1}\la u,QQ^*e_k\ra_H\la u,e_k\ra_{H}\Big)\\
&=\beta e^{\beta{\Psi_0}(u,\eta)}\Big(\L^\mu\Psi_0(u,\eta)+\frac{1}{2}\beta\sum_{k\ge 1}\la u,QQ^*e_k\ra_H\la u,e_k\ra_{H}\Big) .
\end{align*}
In view of \eqref{ineq:L^epsilon.Psi_0}, we readily have
\begin{align}\label{ineq:L^mu.g}
    \L^\mu g(u,\eta)&\le \beta e^{\beta{\Psi_0}(u,\eta)}\Big(-\kappa\|A^{1/2}u\|^2_H-\frac{1}{2}(1-\kappa)\delta\|\eta\|^2_{\Mzeromu}+a_3|\domain| \notag \\
    &\quad\ +\frac{1}{2}\Tr(QQ^*)+\frac{1}{2}\beta\sum_{k\ge 1}\la u,QQ^*e_k\ra_H\la u,e_k\ra_{H}\Big).
\end{align}
To estimate the last term on the above right hand side, we recall  \nameref{cond:Q}
\begin{align*}
\frac{1}{2}\beta\sum_{k\ge 1}\la u,QQ^*e_k\ra_H\la u,e_k\ra_{H}&=\frac{1}{2}\beta\|Qu\|^2_{H}\\&=\frac{1}{2}\beta\sum_{k\ge 1}|\la u,Qe_k\ra_H|^2\\
&\le  \frac{1}{2}\beta\Tr(QQ^*)\|u\|^2_H,
\end{align*}
which can be subsumed into $-\kappa\|A^{1/2}u\|_{H}^2$ by taking $\beta$ sufficiently small, namely,
\begin{align*}
\beta<\frac{\kappa\alpha_1}{\Tr(QQ^*)}.
\end{align*}
So, recalling $\Psi_0$ defined in \eqref{form:Psi_0}, we get from \eqref{ineq:L^mu.g}
\begin{align*}
    \L^\mu g(u,\eta)&\le \beta e^{\beta\Psi_0(u,\eta)}\Big(-\frac{1}{2}\kappa\alpha_1\|u\|^2_H-(1-\kappa)\frac{\delta}{2}\|\eta\|^2_{\Mzeromu}+a_3|\domain|+\frac{1}{2}\Tr(QQ^*)\Big)\\
   &\le \beta e^{\beta\Psi_0(u,\eta)}\Big(-c \Psi_0(u,\eta)+a_3|\domain|+\frac{1}{2}\Tr(QQ^*)\Big).
\end{align*}
We thus combine the above estimate with the identity
\begin{align*}
    \frac{\d}{\d t}\E g(U(t)) = \E\L^\mu g(U(t)),
\end{align*}
to infer the existence of positive constants $c=c(\beta,\kappa,Q,\f)$ and $C=C(\beta,\kappa,Q,\f)$ such that the following holds uniformly in $t$ and $\mu\in\Mdelta$ defined in~\eqref{form:M_delta}
\begin{align*}
\frac{\d}{\d t}\E g(U(t))=\frac{\d}{\d t}\E \,\exp\left({\beta\Psi_0(U(t))}\right)&\le -c\E \,\exp\left({\beta\Psi_0(U(t))}\right)\big(\Psi_0(U(t))-C\big).
\end{align*} }
To further bound the above right hand side, we employ the elementary fact that there exists a constant $\tilde{C}=\tilde{C}(\beta,C)>0$ such that for all $r\ge 0$,
\begin{equation} \label{ineq:e^r(r-C)>e^r-C}
e^{\beta r}(r-C)>e^{\beta r}-\tilde{C}.
\end{equation}
So, there exist $c$ and $C$ independent of $t$, $\mu\in\Mdelta$ and initial condition $U_0$ such that
\begin{equation} \label{ineq:d/dt.E.e^(beta.Psi_1(t))}
\frac{\d}{\d t}\E \,\exp\left({\beta\Psi_0(U(t))}\right)\le - c\,\E \,\exp\left({\beta\Psi_0(U(t))}\right)+C.
\end{equation}
This establishes~\eqref{ineq:exponential-bound:H^0_epsilon} by virtue of Gronwall's inequality.

With regard to~\eqref{ineq:d.Psi_0^n}--\eqref{ineq:Psi_0^n}, we shall proceed by induction on $n$. The case $n=1$ is actually a consequence of the above estimates. Indeed, by It\^o's formula and~\eqref{ineq:L^epsilon.Psi_0}, we readily have
\begin{align}
\d \Psi_0(U(t))
&=\L^\mu\Psi_0(U(t))\d t+\la u(t),Q\d w(t)\ra_H \notag\\
&\le -\kappa\|A^{1/2}u\|^2_H\d t-\frac{1}{2}(1-\kappa)\delta\|\eta\|^2_{\Mzeromu}\d t+\Big(a_3|\domain|+\frac{1}{2}\Tr(QQ^*)\Big)\d t  \notag \\
&\quad\ +\la u(t),Q\d w(t)\ra_H,\label{ineq:d.Psi_0}
\end{align}
which establishes~\eqref{ineq:d.Psi_0^n} for $n=1$. Also,
\begin{align*}
\frac{\d}{\d t}\E\Psi_0(U(t))&=-\kappa\alpha_1\E\|u(t)\|^2_H-\frac{1}{2}(1-\kappa)\delta\E\|\eta(t)\|^2_{\Mzeromu}+a_3|\domain|+\frac{1}{2}\Tr(QQ^*),
\end{align*}
whence
\begin{align*}
\E\Psi_0(U(t)) \le e^{-c_{0,1}t}\Psi_0(U_0)+C_{0,1},
\end{align*}
where
\begin{align} \label{form:c_(0,1)}
c_{0,1} = \min\{2\kappa\alpha_1,(1-\kappa)\delta\},\quad\text{and}\quad C_{0,1}= \frac{a_3|\domain|+\frac{1}{2}\Tr(QQ^*)}{c_{0,1}},
\end{align}
implying \eqref{ineq:Psi_0^n} for $n=1$.

Now consider $n\ge 2$, for $\xi\in\Hzeromu$, the Frechet derivatives of $\Psi_0(u,\eta)^n$ along the direction of $\xi$ are given by
\begin{align*}
\la D_u\Psi_0(u,\eta)^n,\pi_1\xi\ra_{H}&=n\Psi_0(u,\eta)^{n-1}\la u,\pi_1\xi\ra_{H},\\
\la D_\eta \Psi_0(u,\eta)^n,\pi_2\xi\ra_{\Mzeromu}&=n\Psi_0(u,\eta)^{n-1}\la\eta,\pi_2\xi\ra_{\Mzeromu},\\
D_{uu}\Psi_0(u,\eta)^n(\xi)&=n\Psi_0(u,\eta)^{n-1}\pi_1\xi+n(n-1)\Psi_0(u,\eta)^{n-2}\la u,\pi_1\xi\ra_{H} u.
\end{align*}
Applying $\L^\mu$, (see \eqref{form:L^epsilon}), to $\Psi_0(u,\eta)^n$ gives
\begin{align*}
&\L^\mu \Psi_0(u,\eta)^n\\
&=n\Psi_0(u,\eta)^{n-1}\Big(-\kappa\|A^{1/2}u\|^2_H+(1-\kappa)\la \Tcal_\mu\eta,\eta\ra_{\Mzeromu}+\la \f(u),u\ra_H+\frac{1}{2}\Tr(QQ^*)\Big)\\
&\quad +\frac{1}{2} n(n-1)\Psi_0(u,\eta)^{n-2}\|Qu\|^2_H.
\end{align*}
Similarly to the base case $n=1$,
\begin{align*}
\la \Tcal_\mu\eta,\eta\ra_{\Mzeromu}+\la \f(u),u\ra_H&\le -\frac{1}{2}\delta\|\eta\|^2_{\Mzeromu}+a_3|\domain|.
\end{align*}
Also,
\begin{align*}
\frac{1}{2}\|Qu\|^2_H&\le \frac{1}{2}\Tr(QQ^*)\|u\|^2_H\le \Tr(QQ^*)\Psi_0(u,\eta).
\end{align*}
It follows that
\begin{align*}
\L^\mu \Psi_0(u,\eta)^n&\le -c\Psi_0(u,\eta)^n + C\Psi_0(u,\eta)^{n-1}.
\end{align*}
By {Young's} inequality, it is clear that $\Psi_0(u,\eta)^{n-1}$ can be {absorbed} into $-c\Psi_0(u,\eta)^n$. We therefore may infer the existence of positive constants $c_{0,n}$ and $\tilde{c}_{0,n}$ such that
\begin{align}\label{ineq:L^epsilon.Psi_0^n}
\L^\mu \Psi_0(u,\eta)^n\le -c_{0,n}\Psi_0(u,\eta)^n+\tilde{c}_{0,n}.
\end{align}
As a consequence, by It\^o's formula, we arrive at
\begin{align*}
\d \Psi_0(U(t))^n&=\L^\mu\Psi_0(U(t))\d t+n\Psi_0(U(t))^{n-1}\la u(t),Q\d w(t)\ra_H\\
&\le -c_{0,n}\Psi_0(U(t))^n\d t+\tilde{c}_{0,n}\d t+n\Psi_0(U(t))^{n-1}\la u(t),Q\d w(t)\ra_H,
\end{align*}
which proves~\eqref{ineq:d.Psi_0^n}. Furthermore,
\begin{align*}
\frac{\d}{\d t}\E \Psi_0(U(t))^n&\le -c_{0,n}\E\Psi_0(U(t))^n+\tilde{c}_{0,n}.
\end{align*}
This together with Gronwall's inequality produces \eqref{ineq:Psi_0^n} for all $n\ge 2$. The proof is thus finished.
\end{proof}

We next introduce the function
\begin{equation} \label{form:Psi_1}
\Psi_1(u,\eta)=\frac{1}{2}\|u\|^2_{H^1}+\frac{1}{2}(1-\kappa)\|\eta\|^2_{\Monemu},
\end{equation}
and for general $m\in\nbb$
\begin{equation} \label{form:Psi_m}
\Psi_m(u,\eta)=\frac{1}{2}\|u\|^2_{H^m}+\frac{1}{2}(1-\kappa)\|\eta\|^2_{\Mmmu}.
\end{equation}
To prove the $\Honemu$--analog of Lemma \ref{lem:moment-bound:H^0_epsilon}, we will make use of the following elementary inequality.
\begin{lemma}\label{lem:tech:1}
{ Given $c,c'>0$, the following holds
\begin{equation*}
e^{-ct}\int_0^t e^{-(c'-c)r}\emph{d} r\le Ce^{-\tilde{c} t},\quad t\ge 0,
\end{equation*}
for some positive constants $C=C(c,c'),\tilde{c}=\tilde{c}(c,c')$ independent of $t$.}
\end{lemma}
\begin{proof}
There are three cases depending on the sign of $c'-c$. If $c'-c>0$, then
\begin{align*}
e^{-ct}\int_0^t e^{-(c'-c)r}\d r\le \frac{e^{-ct}}{c'-c}.
\end{align*}
Otherwise, if $c'-c< 0$
\begin{align*}
e^{-ct}\int_0^t e^{-(c'-c)r}\d r= e^{-c't}\int_0^t e^{-(c-c')(t-r)}\d r\le \frac{e^{-c't}}{c-c'}.
\end{align*}
Now if $c'-c=0$,
\begin{align*}
e^{-ct}\int_0^t e^{-(c'-c)r}\d r= e^{-ct}t\le \frac{2e^{-ct/2}}{c}.
\end{align*}
Altogether, we observe that
\begin{equation}\notag
e^{-ct}\int_0^t e^{-(c'-c)r}\d r\le Ce^{-\tilde{c}t}.
\end{equation}
\end{proof}

\begin{lemma} \label{lem:moment-boud:H^1_epsilon}
Assume the hypotheses of~Theorem \ref{thm:well-posed} and let $U_0=(u_0,\eta_0)\in \Honemu$. Then
\begin{enumerate}
[noitemsep,topsep=0pt,wide=0pt,label=\arabic*.,ref=\theassumption.\arabic*]
\item For all $\beta$ sufficiently small independent of $\mu\in\Mdelta$,
\begin{equation}\label{ineq:exponential-bound:H^1_epsilon}
\E \exp\left({\beta\Psi_1(U(t))}\right)\le  e^{-c_{1,0}t}\exp\left({\beta C_{1,0}\Psi_1(U_0)}\right)+C_{1,0},\quad t\ge 0,
\end{equation}
where $\Psi_1(u,\eta)$ is as in \eqref{form:Psi_1} and $c_{1,0}=c_{1,0}(\beta)>0$, $C_{1,0}=C_{1,0}(\beta)>0$ do not depend on $U_0$, $\mu\in\Mdelta$ and $t\ge 0$.

\item For all $n\ge 1$, there exist positive constants $c_{1,n}$, $C_{1,n}$ independent of $U_0$, $\mu\in\Mdelta$ and $t\ge 0$ such that
\begin{equation}  \label{ineq:d.Psi_1^n}
\emph{d} \Psi_1(U(t))^n\le -c_{1,n} \Psi_1(U(t))^n\emph{d} t+C_{1,n}\emph{d} t+C_{1,n}\Psi_0(U(t))^{n}\emph{d} t+\emph{d} M_{1,n}(t),
\end{equation}
where $M_{1,n}(t)$ is the semi-martingale given by
\begin{equation} \label{form:M_(1,n)}
M_{1,n}(t)=\int_0^t n\Psi_1(U(r))^{n-1}\la A^{1/2}u(t),A^{1/2}Q\emph{d} w(r)\ra_H.
\end{equation}
Furthermore,
\begin{equation} \label{ineq:Psi_1^n}
\E\Psi_1(U(t))^n \le C_{1,n}e^{-c_{1,n}t}\Psi_1(U_0)^n+C_{1,n},\quad t\ge 0.
\end{equation}
\end{enumerate}
\end{lemma}

\begin{proof} We first prove part 1. For this, we apply $\L^\mu$, (see \eqref{form:L^epsilon}), to $\Psi_1(u,\eta)$ to obtain
\begin{align*}
\L^\mu\Psi_1(u,\eta)&=  -\kappa\|Au\|^2_{H}+(1-\kappa)\la \Tcal_\mu\eta,\eta\ra_{\Monemu}+\la \f'(u)\grad u,\grad u\ra_{H}+\frac{1}{2}\Tr(QAQ^*).
\end{align*}
In view of \nameref{cond:Q}, we readily have
\begin{align*}
\Tr(QAQ^*)<\infty.
\end{align*}
Recalling \eqref{ineq:<T.eta,eta>_(M^n)} (for $\beta=1$), it holds that
\begin{align*}
\la \Tcal_\mu\eta,\eta\ra_{\Monemu}\le -\frac{1}{2}\delta\|\eta\|^2_{\Monemu}.
\end{align*}
To deal with the nonlinear term, note that by \nameref{cond:phi:3} and Cauchy-Schwarz inequality,
\begin{align*}
\la \f'(u)\grad u,\grad u\ra_{H} \le a_\f\la A^{1/2}u,A^{1/2}u\ra_{H}=a_\f \la u,Au\ra_{H}\le \frac{2a_\f^2}{\kappa}\|u\|^2_H+\frac{1}{2}\kappa\|Au\|^2_{H}.
\end{align*}
It follows that
\begin{align}
\L^\mu\Psi_1(u,\eta)&\le  -\frac{1}{2}\kappa\|Au\|^2_{H}-\frac{1}{2}(1-\kappa)\delta\|\eta\|^2_{\Monemu}+\frac{2a_\f^2}{\kappa}\|u\|^2_H+\frac{1}{2}\Tr(QAQ^*) \notag \\
&\le -c_{1,1}\Psi_1(u,\eta)+C_{1,1}\Psi_0(u,\eta)+C_{1,1}.\label{ineq:L^epsilon.Psi_1}
\end{align}

Turning to~\eqref{ineq:exponential-bound:H^1_epsilon}, for $\beta_{1,0}$ to be chosen later, we consider
\begin{equation} \label{form:g_1(u,eta)}
g_1(u,\eta)=\Psi_1(u,\eta)+\beta_{1,0}\Psi_0(u,\eta).
\end{equation}
In view of~\eqref{ineq:L^epsilon.Psi_0} and~\eqref{ineq:L^epsilon.Psi_1}, observe that
\begin{align*}
\L^\mu g_1(u,\eta)&=  \L^\mu\Psi_1(u,\eta)+\beta_{1,0}\L^\mu \Psi_0(u,\eta)\\
&\le -c_{1,1}\Psi_1(u,\eta)+C_{1,1}\Psi_0(u,\eta)+C_{1,1}\\
&\quad\ +\beta_{1,0}\big(-c_{0,1}\Psi_0(u,\eta)+ C_{0,1}\big).
\end{align*}
By picking $\beta_{1,0}$ sufficiently large (independent of $\mu\in\Mdelta$), we obtain
\begin{align*}
\L^\mu g_1(u,\eta)\le -c g_1(u,\eta)+C.
\end{align*}
{Similarly to the proof of~\eqref{ineq:exponential-bound:H^0_epsilon}, we compute
\begin{align*}
\L^\mu e^{\beta g_1(u,\eta)}&= \beta e^{\beta g_1(u,\eta)}\L^\mu g_1(u,\eta)+\frac{1}{2}\beta^2e^{\beta g_1(u,\eta)}\sum_{k\ge 1}\big| \la u,Qe_k\ra_{H^1}+\beta_{1,0}\la u,Qe_k\ra_H \big|^2\\
&\le \beta e^{\beta g_1(u,\eta)}\Big(  -c g_1(u,\eta)+C+\frac{1}{2}\beta\sum_{k\ge 1}\big| \la u,Qe_k\ra_{H^1}+\beta_{1,0}\la u,Qe_k\ra_H \big|^2 \Big).
\end{align*}
By~\nameref{cond:Q}, we have
\begin{align*}
\frac{1}{2}\sum_{k\ge 1}\big| \la u,Qe_k\ra_{H^1}+\beta_{1,0}\la u,Qe_k\ra_H \big|^2&\le \sum_{k\ge 1}|\la u,Qe_k\ra_{H^1}|^2+\sum_{k\ge 1}\beta_{1,0}^2|\la u,Qe_k\ra_H|^2 \\
&\le \Tr(QAQ^*)\|A^{1/2}u\|^2_H+\beta_{1,0}^2\Tr(QQ^*)\|u\|^2_H\\
&\le \Tr(QAQ^*)\Psi_1(u,\eta)+\beta_{1,0}^2\Tr(QQ^*)\Psi_0(u,\eta)\\
&\le \tilde{C}\,g_1(u,\eta),
\end{align*}
for some constant $\tilde{C}>0$ independent of $\beta$ and $\mu\in\Mdelta$. It follows that
\begin{align*}
\L^\mu e^{\beta g_1(u,\eta)}\le \beta e^{\beta g_1(u,\eta)}\big(-c g_1(u,\eta)+\beta \tilde{C}g_1(u,\eta)+C\big).
\end{align*}
Since $c$ and $\tilde{C}$ are both independent of $\beta$,} by choosing $\beta_1$ sufficiently small, observe that for all $\beta\in (0,\beta_1)$, we obtain the bound
\begin{align*}
\L^\mu e^{\beta g_1(u,\eta)}\le \beta e^{\beta g_1(u,\eta)}\big(-c g_1(u,\eta)+C\big).
\end{align*}
We now invoke~\eqref{ineq:e^r(r-C)>e^r-C} to deduce further
\begin{align*}
\L^\mu e^{\beta g_1(u,\eta)}\le -c e^{\beta g_1(u,\eta)}+C.
\end{align*}
By It\^o's formula, this yields
\begin{align*}
\frac{\d}{\d t}\E \,e^{\beta g_1(U(t))}\le -c \, \E\, e^{\beta g_1(U(t))}+C,
\end{align*}
implying
\begin{align*}
\E \,e^{\beta g_1(U(t))} \le e^{-ct}\E \,e^{\beta g_1(U_0)}+C.
\end{align*}
Recalling the expression~\eqref{form:g_1(u,eta)} of $g_1$,
\begin{align*}
\Psi_1\le g_1\le \left(1+\frac{\beta_{1,0}}{\alpha_1}\right)\Psi_1,
\end{align*}
we immediately obtain~\eqref{ineq:exponential-bound:H^1_epsilon}. Now we prove 2.

With regard to~\eqref{ineq:d.Psi_1^n}--\eqref{ineq:Psi_1^n}, we proceed by induction as in the proof of~Lemma \ref{lem:moment-bound:H^0_epsilon}. For the base case $n=1$, from~\eqref{ineq:L^epsilon.Psi_1}, we see that
\begin{align}
\d \Psi_1(U(t))&= \L^\mu\Psi_1(U(t))\d t+ \la A^{1/2}u(t), A^{1/2}Q\d w(t)\ra_H \notag\\
&\le -\frac{1}{2}\kappa\|Au(t)\|^2_{H}\d t-\frac{1}{2}(1-\kappa)\delta\|\eta(t)\|^2_{\Monemu}\d t+\frac{2a_\f^2}{\kappa}\|u(t)\|^2_H\d t\notag\\
&\quad\ +\frac{1}{2}\Tr(QAQ^*)\d t+ \la A^{1/2}u(t), A^{1/2}Q\d w(t)\ra_H. \label{ineq:d.Psi_1}
\end{align}
This proves~\eqref{ineq:Psi_1^n} for $n=1$. Also, from~\eqref{ineq:L^epsilon.Psi_1}, we obtain using Sobolev embedding
\begin{align*}
&\frac{\d}{\d t}\E\Psi_1(U(t))\\
&\le -\frac{1}{2}\kappa\E\|Au(t)\|^2_{H}-\frac{1}{2}(1-\kappa)\delta\E\|\eta(t)\|^2_{\Monemu}+\frac{2a_\f^2}{\kappa}\E\|u(t)\|^2_H+\frac{1}{2}\Tr(QAQ^*)\\
&\le -\frac{1}{2}\kappa\alpha_1\E\|A^{1/2}u(t)\|^2_{H}-\frac{1}{2}(1-\kappa)\delta\E\|\eta(t)\|^2_{\Monemu}+\frac{2a_\f^2}{\kappa}\E\|u(t)\|^2_H+\frac{1}{2}\Tr(QAQ^*).
\end{align*}
By Gronwall's inequality,
\begin{align*}
\E\Psi_1(U(t))\le e^{-c t} \Psi_1(U_0)+ \frac{2a_\f^2}{\kappa}\int_0^t e^{-c(t-r)}\E\|u(r)\|^2_H\d r +C.
\end{align*}
In light of Lemma \ref{lem:moment-bound:H^0_epsilon} together with Sobolev embedding again
\begin{align*}
\E\|u(r)\|^2_H\le 2\E\Psi_0(U(r))\le  2e^{-c_{0,1}r}\Psi_0(U_0)+2C_{0,1}\le 2\alpha_1^{-1} e^{-c_{0,1}r}\Psi_1(U_0)+2C_{0,1}.
\end{align*}
So that
\begin{align*}
\int_0^t e^{-c(t-r)}\E\|u(r)\|^2_H\d r&\le 2\alpha_1^{-1}e^{-ct}\int_0^t e^{-(c_{0,1}-c)r}\d r\Psi_1(U_0)+C.
\end{align*}
By Lemma \ref{lem:tech:1} with $c'=c_{0,1}$, it follows that
\begin{equation} \label{ineq:e^-ct.int.e^(cr)}
e^{-ct}\int_0^t e^{-(c_{0,1}-c)r}\d r\le Ce^{-ct}.
\end{equation}
Hence, we may infer the existence of $c_{1,1}$ and $C_{1,1}$ such that
\begin{align*}
\E\Psi_1(\Phi(t))\le C_{1,1}e^{-c_{1,1} t}\Psi_1(U_0)+C_{1,1},
\end{align*}
which proves \eqref{ineq:Psi_1^n} for $n=1$.

Now assuming \eqref{ineq:d.Psi_1^n}--\eqref{ineq:Psi_1^n} hold up to $n-1$, consider the case $n\ge 2$. Similarly to the proof of Lemma \ref{lem:moment-bound:H^0_epsilon}, we first compute partial derivatives of $\Psi_1(u,\eta)^n$ along a direction $\xi\in \H^1$ as follows:
\begin{align*}
\la D_u\Psi_1(u,\eta)^n,\pi_1\xi\ra_{H^1}&=n\Psi_1(u,\eta)^{n-1}\la u,\pi_1\xi\ra_{H^1},\\
\la D_\eta \Psi_1(u,\eta)^n,\pi_2\xi\ra_{\Monemu}&=n\Psi_1(u,\eta)^{n-1}\la\eta,\pi_2\xi\ra_{\Monemu}\\
D_{uu}\Psi_1(u,\eta)^n(\xi)&=n\Psi_1(u,\eta)^{n-1}\pi_1\xi+n(n-1)\Psi_1(u,\eta)^{n-2}\la u,\pi_1\xi\ra_{H^1} u.
\end{align*}
So that applying $\L^\mu$, (see \eqref{form:L^epsilon}), to $\Psi_1(u,\eta)^n$ yields the identity
\begin{align*}
\L^\mu\Psi_1(u,\eta)^n
&=n\Psi_1(u,\eta)^{n-1}\Big( -\kappa\|Au\|^2_{H}+(1-\kappa)\la \Tcal_\mu\eta,\eta\ra_{\Monemu}\\
&\quad\ +\la \f'(u)\grad u,\grad u\ra_{H}+\frac{1}{2}\Tr(QAQ^*)\Big)\\
&\quad\ +\frac{1}{2} n(n-1)\Psi_1(u,\eta)^{n-2}\sum_{k\ge 1}|\la u,Qe_k\ra_{H^1}|^2.
\end{align*}
As in the case $n=1$, cf.~\eqref{ineq:L^epsilon.Psi_1}, we readily have the bound
\begin{align*}
&-\kappa\|Au\|^2_H+\la \Tcal_\mu\eta,\eta\ra_{\Monemu}+\la \f'(u)\grad u,\grad u\ra_{H}\\
&\le -\frac{1}{2}\kappa\|Au\|^2_{H}-\frac{1}{2}(1-\kappa)\delta\|\eta\|^2_{\Monemu}+\frac{2a_\f^2}{\kappa}\|u\|^2_H.
\end{align*}
Also,
\begin{align*}
\frac{1}{2}\sum_{k\ge 1}|\la u,Qe_k\ra_{H^1}|^2\le \frac{1}{2}\Tr(QAQ^*)\|A^{1/2}u\|^2_H\le \Tr(QAQ^*)\Psi_1(u,\eta).
\end{align*}
It follows that
\begin{align*}
&\L\Psi_1(u,\eta)^n\\
&\le n\Psi_1(u,\eta)^{n-1}\Big( -\frac{1}{2}\kappa\|Au\|^2_{H}-\frac{1}{2}(1-\kappa)\delta\|\eta\|^2_{\Monemu}+\frac{2a_\f^2}{\kappa}\|u\|^2_H+\frac{1}{2}\Tr(QAQ^*)\Big)\\
&\quad\ + n(n-1)\Tr(QAQ^*)\Psi_1(u,\eta)^{n-1}\\
&\le -c\Psi_{1}(u,\eta)^n+\frac{2na_\f^2}{\kappa}\Psi_{1}(u,\eta)^{n-1}\|u\|^2_H+n(n-\frac{1}{2} )\Tr(QAQ^*)\Psi_{1}(u,\eta)^{n-1}.
\end{align*}
By Young's inequality, the last term on the above right--hand side can be subsumed into $-c\Psi_1(u,\eta)^n$. Likewise
\begin{align*}
\frac{2na_\f^2}{\kappa}\Psi_1(u,\eta)^{n-1}\|u\|^2_H\le
{\frac{c}{100}}\Psi_1(u,\eta)^n+C\|u\|^{2n}_H\le
{\frac{c}{100}}\Psi_1(u,\eta)^n+C
\Psi_0(u,\eta)^n.
\end{align*}
Together with Sobolev embedding, we may infer the existence of positive constants $c$ and $C$ such that
\begin{align}\label{ineq:L^epsilon.Psi_1^n}
\L^\mu\Psi_1(u,\eta)^n&\le -c\Psi_1(u,\eta)^n+C\Psi_0(u,\eta)^n+C.
\end{align}
We now employ It\^o's formula making use of~\eqref{ineq:L^epsilon.Psi_1^n} to see that
\begin{align*}
\d \Psi_1(U(t))
&=\L^\mu\Psi_1(U(t))^n\d t+n\Psi_1(U(t))^{n-1}\la A^{1/2}u(t),A^{1/2}Q\d w(t)\ra_H\\
&\le -c\Psi_1(U(t))^n\d t+C\Psi_0(U(t))^n\d t+C\d t\\
&\quad\ +n\Psi_1(U(t))^{n-1}\la A^{1/2}u(t),A^{1/2}Q\d w(t)\ra_H,
\end{align*}
which establishes~\eqref{ineq:d.Psi_1^n}. In addition, from~\eqref{ineq:L^epsilon.Psi_1^n}, we have
\begin{align*}
\frac{\d}{\d t}\E\Psi_1(U(t))^n&\le -c\E\Psi_1(U(t))^n+C\E\Psi_0(U(t))^n+C,
\end{align*}
whence
\begin{align*}
\E\Psi_1(U(t))^n &\le e^{-ct}\Psi_1(U_0) +C\int_0^t e^{-c(t-r)}\E\Psi_0(U(r))^n\d r+C.
\end{align*}
We now invoke Lemma \ref{lem:moment-bound:H^0_epsilon} to see that
\begin{align*}
\int_0^t e^{-c(t-r)}\E\Psi_0(U(r))^n\d r\le \int_0^t e^{-c(t-r)}e^{-c_{0,n}r}\d r \Psi_0(U_0)^n+C.
\end{align*}
Since $\Psi_0(U_0)$ is dominated by $\Psi_1(U_0)$, reasoning as in the base case $n=1$ above, we also obtain
\begin{align*}
\E\Psi_1(U(t))^n\le C_{1,n}e^{-c_{1,n} t}\Psi_1(U_0)^n+C_{1,n},
\end{align*}
for suitable constants $c_{1,n}$ and $C_{1,n}$ that are independent of $\mu\in\Mdelta$, $(U_0)$ and $t$. This proves \eqref{ineq:Psi_1^n} for all $n\ge 2$, {as desired}.
\end{proof}

We now restrict to dimension $d\le 3$ and establish moment bounds in higher regularity of the solution under the extra assumptions~\nameref{cond:phi:d=3} and \nameref{cond:Q:higher_regularity}. In particular, we will employ the result in~Lemma \ref{lem:moment-bound:H^m:d=3} below to prove Theorem \ref{thm:regularity} in Section \ref{sec:regularity}.

\begin{lemma}\label{lem:moment-bound:H^m:d=3}
Let $d\le 3$ and assume the hypotheses of Theorem \ref{thm:well-posed}. Suppose further that \nameref{cond:phi:d=3}, \nameref{cond:Q:higher_regularity} hold for $m\ge 2$. Then for all $k=2,\dots,m$ and $n\ge 1$, there exist constants $c_{k,n}$, $C_{k,n}$ and integers $q_{k,n}$, independent of $U_0$, $t$ and $\mu\in\Mdelta$, such that
\begin{equation}  \label{ineq:d.Psi_k^n}
\emph{d} \Psi_k(U(t))^n\le -c_{k,n} \Psi_k(U(t))^n\emph{d} t+C_{k,n}\emph{d} t+C_{k,n}\Psi_{k-1}(U(t))^{q_{k,n}}\emph{d} t+\emph{d} M_{k,n}(t),
\end{equation}
and
\begin{equation} \label{ineq:moment-bound:H^k}
\E\Psi_k(U(t))^n \le C_{k,n}e^{-c_{k,n}t}\Psi_k(U_0)^{q_{k,n}}+C_{k,n},
\end{equation}
where $\Psi_{k}$ {is given as in} \eqref{form:Psi_m} and $M_{k,n}(t)$ is the semi-martingale given by
\begin{equation} \label{form:M_(k,n)}
M_{k,n}(t)=\int_0^t n\Psi_k(U(r))^{n-1}\la A^{k/2}u(t),A^{k/2}Q\emph{d} w(r)\ra_H.
\end{equation}
Furthermore, it holds that
\begin{equation} \label{ineq:moment-bound:H^m:d=3}
\begin{aligned}
&\E \Psi_m(U(t))^n+c_{m,n}\int_0^t \E \Psi_m(U(r))^{n-1}\big(\|A^{\frac{m+1}{2}}u(r)\|^2_H+\|\eta(r)\|^2_{M^{m}_\mu}\big)\emph{d} r\\
& \le C_{m,n}\Psi_m(U_0)^{q_{m,n}}+C_{m,n}t,
\end{aligned}
\end{equation}
and that
\begin{equation}\label{ineq:exponential-bound:H^m:d=3}
\E \exp\Big\{\beta_{m,0}\Psi_m(U(t))^{\gamma_{m,0}}\Big\}\le C_{m,0}\exp\Big\{C_{m,0}e^{-c_{m,0}t} \big(1+\Psi_m(U_0)^{q_{m,0}}\big)\Big\},
\end{equation}
for some positive constants $\gamma_{m,0}\in (0,1)$, $\beta_{m,0}$, $q_{m,0}$, $c_{m,0}$, $C_{m,0}$
also independent of $U_0$, $t$ and $\mu\in\Mdelta$.
\end{lemma}

Ultimately, the proof of Lemma \ref{lem:moment-bound:H^m:d=3} will proceed by induction, though we treat the case $m=2$ as a special case. Let us state and prove this special case now.

\begin{lemma} \label{lem:moment-boud:H^2xM^2:d=3}
{Let} $d\le 3$ {and assume the hypothesis} of Theorem \ref{thm:well-posed}. {Additionally assume} that \nameref{cond:phi:d=3}, \nameref{cond:Q:higher_regularity} hold for $m\ge 2$. Given $U_0=(u_0,\eta_0)\in \Htwomu$, there exist positive constants $c_{3,n},\,C_{3,n}$ and positive integer $q_{3,n}$ independent of $U_0$, $t$ and $\mu\in\Mdelta$ such that
\begin{equation}  \label{ineq:d.Psi_2^n}
\emph{d} \Psi_2(U(t))^n\le -c_{2,n} \Psi_2(U(t))^n\emph{d} t+C_{2,n}\emph{d} t+C_{2,n}\Psi_{1}(U(t))^{q_{2,n}}\emph{d} t+\emph{d} M_{2,n}(t),
\end{equation}
where $\Psi_{\ell}$ {is given as in} \eqref{form:Psi_m} and $M_{2,n}(t)$ is the semi-martingale given by
\begin{equation} \label{form:M_(2,n)}
M_{2,n}(t)=\int_0^t n\Psi_2(U(r))^{n-1}\la Au(t),A Q\emph{d} w(r)\ra_H.
\end{equation}
Furthermore,
\begin{equation} \label{ineq:moment-bound:H^2:d=3}
\E\Psi_2(U(t))^n \le C_{2,n}e^{-c_{2,n}t}\Psi_2(U_0)^{q_{2,n}}+C_{2,n}.
\end{equation}
\end{lemma}

\begin{proof} Similarly to the proof of Lemma \ref{lem:moment-boud:H^1_epsilon}, we start with the case $n=1$. Applying $\L^\mu$, cf.~\eqref{form:L^epsilon}, to $\Psi_2$ as in \eqref{form:Psi_m} gives
\begin{align*}
\L^\mu \Psi_2(u,v)&= -\kappa\|A^{3/2}u\|^2_H+(1-\kappa)\la \Tcal_\mu\eta,\eta\ra_{\Mtwomu}\\
&\quad\ +\la \f'(u)\grad u, \grad Au\ra_H+\frac{1}{2}\Tr(QA^2Q^*).
\end{align*}
We note that $\Tr(QA^2Q^*)<\infty$, by virtue of \nameref{cond:Q:higher_regularity}. Also, by \eqref{ineq:<T.eta,eta>}, we readily have the useful estimate
\begin{align*}
\la \Tcal_\mu\eta,\eta\ra_{\Mtwomu}\le -\frac{\delta}{2}\|\eta\|^2_{\Mtwomu}.
\end{align*}
It remains to bound $\la \f'(u)\grad u,\grad Au\ra_H $. To this end, we recall from \nameref{cond:phi:d=3} that $\f'(u)\le c(1+|u|^{p_1})$ where $p_1<4$. We then estimate
\begin{align*}
\la \f'(u)\grad u,\grad Au\ra_H\le c(1+\|u\|^{p_1}_{\infty})\|A^{1/2}u\|_H \|A^{3/2}u\|_H.
\end{align*}
Furthermore,
\begin{align*}
\|A^{1/2}u\|_H \|A^{3/2}u\|_H\le {c}\|A^{1/2}u\|^2_H+\frac{1}{100}\kappa\|A^{3/2}u\|^2_H.
\end{align*}
Using Agmon's inequality and Sobolev interpolation in $d\le 3$, we have
\begin{align*}
\|u\|^{p_1}_{\infty}\|A^{1/2}u\|_H \|A^{3/2}u\|_H&\le c\|A^{1/2}u\|^{\frac{p_1}{2}}_H \|Au\|^{\frac{p_1}{2}}_H \|A^{1/2}u\|_H \|A^{3/2}u\|_H\\
&= c\|A^{1/2}u\|^{\frac{p_1+2}{2}}_H \|Au\|^{\frac{p_1}{2}}_H  \|A^{3/2}u\|_H\\
&\le c\|A^{1/2}u\|^{\frac{p_1+2}{2}}_H \|A^{1/2}u\|^{\frac{p_1}{4}}_H\|A^{3/2}u\|^{\frac{p_1}{4}}_H  \|A^{3/2}u\|_H\\
&=c\|A^{1/2}u\|^{\frac{3p_1+4}{4}}_H \|A^{3/2}u\|^{\frac{p_1+4}{4}}_H \\
&\le
{c}\|A^{1/2}u\|^{\frac{6p_1+8}{4-p_1}}_H+
{\frac{1}{100}}\kappa\|A^{3/2}u\|^{2}_H ,
\end{align*}
where in the last estimate above, we employed Young inequality with the fact that $p_1<4$. {We may thus} infer the existence of a {positive} integer $q$ such that
\begin{align*}
&\la \grad(\f(u)), \grad A u\ra_H \notag \\
&\le c\|A^{1/2}u\|^{2q}_H+\frac{1}{2}\|A^{3/2}u\|^2_H+c\notag \\
&\le c\Psi_1(u,\eta)^{q}+\frac{1}{2}\kappa\|A^{3/2}u\|^2_H+c .
\end{align*}
We now combine everything to arrive at the estimate
\begin{align}
&\L^\mu \Psi_2(u,\eta) \notag \\
&= -\kappa\|A^{3/2}u\|^2_H+(1-\kappa)\la \Tcal_\mu\eta,\eta\ra_{\Mtwomu}+\la A^{1/2}(\f(u)), A^{3/2}u\ra_H+\frac{1}{2}\Tr(QA^2Q^*) \nonumber\\
&\le -\frac{1}{2}\kappa\|A^{3/2}u\|^2_H-\frac{1}{2}(1-\kappa)\delta\|\eta\|^2_{\Mtwomu}+c\Psi_1(u,\eta)^{q}+\frac{1}{2}\Tr(QA^2Q^*)+c, \label{ineq:L^epsilon.Psi_2}
\end{align}
whence by It\^o's formula
\begin{align*}
\d \Psi_2(U(t))&= \L^\mu\Psi_2(U(t))\d t+\la Au(t),AQ\d w(t)\ra_H\\
&\le -c\Psi_2(U(t))\d t +C\Psi_1(U(t))^q\d t+C\d t+\la Au(t),AQ\d w(t)\ra_H.
\end{align*}
This establishes~\eqref{ineq:d.Psi_2^n} for $n=1$. Also, by Gronwall's inequality
\begin{align*}
\E\Psi_2(U(t))\le e^{-ct}\Psi_2(U_0)+C\int_0^te^{-c(t-r)}\E\Psi_1(U(r))^q\d r+C.
\end{align*}
In view of Lemma \ref{lem:moment-boud:H^1_epsilon},
\begin{align*}
\E\Psi_1(U(r))^q\le C_{1,q}e^{-c_{1,q}r}\Psi_1(U_0)^q+C_{1,q}\le Ce^{-cr}\Psi_2(U_0)^q+C.
\end{align*}
Reasoning as in \eqref{ineq:e^-ct.int.e^(cr)}, we see that
\begin{align*}
\int_0^te^{-c(t-r)}\E\Psi_1(U(r))^q\d r\le Ce^{-ct}\Psi_2(U_0)^q+C.
\end{align*}
It follows that we may infer the existence of $c_{2,1}$ and $C_{2,1}$ such that
\begin{align*}
\E\Psi_2(U(r))\le C_{2,1}e^{-c_{2,1}t}\Psi_2(U_0)^{q_{2,1}}+C_{2,1}.
\end{align*}

We now consider the general case $n\ge 2$. Similarly to the proof of Lemma \ref{lem:moment-boud:H^1_epsilon}, the partial derivatives of $\Psi_2(u,\eta)^n$ along a direction $\xi\in \Htwomu$ are given by:
\begin{align*}
\la D_u\Psi_2(u,\eta)^n,\pi_1\xi\ra_{H^2}&=n\Psi_2(u,\eta)^{n-1}\la u,\pi_1\xi\ra_{H^2},\\
\la D_\eta \Psi_2(u,\eta)^n,\pi_2\xi\ra_{\Mtwomu}&=n\Psi_2(u,\eta)^{n-1}\la\eta,\pi_2\xi\ra_{\Mtwomu}\\
D_{uu}\Psi_2(u,\eta)^n(\xi)&=n\Psi_2(u,\eta)^{n-1}\pi_1\xi+n(n-1)\Psi_2(u,\eta)^{n-2}\la u,\pi_1\xi\ra_{H^2} u.
\end{align*}
So,
\begin{align*}
&\L^\mu\Psi_2(u,\eta)^n\\
&=n\Psi_2(u,\eta)^{n-1}\Big( -\kappa\|A^{3/2}u\|^2_{H}+(1-\kappa)\la \Tcal_\mu\eta,\eta\ra_{\Mtwomu}+\la A^{1/2}(\f(u)),A^{3/2}u\ra_{H}\\
&\quad\ +\frac{1}{2}\Tr(QA^2Q^*)\Big)\\
&\quad\ +\frac{1}{2} n(n-1)\Psi_2(u,\eta)^{n-2}\sum_{k\ge 1}|\la u,Qe_k\ra_{H^2}|^2\\&=I_1+I_2.
\end{align*}
In view of \eqref{ineq:L^epsilon.Psi_2}, we have
\begin{align*}
I_1&\le -n\Psi_2(u,\eta)^{n-1}\big(\frac{1}{2}\kappa\|A^{3/2}u\|^2_H+\frac{1}{2}(1-\kappa)\delta\|\eta\|^2_{\Mtwomu}\big)\\
&\quad\ +n\Psi_2(u,\eta)^{n-1}\big(c\Psi_1(u,\eta)^q+\frac{1}{2}\Tr(QA^2Q^*)+c\big).
\end{align*}
To further estimate the second term on the above right-hand side, we employ Young inequality to see that
\begin{align*}
&n\Psi_2(u,\eta)^{n-1}\big(c\Psi_1(u,\eta)^q+\frac{1}{2}\Tr(QA^2Q^*)+c\big)\\
&\le \frac{1}{100}n\min\{\kappa\alpha_1,(1-\kappa)\delta\}\Psi_2(u,\eta)^n+{c}
\big(\Psi_1(u,\eta)^{qn}+1\big).
\end{align*}
Likewise
\begin{align*}
I_2&=\frac{1}{2} n(n-1)\Psi_2(u,\eta)^{n-2}\sum_{k\ge 1}|\la u,Qe_k\ra_{H^2}|^2\\
&\le \frac{1}{2} n(n-1)\Tr(QA^2Q^*)\Psi_2(u,\eta)^{n-2}\|u\|^2_{H^2}\\
&\le n(n-1)\Tr(QA^2Q^*)\Psi_2(u,\eta)^{n-1}\\
&\le \frac{1}{100}n\min\{\kappa\alpha_1,(1-\kappa)\delta\}\Psi_2(u,\eta)^n+c.
\end{align*}
We now combine the estimates on $I_1$ and $I_2$ to infer the bound
\begin{align} \label{ineq:L^epsilon.Psi_2^n}
\L^\mu\Psi_2(u,\eta)^n \le -c\Psi_2(u,\eta)^n+C\Psi_1(u,\eta)^{qn}+C.
\end{align}
By It\^o's formula, we immediately obtain~\eqref{ineq:d.Psi_2^n}. Also,
\begin{align*}
\frac{\d}{\d t}\E\Psi_2(U(t))^n &\le -c\E\Psi_2(U(t))^n+C\E\Psi_1(U(t))^{qn}+C,
\end{align*}
whence
\begin{align} \label{ineq:moment-bound:H2:a}
\E\Psi_2(U(t))^n\le e^{-ct}\Psi_2(\x)^n+\int_0^t e^{-c(t-r)}\E\Psi_1(U(r))^{qn}\d r+C.
\end{align}
We invoke Lemma \ref{lem:moment-boud:H^1_epsilon} again to see that
\begin{align*}
\E\Psi_1(U(r))^{qn}\le C_{1,qn}e^{-c_{1,qn}r}\Psi_1(\x)^{qn}+C_{1,qn}\le Ce^{-ct}\Psi_2(\x)^{qn}+C.
\end{align*}
Similarly to \eqref{ineq:e^-ct.int.e^(cr)},
\begin{align*}
\int_0^te^{-c(t-r)}\E\Psi_1(U(r))^{qn}\d r\le Ce^{-ct}\Psi_2(\x)^{qn}+C,
\end{align*}
which together with \eqref{ineq:moment-bound:H2:a} proves \eqref{ineq:moment-bound:H^2:d=3}, thereby finishing the proof.
\end{proof}

Now let us prove Lemma \ref{lem:moment-bound:H^m:d=3}.

\begin{proof}[Proof of Lemma \ref{lem:moment-bound:H^m:d=3}]
We proceed by induction on $m$. {We note that the case $m=2$ was treated in} Lemma \ref{lem:moment-boud:H^2xM^2:d=3}. Due to the nonlinearity of $\f$, we will treat the case $m=3$ next.

Recall from \eqref{form:Psi_m} that for $(u,\eta)\in \Hthreemu$,  $\Psi_3(u,\eta)=\frac{1}{2}\|A^{3/2}u\|^2_H+\frac{1}{2}(1-\kappa)\|\eta\|^2_{M^3}$. {A calculation then yields}
\begin{align}\label{eq:L.Psi_3}
\L^\mu \Psi_3(u,\eta)&= -\kappa\|A^{2}u\|^2_H+(1-\kappa)\la \Tcal_\mu\eta,\eta\ra_{M^3}+\la A(\f(u)), A^{2}u\ra_H+\frac{1}{2}\Tr(QA^3Q^*).
\end{align}
In the above, $\Tr(QA^3Q^*)<\infty$ thanks to \nameref{cond:Q:higher_regularity}. Also, recalling \eqref{ineq:<T.eta,eta>},
\begin{align*}
\la \Tcal_\mu\eta,\eta\ra_{\Mthreemu}\le -\frac{\delta}{2}\|\eta\|^2_{\Mthreemu}.
\end{align*}
It remains to estimate the term involving $\f(u)$. To see this, by H\"older's  and Agmon's inequalities, we have
\begin{align*}
&\la A(\f(u)), A^{2}u\ra_H\\
&=\la \f'(u)Au,A^2u\ra_H-\la \f''(u)|\grad u|^2, A^{2}u\ra_H\\
&\le c\|u\|^{p_1}_{L^\infty}\|Au\|_H\|A^2u\|_H+c\|u\|^{p_2}_{L^\infty}\|\grad u\|_{L^\infty}\|A^{1/2}u\|_H\|A^2u\|_H\\
&\le c\|A^{1/2}u\|^{\frac{p_1}{2}}_H\|Au\|^{\frac{p_1}{2}}_H\|Au\|_H\|A^2u\|_H\\
&\quad +c\|A^{1/2}u\|^{\frac{p_2}{2}}_H\|Au\|^{\frac{p_2}{2}}_H\|Au\|^{\frac{1}{2}}_H\|A^{3/2}u\|^{\frac{1}{2}}_H\|A^{1/2}u\|_H\|A^2u\|_H\\
&\le c\|A^{1/2}u\|^{\frac{p_1}{2}}_H\|Au\|^{1+\frac{p_1}{2}}_H\|A^2u\|_H+c\|A^{1/2}u\|^{1+\frac{p_2}{2}}_H\|Au\|^{\frac{p_2+1}{2}}_H\|A^2u\|_H^{\frac{3}{2}}.
\end{align*}
Using Young inequality on the above right hand side, we infer the bound
\begin{align} \label{ineq:<A.phi,A^2.u>}
\la A(\f(u)), A^{2}u\ra_H\le c\big(\|A^{1/2}u\|^{n_1}_H+\|Au\|^{n_2}_H+1\big)+\frac{1}{2}\kappa\|A^2u\|_H^2.
\end{align}
As a consequence, we obtain the estimate
\begin{align*}
&-\kappa\|A^{2}u\|^2_H+(1-\kappa)\la \Tcal_\mu\eta,\eta\ra_{\Mthreemu}+\la A(\f(u)), A^{2}u\ra_H+\frac{1}{2}\Tr(QA^3Q^*)\\
&\le -\frac{1}{2}\kappa\|A^{2}u\|^2_H -\frac{1}{2}(1-\kappa)\delta\|\eta\|^2_{\Mthreemu}+c\big(\|A^{1/2}u\|^{n_1}_H+\|Au\|^{n_2}_H+1\big)\\
&\le -\frac{1}{2}\kappa\|A^{2}u\|^2_H -\frac{1}{2}(1-\kappa)\delta\|\eta\|^2_{\Mthreemu}+c\big(\Psi_1(u,\eta)^{n_1}+\Psi_2(u,\eta)^{n_2}+1\big)\\
&\le -\frac{1}{2}\kappa\|A^{2}u\|^2_H -\frac{1}{2}(1-\kappa)\delta\|\eta\|^2_{\Mthreemu}+c(\Psi_2(u,\eta)^q+1),
\end{align*}
which combines with \eqref{eq:L.Psi_3} yields
\begin{align*}
\d \Psi_3(U(t)) &=\L^\mu\Psi_3(U(t))\d t+\la A^{3/2}u(t),A^{3/2}Q\d w(t)\ra_H\\
&\le - c \Psi_3(U(t))\d t+C\Psi_2(U(t))^q\d t+C\d t++\la A^{3/2}u(t),A^{3/2}Q\d w(t)\ra_H.
\end{align*}
Also,
\begin{align*}
\E\Psi_3(U(t))\le e^{-ct}\Psi_3(\x)+C\int_0^te^{-c(t-r)}\E\Psi_2(U(r))^{q}\d r+C.
\end{align*}
Recalling the estimates in Lemma \ref{lem:moment-boud:H^2xM^2:d=3}, then applying Poincar\'e's inequality yields
\begin{align*}
\E\Psi_2(U(r))^{q}&\le Ce^{-cr}\Psi_2(\x)^{q'}+C\le Ce^{-cr}\Psi_3(\x)^{q'}+C.
\end{align*}
Similarly to \eqref{ineq:e^-ct.int.e^(cr)} making use of~Lemma \ref{lem:tech:1}, we therefore arrive at the bound
\begin{align*}
\E\Psi_3(U(t))\le C_{3,1}e^{-c_{3,1}t}\Psi_3(\x)^{q_{3,1}}+C_{3,1},
\end{align*}
where $c_{3,1}$, $C_{3,1}$ and $q_{3,1}$ do not depend on $\x,\,t$ and $\mu\in\Mdelta$.

In order to establish higher moment {bounds} for $\Psi_3$, we first {observe that we} have the following identity for $n\ge 2$
\begin{align*}
\L^\mu\Psi_3(u,\eta)^n&=n\Psi_3(u,\eta)^{n-1}\Big( -\kappa\|A^{2}u\|^2_{H}+(1-\kappa)\la \Tcal_\mu\eta,\eta\ra_{\Mthreemu}+\la A(\f(u)),A^{2}u\ra_{H}\\
&\quad\ +\frac{1}{2}\Tr(QA^3Q^*)\Big)\\
&\quad\ +\frac{1}{2} n(n-1)\Psi_3(u,\eta)^{n-2}\sum_{k\ge 1}|\la u,Qe_k\ra_{H^3}|^2.
\end{align*}
By employing a similar argument as in the proof of Lemma \ref{lem:moment-boud:H^2xM^2:d=3} and recalling \eqref{ineq:<A.phi,A^2.u>}, we see that the above right-hand side is dominated by
\begin{align*}
-c\Psi_3(u,\eta)^n+C\big(\Psi_2(u,\eta)^{\hat{q}}+1\big),
\end{align*}
where $\hat{q}$ is a constant integer that only depends on $n$ and $\f$. So that, for all $\x\in\Hthreemu$ and $t\ge 0$,
\begin{align}
\d \Psi_3(U(t))^n& \le -c \Psi_3(U(t))^n\d t+C(\Psi_2(U(t))^{ \hat{q} }+1)\d t\notag\\
&\quad\ +n\Psi_3(U(t))^{n-1}\la A^{3/2}u(t),A^{3/2}Q\d w(t)\ra_H.\label{ineq:d.Psi_3^n}
\end{align}
Furthermore, employing Lemma \ref{lem:moment-boud:H^2xM^2:d=3} and Lemma \ref{lem:tech:1}, the following holds
\begin{align*}
\E\Psi_3(U(t))^n&\le e^{-ct}\Psi_3(\x)^n+C\int_0^te^{-c(t-r)}\E\Psi_2(U(r))^{\hat{q}}\d r+C\\
&\le C_{3,1}e^{-c_{3,n}t}\Psi_3(\x)^{q_{3,n}}+C_{3,n}.
\end{align*}
This finishes the proofs of \eqref{ineq:d.Psi_k^n}--\eqref{ineq:moment-bound:H^k} for {$m=3$}.

{For the induction hypothesis, we} now suppose that \eqref{ineq:moment-bound:H^k} holds for $k=1,\dots,m-1\ge 3$. {Note that} $m\ge 4$. Letting $\Psi_m(u,\eta)$ be as in \eqref{form:Psi_m}, we formally have the identity
\begin{align}\label{eq:L.Psi_m}
\L^\mu \Psi_m(u,\eta)&= -\kappa\|A^{\frac{m+1}{2}}u\|^2_H+(1-\kappa)\la \Tcal_\mu\eta,\eta\ra_{\Mmmu}+\la A^{\frac{m-1}{2}}(\f(u)), A^{\frac{m+1}{2}}u\ra_H \notag \\
&\quad\ +\frac{1}{2}\Tr(QA^mQ^*).
\end{align}
We note that by the previous arguments, it is established that $u\in H^{m-1}\cap L^\infty(\domain)$ since $H^2\hookrightarrow L^\infty(\domain)$ in dimension $d\le 3$. Also, recalling \nameref{cond:Q:higher_regularity} and \eqref{ineq:<T.eta,eta>}, {we see that}
\begin{align*}
\Tr(QA^mQ^*)<\infty,\quad\text{and}\quad\la \Tcal_\mu\eta,\eta\ra_{\Mmmu}\le -\frac{\delta}{2}\|\eta\|^2_{\Mmmu}.
\end{align*}
Due to the difficulty from the nonlinearity, to estimate the right-hand side of \eqref{eq:L.Psi_m}, we will consider two cases depending on whether $m$ is even or not. {For what follows, we recall from \nameref{cond:phi:d=3} that $\f\in H^{m-1}$ (see \cite{conti2006singular}).}

\vspace*{4pt}\noindent
\textbf{Case 1.} $m$ is even. In this case, it holds that
\begin{align*}
\la A^{\frac{m-1}{2}}(\f(u)), A^{\frac{m+1}{2}}u\ra_H&=\la\grad  A^{\frac{m-2}{2}}(\f(u)), \grad A^{\frac{m}{2}}u\ra_H.
\end{align*}
In view of identity \eqref{form:A^(n)phi(u)}, we apply $\grad$ to $A^{\frac{m-2}{2}}(\f(u))$ to obtain
\begin{align*}
\grad  A^{\frac{m-2}{2}}(\f(u)) &= \f'(u)\grad A^{\frac{m-2}{2}}u+c\f''(u)\grad u\cdot \grad^2 A^{\frac{m-4}{2}}u+ \sum_{i\ge 2}^{m-1}\f^{(i)}(u)I_{i}(u),
\end{align*}
where for $i=2,\dots, m-1$, $I_i(u)$ satisfies
\begin{align*}
\|I_{i}(u)\|_{L^\infty}\le c\Big(\sum_{j=1}^{m-3}\|A^{\frac{j}{2}}u\|^{n_j}_{L^\infty}+1\Big).
\end{align*}
We then use Agmon's, H\"older's, {and Young's} inequalities to estimate (recalling \nameref{cond:phi:d=3})
\begin{align*}
&\la\f'(u)\grad A^{\frac{m-2}{2}}u+c\f''(u)\grad u\cdot \grad^2 A^{\frac{m-4}{2}}u, \grad A^{\frac{m}{2}}u\ra_H \\
 &\le c\Big[(\|u\|^{p_1}_{L^\infty}+1)\|A^{\frac{m-1}{2}}u\|_H+(\|u\|^{p_2}_{L^\infty}+1)\|A^{1/2}u\|_{L^\infty}\|A^{\frac{m-2}{2}}u\|_H    \Big]\|A^{\frac{m+1}{2}}u\|_H \\
 &\le
 {c}\Big[ 1+\|u\|^{4p_1}_{L^\infty}+\|u\|^{6p_2}_{L^\infty}+\|A^{1/2}u\|_{L^\infty}^{6}+  \|A^{\frac{m-2}{2}}u\|_H^{6}+ \|A^{\frac{m-1}{2}}u\|_H^{4}\Big]\notag
 \\
 &\quad\ +
 {\frac{1}{100}}\|A^{\frac{m+1}{2}}u\|^2_H\\
 &\le {c}
 \Big[ 1+\|A^{1/2}u\|^{2p_1}_H\|Au\|^{2p_1}_H+\|A^{1/2}u\|^{3p_2}_H\|Au\|^{3p_2}_H+\|Au\|_{H}^{3}\|A^{3/2}u\|_{H}^{3}\\
 &\quad\ +  \|A^{\frac{m-2}{2}}u\|_H^{6}+ \|A^{\frac{m-1}{2}}u\|_H^{4}\Big]\\
 &\quad\ +\frac{1}{100}\kappa\|A^{\frac{m+1}{2}}u\|^2_H.
\end{align*}
Likewise,
\begin{align*}
&\Big\la \sum_{i\ge 2}^{m-1}\f^{(i)}(u)I_{i}(u),  \grad A^{\frac{m}{2}}u\Big\ra_H\\
&\le c (\|u\|^{p_2}_{L^\infty}+1)\Big(\sum_{j=1}^{m-3}\|A^{j}u\|^{n_j}_{L^\infty}+1\Big)\|A^{\frac{m+1}{2}}u\|_H\\
&\le
{c}\Big(\|u\|^{4p_2}_{L^\infty}+\sum_{j=1}^{m-3}\|A^{\frac{j}{2}}u\|^{4n_j}_{L^\infty}+1\Big)+
{\frac{1}{100}}\|A^{\frac{m+1}{2}}u\|^2_H\\
&\le
{c}\Big(\|A^{1/2}u\|^{2p_2}_{H}\|Au\|^{2p_2}_{H}+\sum_{j=1}^{m-3}\|A^{\frac{j+1}{2}}u\|^{2n_j}_{H}\|A^{\frac{j+2}{2}}u\|^{2n_j}_{H}+1\Big)+\frac{1}{100}\kappa\|A^{\frac{m+1}{2}}u\|^2_H.
\end{align*}
We now combine the above estimates to infer the existence of a constant $n_*$ sufficiently large such that
\begin{align} \label{ineq:<A^(m-1)/2.phi(u),A^(m+1)/2.u>}
\la A^{\frac{m-1}{2}}(\f(u)), A^{\frac{m+1}{2}}u\ra_H
&\le {c}
\Big(1+\sum_{j=0}^{m-1}\|A^{\frac{j}{2}}u\|^{n_*}_H\Big)+\frac{1}{2}\kappa\|A^{\frac{m+1}{2}}u\|^2_H.
\end{align}

\vspace*{4pt}\noindent\textbf{Case 2.} $m$ is odd. In this situation, since $\frac{m-1}{2}$ is an integer, we may invoke \eqref{form:A^(n)phi(u)} directly to see that
\begin{align*}
A^{\frac{m-1}{2}}(\f(u))&= \f'(u)A^{\frac{m-1}{2}}u+c\f''(u)\grad u\cdot  \grad A^{\frac{m-3}{2}}u+ \sum_{i\ge 2}^{m-1}\f^{(i)}(u)\widehat{I}_{i}(u),
\end{align*}
where for $i=2,\dots, m-1$, $\widehat{I}_i(u)$ satisfies (for possibly different $n_j$'s from those of case 1)
\begin{align*}
\|\widehat{I}_{i}(u)\|_{L^\infty}\le c\Big(\sum_{j=1}^{m-3}\|A^{\frac{j}{2}}u\|^{n_j}_{L^\infty}+1\Big).
\end{align*}
Employing the same argument as in {Case} 1 above, we also establish the estimate \eqref{ineq:<A^(m-1)/2.phi(u),A^(m+1)/2.u>} for a suitably large $n^*$.

Combining the above estimates with \eqref{eq:L.Psi_m}, we observe that (using Sobolev embedding)
\begin{align}
\L^\mu \Psi_m(\x)&\le -\frac{1}{2}\kappa\|A^{\frac{m+1}{2}}u\|^2_H-\frac{1}{2}(1-\kappa)\delta\|\eta\|^2_{\Mmmu}+C\sum_{j=0}^{m-1}\|A^{\frac{j}{2}}u\|^{n_*}_H+C\nonumber\\
&\le -\frac{1}{2}\kappa\|A^{\frac{m+1}{2}}u\|^2_H-\frac{1}{2}(1-\kappa)\delta\|\eta\|^2_{\Mmmu}+C\|A^{\frac{m-1}{2}}u\|^{n_*}_H+C\nonumber\\
&\le -\frac{1}{2}\kappa\|A^{\frac{m+1}{2}}u\|^2_H-\frac{1}{2}(1-\kappa)\delta\|\eta\|^2_{\Mmmu}+C\Psi_{m-1}(u,\eta)^{n_*}+C. \label{ineq:L.Psi_m}
\end{align}
By It\^o's formula, this implies
\begin{align}
\d \Psi_m(U(t))&=\L^\mu\Psi_m(U(t))\d t+\la A^{\frac{m}{2}}u(t),A^{\frac{m}{2}}Q\d w(t)\ra_H \notag  \\
&\le -c\Psi_m(U(t))\d t+ C\Psi_{m-1}(U(t))^{n_*}\d t+\la A^{\frac{m}{2}}u(t),A^{\frac{m}{2}}Q\d w(t)\ra_H. \label{ineq:d.Psi_m}
\end{align}
Also, using Gronwall's inequality, we arrive at the bound
\begin{align*}
\E\Psi_m(U(t))&\le e^{-ct}\Psi_m(\x)+C\int_0^t e^{-c(t-r)}\E\Psi_{m-1}(U(r))^{n_*}\d r+C.
\end{align*}
In view of the induction hypothesis, namely, \eqref{ineq:moment-bound:H^k} holds for up to $m-1$, we readily have
\begin{align} \label{ineq:sum.Psi_j}
\E\Psi_{m-1}(U(r))^{n_*}& \le Ce^{-c_{m-1,n_*}r}\Psi_{m-1}(\x)^{q_{m-1,n_*}}+C \notag\\
&\le Ce^{-c_{m-1,n_*}r}\Psi_{m}(\x)^{q_{m-1,n_*}}+C,
\end{align}
where $C$ does not depend on time $r$, $\x$ and $\mu\in\Mdelta$. As a consequence
\begin{align*}
\E\Psi_m(U(t))&\le e^{-ct}\Psi(\x)+C\int_0^t e^{-c(t-r)}e^{-c_{m-1,n_*}r}\d r \Psi_m(\x)^{q_{m-1,n_*}}+C.
\end{align*}
In the same spirit of \eqref{ineq:e^-ct.int.e^(cr)}, we therefore may infer the existence of positive constants $c_{m,1},C_{m,1}$ and $q_{m,1}$ independent of $\x$, $\mu\in\Mdelta$ and $t$ such that
\begin{align*}
\E\Psi_m(U(t))\le C_{m,1}e^{-c_{m,1}t}\Psi_m(\x)^{q_{m,1}}+C_{m,1},
\end{align*}
thereby proving~\eqref{ineq:d.Psi_k^n}--\eqref{ineq:moment-bound:H^k} for $k=m$ and $n=1$. For general $n\ge 2$, a similar argument as in the case $k=3$, also produces the desired bounds \eqref{ineq:d.Psi_k^n}--\eqref{ineq:moment-bound:H^k}.

Regarding \eqref{ineq:moment-bound:H^m:d=3}, we integrate \eqref{ineq:L.Psi_m} with respect to time and take expectation to arrive at the following bounds for all $\x\in \H^m$ and $t\ge 0$,
\begin{align*}
&\E\Psi_m(U(t))+\int_0^t\frac{1}{2}\kappa \E\|A^{\frac{m+1}{2}}u(r)\|^2_H+\frac{1}{2}(1-\kappa)\delta\E\|\eta(r)\|^2_{\Mmmu}\d r\\
&\le \Psi_m(\x)+C\int_0^t\E\Psi_{m-1}(U(r))^{n_*}\d r + Ct.
\end{align*}
Together with estimate \eqref{ineq:sum.Psi_j}, the above inequality implies
\begin{align*}
&\E\Psi_m(U(t))+\int_0^t\frac{1}{2}\kappa \E\|A^{\frac{m+1}{2}}u(r)\|^2_H+\frac{1}{2}(1-\kappa)\delta\E\|\eta(r)\|^2_{\Mmmu}\d r\\
&\le \Psi_m(\x)+C\Big(\Psi_m(\x)^{q_{m-1,n_*}}+t\Big).
\end{align*}
This proves \eqref{ineq:moment-bound:H^m:d=3} for $n=1$. For general $n\ge 2$, we may employ the same strategy of proving \eqref{ineq:moment-bound:H^k} for $k=m$ to establish \eqref{ineq:moment-bound:H^m:d=3}.

Turning to~\eqref{ineq:exponential-bound:H^m:d=3}, for positive constants $b_i$, $n_i$, $i=0,\dots,m-1$, to be chosen later, let
\begin{align} \label{form:G}
G(u,\eta)=\Psi_m(u,\eta)+\sum_{i=0}^{m-1}b_i \Psi_i(u,\eta)^{n_i}.
\end{align}
By It\^o's formula, we have the following identity
\begin{align*}
\d G(U(t))= \L^\mu\Psi_m(U(t))\d t+ \d M_{m,1}(t)+\sum_{i=0}^{m-1}b_i\L^\mu\Psi_i(U(t))^{n_i}+b_i\d M_{i,n_i}(t)
\end{align*}
where the semi-martingales $M_{i,n_i}(t)$ are given by the formula~\eqref{form:M_(k,n)}. In light of~\eqref{ineq:d.Psi_k^n}, we obtain the estimate
\begin{align*}
&\d G(U(t))\\
&\le -c_{m,1}\Psi_{m}(U(t))\d t+C_{m,1}(\Psi_{m-1}(U(t))^{q_{m,1}}+1)\d t +\d M_{m,1}(t)\\
&\quad\ +\sum_{i=1}^{m-1}-b_i c_{i,n_i}\Psi_i(U(t))^{n_i}+b_iC_{i,n_i}(\Psi_{i-1}(U(t))^{q_{i,n_i}}  +1)\d t+b_i\d M_{i,n_i}(t)\\
&\quad\ -b_0 c_{0,n_0}\Psi_0(U(t))^{n_0}\d t+b_0C_{0,n_0}\d t +b_0\d M_{0,n_0}(t).
\end{align*}
In order to subsume the positive $\d t-$terms involving $U(t)$ on the above right--hand side, we simply choose
\begin{align*}
b_{m-1}=2\frac{C_{m,1}}{c_{m-1,n_{m-1}}},\quad n_{m-1}=q_{m,1},
\end{align*}
and
\begin{align*}
b_{i}=2\frac{b_{i+1}C_{i+1,n_{i+1}}}{c_{i,n_{i}}},\quad n_{i}=q_{i+1,n_{i+1}},\quad i=0,\dots,m-2.
\end{align*}
As a consequence, we may infer the existence of positive constants $c,C$ independent of $\x,\,t$ and $\mu\in\Mdelta$ such that
\begin{align} \label{ineq:d.G}
\d G(U(t))&\le -c\,G(U(t))\d t+C+\d M_G(t),
\end{align}
where
\begin{align*}
M_G(t)= M_{m,1}(t)+\sum_{i=0}^{m-1}b_i M_{i,n_i}(t).
\end{align*}
Now, for $\gamma\in(0,1)$ to be chosen later, we apply It\^o's formula to the function $(1+G)^\gamma$ and obtain
\begin{align}\label{eqn:d.(1+G)^gamma:Ito}
\d (1+G(U(t)))^{\gamma}
&= \gamma(1+G(U(t)))^{\gamma-1}\d G(U(t)) \notag \\
&\quad\ +\frac{1}{2}\gamma(\gamma-1)(1+G(U(t)))^{\gamma-2}\d\la M_G\ra(t),
\end{align}
where $\la M_G\ra(t)$ is the quadratic variation process of $M_G$ and given by
\begin{align} \label{form:<M_G>(t)}
\la M_G\ra(t) = \int_0^t \sum_{k\ge 1}\Big|\la u(r),Qe_k\ra_{H^m}+\sum_{i=0}^{m-1}b_in_i\Psi_i(U(r))^{n_i-1}\la u(r),Qe_k\ra_{H^i}\Big|^2\d r.
\end{align}
From~\eqref{ineq:d.G}, we readily have
\begin{align*}
&\gamma(1+G(U(t)))^{\gamma-1}\d G(U(t)) \\
&\le -c \gamma(1+G(U(t)))^{\gamma-1}G(U(t))\d t+C\gamma(1+G(U(t)))^{\gamma-1}\d t\\
&\quad\ +\gamma(1+G(U(t)))^{\gamma-1}\d M_G(t)\\
&\le -c \gamma(1+G(U(t)))^{\gamma}\d t+C\gamma(1+G(U(t)))^{\gamma-1}\d t\\
&\quad\ +\gamma(1+G(U(t)))^{\gamma-1}\d M_G(t).
\end{align*}
Since $\gamma\in(0,1)$ and $G\ge0$, it is clear that
\begin{align*}
C\gamma(1+G(U(t)))^{\gamma-1}\le C\gamma.
\end{align*}
Also, the last term on the right--hand side of~\eqref{eqn:d.(1+G)^gamma:Ito} is negative, thanks to $\gamma\in(0,1)$ again. We thus deduce the estimate
\begin{align}\label{ineq:d.(1+G)^gamma}
&\d (1+G(U(t)))^{\gamma} \notag \\
&\le  -c\gamma(1+G(U(t)))^{\gamma}\d t+C\gamma\d t+\gamma(1+G(U(t)))^{\gamma-1}\d M_G(t),
\end{align}
for some positive constants $c,\,C$ independent of $\gamma$. Now, to establish~\eqref{ineq:exponential-bound:H^m:d=3}, we aim to employ the well-known exponential Martingale inequality applied to~\eqref{ineq:d.(1+G)^gamma}. The argument below is similarly to those found in \cite[Page 23--25]{glatt2021long} and \cite[Lemma 5.1]{hairer2008spectral}.

Fixing $T>0$, applying I\^o's formula to $e^{\frac{1}{2}c\gamma(t-T)}(1+G(U(t)))^{\gamma}$ making use of~\eqref{ineq:d.(1+G)^gamma} gives
\begin{align}
&\d\Big[e^{\frac{1}{2} c\gamma(t-T)}(1+G(U(t)))^{\gamma}\Big] \notag\\
&=\frac{1}{2}c\gamma e^{\frac{1}{2} c\gamma(t-T)}(1+G(U(t)))^{\gamma}\d t +e^{\frac{1}{2} c\gamma(t-T)}\d (1+G(U(t)))^{\gamma}\notag\\
&\le -\frac{1}{2}c\gamma e^{\frac{1}{2} c\gamma(t-T)}(1+G(U(t)))^{\gamma}+e^{\frac{1}{2} c\gamma(t-T)}C\gamma\d t  \notag \\
&\quad\  +\gamma e^{\frac{1}{2} c\gamma(t-T)}(1+G(U(t)))^{\gamma-1}\d M_G(t). \label{ineq:d.e^t.(1+G)^gamma:a}
\end{align}
We note that~\eqref{form:<M_G>(t)} together with~\nameref{cond:Q:higher_regularity} implies
\begin{align*}
&\d \la M_G\ra(t)\\
&\le m\sum_{k\ge 1}|\la u(t),Qe_k\ra_{H^m}|^2\d t+m\sum_{i=0}^{m-1}b_i^2n_i^2\Psi_i(U(r))^{2n_i-2}\sum_{k\ge 1}|\la u(t),Qe_k\ra_{H^i}|^2\d t\\
&\le m\Tr(QA^mQ^*)\|u(t)\|^2_{H^m}\d t+m\sum_{i=0}^{m-1}b_i^2n_i^2\Psi_i(U(r))^{2n_i-2}\Tr(QA^iQ^*)\| u(t)\|^2_{H^i}\d t\\
&\le C\Big(\Psi_{m}(U(t)) +\sum_{i=0}^{m-1}\Psi_i(U(t))^{2n_i-1} \Big)\d t.
\end{align*}
So, recalling $G$ as in~\eqref{form:G}, for all $\gamma\in (0,1)$ sufficiently small, we have the bound
\begin{align*}
\frac{1}{\sqrt{\gamma}}(1+G(U(t)))^{2-\gamma}\d t&\ge \frac{C}{\sqrt{\gamma}}\Big(1+ \Psi_m(U(t))^{2-\gamma} +\sum_{i=0}^{m-1}\Psi_i(U(t))^{2n_i-\gamma n_i} \Big)\\
&\ge C\Big(\Psi_{m}(U(t)) +\sum_{i=0}^{m-1}\Psi_i(U(t))^{2n_i-1} \Big)\d t\\&\ge \d \la M_G\ra(t),
\end{align*}
whence for all $t\in [0,T]$
\begin{align}
&\frac{1}{2}c\gamma e^{\frac{1}{2} c\gamma(t-T)}(1+G(U(t)))^{\gamma}\d t \notag \\&=\frac{c}{2\sqrt{\gamma}}\gamma^2 e^{\frac{1}{2} c\gamma(t-T)}\cdot \frac{1}{\sqrt{\gamma}}(1+G(U(t)))^{\gamma}\d t\notag\\
&\ge \frac{c}{2\sqrt{\gamma}}\cdot \gamma^2 e^{ c\gamma(t-T)}(1+G(U(t)))^{2\gamma-2}\d \la M_G\ra(t).\label{ineq:d.e^t.(1+G)^gamma:b}
\end{align}
Combining~\eqref{ineq:d.e^t.(1+G)^gamma:a}--\eqref{ineq:d.e^t.(1+G)^gamma:b}, we obtain for $t\in [0,T]$ \newcommand{\Mtilde}{\tilde{M}}
\begin{equation} \label{ineq:d.e^t.(1+G)^gamma:c}
\begin{aligned}
\d \Big[e^{\frac{1}{2} c\gamma(t-T)}(1+G(U(t)))^{\gamma}\Big]
&\le C\gamma e^{\frac{1}{2}c\gamma(t-T)}\d t+\d \Mtilde_G(t)-\frac{c}{2\sqrt{\gamma}}\d \la \Mtilde_G\ra(t),
\end{aligned}
\end{equation}
where $\Mtilde_G$ is the semi-martingale defined as
\begin{align*}
\Mtilde_G=\gamma e^{\frac{1}{2} c\gamma(t-T)}(1+G(U(t)))^{\gamma-1}\d M_G(t),
\end{align*}
whose variation process $\la \Mtilde_G\ra(t)$ is given by
\begin{align*}
 \d\la \Mtilde_G\ra(t)=\gamma^2 e^{ c\gamma(t-T)}(1+G(U(t)))^{2\gamma-2}\d \la M_G\ra(t).
\end{align*}
By the exponential Martingale inequality applying to $\Mtilde(t)$,
\begin{align}\label{ineq:exponential-Martingale}
\P\Big(\sup_{t\ge 0}\Big[\Mtilde_G(t)-\frac{c}{2\sqrt{\gamma}}\la\Mtilde_G\ra(t)\Big] >r\Big)\le e^{-\frac{c}{\sqrt{\gamma}}r}, \quad r\ge 0,
\end{align}
we integrate~\eqref{ineq:d.e^t.(1+G)^gamma:c} on $t\in[0,T]$ to see that
\begin{align}\label{ineq:d.e^t.(1+G)^gamma:d}
&\P\Big( (1+G(U(T)))^{\gamma}-e^{-\frac{1}{2} c\gamma T}(1+G(\x))^{\gamma}-C\gamma\int_0^T e^{\frac{1}{2} c\gamma(\ell-T)}\d \ell \ge r    \Big) \notag \\
&\le e^{-\frac{c}{\sqrt{\gamma}}r}.
\end{align}
In particular, by choosing $\gamma$ small enough such that $c/\sqrt{\gamma}>1$, ~\eqref{ineq:d.e^t.(1+G)^gamma:d} implies
\begin{align*}
\E\exp\Big\{  (1+G(U(T)))^{\gamma}  \Big\} \le C\exp\Big\{ e^{-\frac{1}{2} c\gamma T}(1+G(\x))^{\gamma}\Big\}.
\end{align*}
In the above, we emphasize that $c,\,C$ may depend on $\gamma$ but do not depend on $\x$, $t$ and $\mu\in\Mdelta$. Also,
\begin{align*}
c\Psi_m(\x)\le G(\x)\le C(1+\Psi_m(\x)^{n}),
\end{align*}
for some positive integer $n$. As a consequence, the following holds for all $T\ge 0$,
{\begin{align*}
\E\exp\Big\{ c \Psi_m(U(T))^{\gamma}  \Big\} &\le C\exp\Big\{ e^{-\frac{1}{2} c\gamma T}(1+\Psi_m(\x)^{ n})^\gamma\Big\}\\
&\le C\exp\Big\{C e^{-\frac{1}{2} c\gamma T}(1+\Psi_m(\x)^{ n\gamma})\Big\}.
\end{align*}
We therefore establish~\eqref{ineq:exponential-bound:H^m:d=3} with $\gamma_{m,0}:=\gamma\in(0,1)$ and $q_{m,0}:=n\gamma$}. The proof is thus complete.
\end{proof}

\section{Existence of invariant probability measures}

\label{sec:existence}

In this section, we proceed to show that the Markov semigroup $P_t^{\mu}$ as in \eqref{form:P_t^epsilon} possesses an invariant probability $\numu$. To do so, we first observe that Theorem \ref{thm:well-posed} implies that $P_t^{\mu}$ is Feller, that is, $P_t^{\mu}f\in {C(\Hzeromu)}$ for every $f\in {C(\Hzeromu)}$. We then follow the classical Krylov-Bogoliubov argument \cite{da2014stochastic} to show that the following family of measures $\{\numu_T\}_{T>0}$ given by
\begin{equation} \label{form:nu.time-average}
\numu_T(\cdot)=\frac{1}{T}\int_0^T\close P_t^{\mu}(0,\cdot)\d t,
\end{equation}
is tight in $\Pcal r(\Hzeromu)$. Recall that $P_t^{\mu}(0,\cdot)$ denotes the Markov semigroup associated with the process $U^0(t)=(u(t;0),\eta(t,\cdotp;0))$ that satisfies \eqref{eqn:react-diff:mu} corresponding to the initial condition $\x=0\in\Hzeromu$. As mentioned earlier in Section \ref{sect:notation}, the embedding $\Honemu\subset\Hzeromu$ is only continuous whereas the embedding $\Z_\mu^1=H^1\times \Ecal_\mu^1\subset\Hzeromu$ is compact \cite{gatti2004exponential,joseph1989heat,joseph1990heat}. It is thus important to obtain a moment bound in norm of $\Ecal^1_\mu$ that grows at most linearly in time $t$. Here we recall from \eqref{form:space:L:norm} that
$$ \|\eta\|_{\Ecal_\mu^{1}}^2 = \|\eta\|_{\Monemu}^2+\|\Tcal_\mu\eta\|^2_{\Mzeromu}+\sup_{r\geq 1}r\T^\mu_\eta(r),$$
where $\Tcal_\mu$ and $\T^\mu$ are as in \eqref{form:Tcal} and \eqref{form:tailfunction}, respectively.

{In order to prove Theorem \ref{thm:existence}, we will establish two preliminary lemmas. In the first, we will obtain estimates for $u(t)$ in {the} $H^1$--norm. More precisely, we have the following result.}

\begin{lemma} \label{lem:int_0^t e^(delta/epsilon)|A^(1/2)u|ds}
Under the hypotheses of Theorem \ref{thm:well-posed}, let $U(t)=(u(t),\eta(t))$ denote the solution of \eqref{eqn:react-diff:mu} corresponding to initial condition $U_0\in\Hzeromu$. Then for all $0<q\le\delta$, where $\delta$ is the constant from \eqref{ineq:<T.eta,eta>}, there exists a positive constant $c=c(U_0,q)$ independent of $t$ and $\mu\in\Mdelta$ such that
\begin{equation}\label{ineq:int_0^t e^(delta/epsilon)|A^(1/2)u|ds}
\E\int_0^t e^{qr}\|A^{1/2}u(r)\|^2_H\emph{d} r\leq\|\x\|^2_{\Hzeromu}+ c\,e^{qt}.
\end{equation}

\end{lemma}

\begin{proof}
We apply Ito's formula to $e^{qt}\Psi_0(U(t))$ to see that
\begin{align*}
\d \big(e^{qt}\Psi_0(U(t))\big)&=e^{qt}\Big[q\Psi_0(U(t))-\kappa\|A^{1/2}u(t)\|^2_H+(1-\kappa)\la \Tcal_\mu\eta(t),\eta(t)\ra_{\Mzeromu}\\
&\quad\ +\la\f(u(t)),u(t)\ra_H+\frac{1}{2}\Tr(QQ^*)\Big]\d t +e^{qt}\la u(t),Q\d w(t)\ra_H.
\end{align*}
Thanks to \eqref{ineq:<T.eta,eta>} and the assumption $q\le\delta$, we have
$$\frac{1}{2}(1-\kappa)q\|\eta\|^2_{\Mzeromu}+ (1-\kappa)\la \Tcal_\mu\eta,\eta\ra_{\Mzeromu}\le 0.$$
Also, we invoke \nameref{cond:phi:2} and Young's inequality to see that
\begin{align*}
\frac{q}{2}\|u\|^2_H+\la\f(u),u\ra_H
&\le \frac{q}{2 }\|u\|^2_H-a_2\|u\|^{p_0+1}_{L^{p_0+1}(\domain)}+a_3|\domain|\\
&\le a_2\big(\frac{q}{2a_2 }\big)^{\frac{p_0+1}{p_0-1}}|\domain|+a_3|\domain|.
\end{align*}
We then infer the existence of a positive constant $c=c(q,\f,\domain,Q)$, independent of $t$ and $\mu\in\Mdelta$, such that
\begin{equation*}
e^{qt}\E\Psi_0(U(t))+\E\int_0^t e^{qr}\|A^{1/2}u(r)\|_H^2\d r\le \|\x\|^2_{\Hzeromu}+c\int_0^t e^{qr}\d r\le\|\x\|^2_{\Hzeromu}+ c\,e^{qt}.
\end{equation*}
This implies \eqref{ineq:int_0^t e^(delta/epsilon)|A^(1/2)u|ds}, as claimed.
\end{proof}

{In the second lemma, we assert that the solution of \eqref{eqn:react-diff:mu} indeed satisfies a stronger bound in $\Tcal_\mu$ and $\T^\mu$ uniform in time $t$. This will rely on Lemma \ref{lem:int_0^t e^(delta/epsilon)|A^(1/2)u|ds}. {For this, we will make use of the following elementary fact for locally integrable functions $f:\rbb\to[0,\infty)$ which are non-negative: for all $0<s\le t$ and $\delta>0$:}
\begin{align}\label{eq:tech:2}
e^{-\frac{\delta}{2}s} \int_0^s f(t-r)\d r&=\int_0^s e^{-\frac{\delta}{2}r} f(t-r)\d r-\frac{\delta}{2}\int_0^s e^{-\frac{\delta}{2}r} \int_0^r f(t-r')\d r'\d r \notag \\
&\le \int_0^t e^{-\frac{\delta}{2}r} f(t-{r})\d r.
\end{align}
}
\begin{lemma} \label{lem:T:bound}
Under the same hypothesis of~Theorem \ref{thm:well-posed}, let $U^0(t)=(u(t;0),\break\eta(t,\cdotp;0))$ be the solution of \eqref{eqn:react-diff:mu} with zero initial condition in $\Hzeromu$. Then there exists a constant $c$ independent of $t$ such that the following bounds hold for all $t\geq0$:
$$\E\| \Tcal_\mu\eta(t)\|^2_{\Mzeromu}\le c\qquad\text{and}\qquad \E\sup_{r\ge 1}r\T^\mu_{\eta(t)}(r)\le c.$$
\end{lemma}

\begin{proof}[Proof of Lemma \ref{lem:T:bound}]
Given the initial condition $\eta_0=0$, we recast \eqref{form:eta(t)-representation} as
\begin{align}
\eta(t,s)&=\chi_{(0,t]}(s)\int_0^s u(t-r)\d r+
\chi_{(t,\infty)}(s)\int_0^t u(t-r)\d r= \int_0^{s\wedge t}u(t-r)\d r.\label{form:eta(t)-representation:eta_0=0}
\end{align}
This implies that $\Tcal_\mu\eta(t)$ can be written explicitly as
$$\Tcal_\mu\eta(t,s)=-u(t-s)\chi_{(0,t]}(s).$$
It follows that
$$\E\| \Tcal_\mu\eta(t)\|^2_{\Mzeromu}=\E\int_0^t\mu(s)\|A^{1/2}u(t-s)\|^2_H\d s=\E\int_0^t\mu(t-s)\|A^{1/2}u(s)\|^2_H\d s,$$
where we have simply made a change of variable in the last equality. By \nameref{cond:mu}
$$\mu(s)\le \mu(0)e^{-\delta s}.$$
Thus
\begin{align*}
\E\| \Tcal_\mu\eta(t)\|^2_{\Mzeromu}\le\mu(0)e^{-\delta t}\E\int_0^te^{\delta s}\|A^{1/2}u(s)\|^2_H\d s.
\end{align*}
In light of Lemma \ref{lem:int_0^t e^(delta/epsilon)|A^(1/2)u|ds} (with $q=\delta$), the claimed uniform-in-time on $\E\| \Tcal_\mu\eta(t)\|^2_{\Mzeromu}$ follows.

With regards to $\sup_{r\ge 0}r\T^\mu_{\eta(t)}(r)$, we first note that \eqref{form:eta(t)-representation:eta_0=0} {and the Cauchy-Schwarz inequality together yield}
\begin{align*}
\|A^{1/2}\eta(t,s)\|^2_H&= \int_0^{s\mi t}\close\int_0^{s\mi t}\close \la A^{1/2}u(t-r_1),A^{1/2}u(t-r_2)\ra_H \d r_1\d r_2\\
&\le (s\mi t)\int_0^{s\mi t}\close\|A^{1/2}u(t-r')\|^2_H\d r'\\
&\le  s\int_0^{s\mi t}\close\|A^{1/2}u(t-r')\|^2_H\d r'.
\end{align*}
Now for $s\in[0,t]$, let
    \[
        I(s):=e^{-\frac{\delta}{2}s}\int_0^{s}\|A^{1/2}u(t-r')\|_H\d r'.
    \]
Recalling the expression of $\T^\mu_{\eta(t)}$ in \eqref{form:tailfunction} and the fact that $\mu(s)\le \mu(0)e^{-\delta s}$, we may estimate for all $r\ge 1$ as follows:
\begin{align*}
r\T^\mu_{\eta(t)}(r)&= r\int_{(0,\frac{1}{r})\cup(r,\infty)}\close\close\mu(s)\|A^{1/2}\eta(t,s)\|^2_H\d s\\
&\le r\mu(0)\int_{(0,\frac{1}{r})\cup(r,\infty)}\close \close e^{-\delta s}s\int_0^{s\mi t}\close\|A^{1/2}u(t-r;)\|^2_H\d r'\,\d s\\
&=r\mu(0)\int_{(0,\frac{1}{r})\cup(r,\infty)}\close\close s\,e^{-\frac{\delta}{2}s} e^{-\frac{\delta}{2}s} \int_0^{s\mi t}\close\|A^{1/2}u(t-r')\|^2_H\d r'\,\d s\\
&\le r\mu(0)\int_{(0,\frac{1}{r})\cup(r,\infty)}\close\close s\,e^{-\frac{\delta}{2}s}  e^{-\frac{\delta}{2}(s\mi t)} \int_0^{s\mi t}\close\|A^{1/2}u(t-r')\|^2_H\d r'\,\d s\\
&=r\mu(0)\int_{(0,\frac{1}{r})\cup(r,\infty)}\close\close s\,e^{-\frac{\delta}{2}s}{I(s\mi t)}\d s.
\end{align*}
By \eqref{eq:tech:2}, we have
$$I(s\mi t)\le \int_0^t e^{-\frac{\delta}{2}r'} \|A^{1/2}u(t-r')\|^2_H\d r'.$$
It follows that
$$ r\T_{\eta(t)}^\mu(r)\le r\mu(0)\int_{(0,\frac{1}{r})\cup(r,\infty)}\close\close s\,e^{-\frac{\delta}{2}s} \d s \int_0^t e^{-\frac{\delta}{2} r'}\|A^{1/2}u(t-r')\|^2_H\d r'.$$
For $r\ge 1$, we have
\begin{align*}
r\int_{(0,\frac{1}{r})\cup(r,\infty)}\close\close s\,e^{-\frac{\delta}{2}s} \d s<c(\delta)<\infty.
\end{align*}
Note that $c(\delta)$ is independent of $r$. Thus, taking supremum on $r\ge1$ yields
\begin{align*}
\sup_{r\ge 1} r\T_{\eta(t)}^\mu(r)&\le \sup_{r\ge 1} r \mu(0)\int_{(0,\frac{1}{r})\cup(r,\infty)}\close\close s\,e^{-\frac{\delta}{2}s}\d s\int_0^t e^{-\frac{\delta}{2}r'}\|A^{1/2}u(t-r')\|^2_H\d r'\\&\le c(\mu) \int_0^t e^{-\frac{\delta}{2}r'}\|A^{1/2}u(t-r')\|^2_H\d r'\\
&=c(\mu)e^{-\frac{\delta}{2}t} \int_0^t e^{\frac{\delta}{2}r'}\|A^{1/2}u(r')\|^2_H\d r'.
\end{align*}
Taking expectation on both sides and employing Lemma \ref{lem:int_0^t e^(delta/epsilon)|A^(1/2)u|ds} with $q=\delta/2$, we obtain the estimate
\begin{align*}
\E\sup_{r\ge 1}r\T^\mu _{\eta(t)}(r)&\le  c(
\mu),
\end{align*}
where $c(\mu)$ is independent of $t$, as desired.
\end{proof}

{We are now ready to} give the proof of Theorem \ref{thm:existence}, thus ensuring the existence of an invariant probability measure $\nu$ for \eqref{eqn:react-diff:mu}.
\begin{proof}[Proof of Theorem \ref{thm:existence}]
Recalling the space $\Z^1_\mu=H^1\times \Ecal^1_\mu$ as in \eqref{form:space:H_epsilon}, since $\Ecal^1_\mu$ is compactly embedded into $\Mzeromu$  \cite{gatti2004exponential,joseph1989heat,joseph1990heat}, it is clear that the following bounded set
$$\mathcal{B}_R=\{(u,\eta)\in \Z^1_\mu :\|(u,\eta)\|_{\Z^1_\mu }\le R\},$$
is precompact in $\Hzeromu=H\times \Mzeromu$. Using Markov's inequality, we estimate
\begin{align}\label{eq:time:avg}
\numu_t(\mathcal{B}_R^c)=\frac{1}{t}\int_0^t P^\mu_s(0;\mathcal{B}_R^c)\d s\le \frac{1}{t}\int_0^t \frac{\E\|A^{1/2}u(s)\|^2_H+\E\|\eta(s)\|^2_{\Ecal^1 }}{R^2}\d s.
\end{align}
To bound the term $\E\|A^{1/2}u(s)\|^2_H$, we note that $\|A^{1/2}u(s)\|^2_H\le 2\Psi_1(u(s),\eta(s))$ where $\Psi_1$ is given by~\eqref{form:Psi_1}. We then may invoke \eqref{ineq:Psi_1^n} with zero initial condition and $n=1$ to deduce
\begin{equation}\label{ineq:int_0^t|A^(1/2)u^epsilon|ds<t}
\int_0^t\E\|A^{1/2}u(r)\|^2_H\d r\le C(1+t).
\end{equation}
With regards to the second term, $\E\|\eta(s)\|^2_{\Ecal^1_\mu}$, appearing in \eqref{eq:time:avg}, thanks to Lemma \ref{lem:moment-boud:H^1_epsilon} and Lemma \ref{lem:T:bound}, we infer the existence of a positive constant $c$, independent of $t$, $R$, and such that
    $$
        \int_0^t \E\|\eta(s)\|^2_{\Ecal^1_\mu }\d s\le ct.
    $$
It follows that there exists $t_R>1$, sufficiently large, such that
$$\nu_t(\mathcal{B}_R)\le \frac{c}{R^2},$$
for all $t\geq t_R$. This, in turn, implies that $\{\numu_t\}_{t>0}$ is tight in $\Hzeromu$. An application of Krylov-Bogoliubov's Theorem then implies the existence of $\numu\in \Pcal r(\Hzeromu)$ which is invariant for \eqref{eqn:react-diff:mu}.

Regarding \eqref{ineq:exponential-bound:nu^epsilon}, given $N>0$ we consider $\phi_N(U)=\exp\left({\beta\|U\|^2_{\Hzeromu}}\right)\mi N$. Since $\phi_N$ is bounded, we see by invariance of $\nu^\mu$, that
\begin{align*}
\int_{\Hzeromu}\phi_N(\x)\numu(\d\x)=\int_{\Hzeromu}P_t^{\mu}\phi_N(\x)\numu(\d\x).
\end{align*}
Now, given $\epsilon>0$, we choose $R=R(\epsilon)$ sufficiently large such that $\numu(\mathcal{D}_R^c)<\epsilon$, where
$$\mathcal{D}_R=\{U\in \Hzeromu :\|U\|_{\Hzeromu }\le R\}.$$
It follows that
\begin{align*}
\int_{\Hzeromu}P_t^{\mu}\phi_N(\x)\numu(\d\x)&=\int_{\mathcal{D}_R}P_t^{\mu}\phi_N(\x)\numu(\d\x)+\int_{\mathcal{D}_R^c}P_t^{\mu}\phi_N(\x)\numu(\d\x)\\
&\le \int_{\mathcal{D}_R}P_t^{\mu}\phi_N(\x)\numu(\d\x)+N\epsilon.
\end{align*}
We now fix the choice $\epsilon=1/N$. Next, we recall $\Psi_0(U)$ given in~\eqref{form:Psi_0} and see that
\begin{align*}
\frac{1}{2}(1-\kappa)\|U\|^2_{\Hzeromu}\le \Psi_0(U)\le \frac{1}{2}\|U\|^{2}_{\Hzeromu}.
\end{align*}
Upon invoking \eqref{ineq:exponential-bound:H^0_epsilon} in Lemma \ref{lem:moment-bound:H^0_epsilon}, we have
\begin{align*}
P_t^{\mu}\phi_N(\x)&=\E\Big[\exp\left({\beta\|U(t)\|^2_{\Hzeromu}}\right)\mi N\Big] \notag\\
&\le \E\Big[ \exp\left({\frac{2\beta}{1-\kappa}\Psi_0(U(t))}\right)  \Big]\notag\\
&\le e^{-ct}\exp\left({\frac{2\beta}{1-\kappa}\Psi_0(\x)}\right)+C,
\end{align*}
for all $t\geq0$,  for all $\beta$ sufficiently small by virtue of~\eqref{ineq:exponential-bound:H^0_epsilon} in~Lemma \ref{lem:moment-bound:H^0_epsilon}. Hence
\begin{align*}
\int_{\mathcal{D}_R}P_t^{\mu}\phi_N(\x)\numu(\d\x)&\le \int_{\mathcal{D}_R}\left(e^{-ct}\exp\left({\frac{2\beta}{1-\kappa} \Psi_0(\x)}\right)+C\right)\nu^\mu(\d\x)\notag\\
&\le \int_{\mathcal{D}_R}\left( e^{-ct}\exp\left({\frac{\beta}{1-\kappa} \|\x\|^2_{\Hzeromu}}\right)+C\right)\numu(\d\x)\\
&\le e^{-ct}\exp\left({\frac{\beta}{1-\kappa}R^2}\right)+C.
\end{align*}
In the last implication above, we emphasize that $c,\,C$ and $\beta$ are independent of $R$ and $\mu\in\Mdelta$. Altogether, we arrive at the bound
\begin{align*}
\int_{\Hzeromu}P_t^{\mu}\phi_N(\x)\numu(\d\x)\le e^{-ct}\exp\left({\frac{\beta}{1-\kappa}R^2}\right)+1+ C.
\end{align*}
Consequently
\begin{align*}
 \int_{\Hzeromu}\big(\|\x\|^n_{\Hzeromu}\mi N\big)\numu(\d\x)=\int_{\Hzeromu}\phi_N(\x) \numu(\d\x)\le e^{-ct}\exp\left({\frac{\beta}{1-\kappa}R^2}\right)+1+C.
\end{align*}
Thus, upon taking $ct\geq\beta(1-\kappa)^{-1}R^2$, we obtain
\begin{align*}
\int_{\Hzeromu}\big(e^{\beta\|\x\|^2_{\Hzeromu}}\mi N\big)\numu(\d\x)\le 2+C.
\end{align*}
Observe that this bound in uniform in $N$. Thus, \eqref{ineq:exponential-bound:nu^epsilon} follows by virtue of the Monotone Convergence Theorem, as desired.
\end{proof}

\section{Regularity of $\numu$ in dimension $d\le 3$} \label{sec:regularity}

{We recall that Theorem \ref{thm:well-posed} guarantees that the Markovian dynamics, as represented by the family of operators $P^\mu_t$, can be viewed as a mapping from $\Pcal r(\Hzeromu)\rightarrow \Pcal r(\Hzeromu)$ via push-forward. In this section, we show that when restricted to the subspace, $\mathcal{I}$, of invariant probability measures, one in fact has $(P^\mu_t)^*:\mathcal{I}\rightarrow \Pcal r(\Hmmu)$, where $m>0$, whenever the potential field and random heat sources are sufficiently smooth. From this point of view, a type of ``smoothing" occurs asymptotically since a gain in derivatives  from $m=0$ to $m>0$ is achieved. To establish this, it suffices to show that the support of every invariant probability measure $\numu$ is contained in $H^{m+1}\times M^m_\mu$. In proving the main result of this section, Theorem \ref{thm:regularity}, we will in fact prove more, and establish exponential moment bounds with respect to the topology of $\Hmmu$. The proof of this fact} will rely on the series of bootstrap arguments on moment bounds that we established in Section \ref{sec:apriori-moment-estimate}. We note that since the evolution of $\eta$ has a hyperbolic structure, it does not naturally possess a mechanism for smoothing in the manner described above. {Generally speaking, hyperbolic equations are only expected to propagate the initial regularity, not gain regularity. This suggests that the mechanism for the type of asymptotic smoothing described above is nuanced in our context. Indeed, this apparent structural obstruction is overcome by the fact that the temperature field, $u$, is coupled inhomogeneously to the history variable, $\eta$. Therefore, as long as a smoothing mechanism is present for the temperature field, the hyperbolic structure of the memory variable will propagate its regularity. Lastly, the desired smoothing mechanism for the temperature field is provided by its diffusive properties through the operator $A$.}
We {show how to successfully exploit this insight} by employing a control argument that ``asymptotically guides'' solutions into a subspace of phase space of higher regularity.

First, let us introduce the following controlled system:
\begin{equation}\label{eqn:react-diff:mu:uhat}
\begin{aligned}
\d\,\uhat(t)&=- A\uhat(t)\d t-\int_0^\infty\close\mu(s)A\etahat(t;s)\d t+\f(\uhat(t))\d t+Q\d w(t)\\
&\quad -\kappa\alpha_{\nhat}P_{\nhat}\big(\uhat(t)-u(t)\big)\d t,\\
\frac{\d}{\d t}\,\etahat(t)&=\Tcal_\mu\etahat(t)+\uhat(t),\\
(\uhat(0),\etahat(0))&=0\in\Hzeromu,
\end{aligned}
\end{equation}
where $P_{\nhat}$ is {defined in} \eqref{form:P_Nu} and $\nhat$ is chosen sufficiently large such that $\alpha_{\nhat}$ satisfies
\begin{equation} \label{cond:alpha_n.hat>a_phi}
\kappa\alpha_{\nhat}>a_\f =\sup_{x\in\rbb}\f'(x).
\end{equation}
We note that this choice of $\alpha_{\nhat}$ is possible since the {$\lim_{n\to\infty}\alpha_n=\infty$}. Observe that \eqref{eqn:react-diff:mu:uhat} only differs from \eqref{eqn:react-diff:mu} by the appearance of the term $-\kappa\alpha_{\nhat}P_{\nhat}\big(\uhat(t)-u(t)\big)$, which is a control that serves to drive the signal $\uhat$ towards $u$ on the subspace spanned by the first $n$ eigenfunctions of $A$; this particular feature is key so that it ensures the control is smooth in space.

\begin{lemma}\label{lem:moment-bound:uhat}
{ Let} $d\le 3$ { and assume the conditions} of Theorem \ref{thm:well-posed}. {Additionally assume that} \nameref{cond:phi:d=3}, \nameref{cond:Q:higher_regularity} hold for $m\ge 2$. Then

\begin{enumerate}[noitemsep,topsep=0pt,wide=0pt,label=\arabic*.,ref=\theassumption.\arabic*]
\item There exist $c,C>0$ such that for all $\x\in\Hzeromu$ and $t\ge 0$, it holds that
\begin{equation} \label{ineq:moment-bound:|u-uhat|_H<e^-ct}
\E\big\|\big(u(t)-\uhat(t),\eta(t)-\etahat(t)\big)\big\|^2_{\Hzeromu}\le Ce^{-ct}\|\x\|^2_{\Hzeromu};
\end{equation}

\item There exist positive constants $C$ and $q\in\nbb$ independent of $\x$ and $t$ such that
 \begin{equation} \label{ineq:moment-bound:|uhat|_(H^m)<e^(-ct)+|x|_H^q}
\E \|(\uhat(t),\etahat(t))\|^2_{\Hmmu}\le C(\|\x\|^q_{\Hzeromu}+1).
 \end{equation}
 \end{enumerate}
\end{lemma}

Let us first prove Theorem \ref{thm:regularity} assuming that Lemma \ref{lem:moment-bound:uhat} holds.
\begin{proof}[Proof of Theorem \ref{thm:regularity}]
For $R,N\ge 1$, set
\begin{align*}
\psi_{R,N}(U):=R\mi\big( \|P_Nu\|_{H^m}+\|P_N\eta\|_{\Mmmu}\big),
\end{align*}
where $P_N$ is as in \eqref{form:P_Nu}. By invariance, it holds that
\begin{align} \label{ineq:nu-invariance}
\int_{\Hzeromu}\psi_{R,N}(\x)\numu(\d\x)=\int_{\Hzeromu}P_t^{\mu}\psi_{R,N}(\x)\numu(\d\x).
\end{align}
Recalling the ``shifted" system \eqref{eqn:react-diff:mu:uhat}, for every $\x\in\Hzeromu$, observe that
\begin{align*}
P_t^{\mu}\psi_{R,N}(\x)
&=\E\big[R\mi \big( \|P_Nu(t)\|_{H^m}+\|P_N\eta(t)\|_{\Mmmu}\big)\big]\\
&\le \E\|P_N(u(t)-\uhat(t))\|_{H^m}+\E\|P_N(\eta(t)-\etahat(t))\|_{\Mmmu}\\
&\quad+\E\|P_N\uhat(t)\|_{H^m}+\E\|P_N\etahat(t)\|_{\Mmmu}\\
&\le  \alpha_N^{m/2}\E\big\|\big(u(t)-\uhat(t),\eta(t)-\etahat(t)\big)\big\|_{\Hzeromu}+\E\big\|\big(\uhat(t),\etahat(t)\big)\big\|_{\Hmmu}.
\end{align*}
In light of Lemma \ref{lem:moment-bound:uhat}, we infer the existence of positive constants $c,C$ and $q$ independent of $\x,\,t,\,R$ and $N$ such that
\begin{align*}
P_t^{\mu}\psi_{R,N}(\x) \le C\alpha_N^{m/2}e^{-ct}\|\x\|_{\Hzeromu}+C\|\x\|^q_{\Hzeromu}+C.
\end{align*}
Hence
\begin{align*}
\int_{\Hzeromu} P_t^{\mu}\psi_{R,N}(\x)\numu(\d\x)&\le  C\alpha_N^{m/2}e^{-ct}\int_{\Hzeromu}\|\x\|_{\Hzeromu}\numu(\d\x)\\
&\quad\ +C\int_{\Hzeromu}\|\x\|^q_{\Hzeromu}\numu(\d\x)+C.
\end{align*}
By Theorem \ref{thm:existence} (see \eqref{ineq:exponential-bound:nu^epsilon}), we readily deduce
\begin{align*}
\int_{\Hzeromu}\|\x\|_{\Hzeromu}\numu(\d\x)+\int_{\Hzeromu}\|\x\|^q_{\Hzeromu}\numu(\d\x)<\infty.
\end{align*}
In particular, it follows that
\begin{align*}
\int_{\Hzeromu} P_t^{\mu}\psi_{R,N}(\x)\numu(\d\x)&\le C\alpha_N^{m/2}e^{-ct}+C.
\end{align*}
Given $N,R$, we may now choose $t$ sufficiently large, so that
\begin{align*}
\int_{\Hzeromu} P_t^{\mu}\psi_{R,N}(\x)\numu(\d\x)\le C,
\end{align*}
holds. Together with \eqref{ineq:nu-invariance}, this implies
\begin{align*}
\int_{\Hzeromu} R\mi\big( \|P_Nu\|_{H^m}+\|P_N\eta\|_{\Mmmu}{\big)}\numu(\d u,\d\eta)<C,
\end{align*}
An application of the Monotone Convergence Theorem then yields
\begin{align*}
\int_{\Hzeromu}\|\x\|_{\Hmmu}\numu(\d\x)<\infty,
\end{align*}
that is, the support of $\numu$ is contained in $\Hmmu$.

With regard to the exponential moment bound~\eqref{ineq:exponential-bound:nu(H^m)}, we employ an argument similar to the proof of~\eqref{ineq:exponential-bound:nu^epsilon} by making use of~\eqref{ineq:exponential-bound:H^m:d=3} to establish~\eqref{ineq:exponential-bound:nu(H^m)}; see the proof of Theorem \ref{thm:existence} in Section \ref{sec:existence}.

Lastly, concerning the moment bound \eqref{ineq:moment-bound:nu(H^m)}, we first note that~\eqref{ineq:exponential-bound:nu(H^m)} implies the following estimate for all $p> 0$
\begin{align*}
\sup_{\mu\in\Mdelta}\int_{\Hmmu} \|\x\|^p_{\Hmmu}\numu(\d\x)<\infty.
\end{align*}
Now for $R>1$, from estimate \eqref{ineq:moment-bound:H^m:d=3}, observe that for all $t\ge 0$
\begin{align*}
\int_0^t \int_{\Hmmu}\E\big[R\mi \|u(r)\|^n_{H^m}\|u(r)\|^2_{H^{m+1}}\big]\numu(\d \x)\d r & \le C\int_{\Hmmu}\|\x\|^{q}\numu(\d\x)+Ct\\
&\le C+Ct.
\end{align*}
We invoke invariance again
\begin{align*}
\int_{\Hmmu}\E\big[R\mi \|u(r)\|^n_{H^m}\|u(r)\|^2_{H^{m+1}}\big]\numu(\d \x)= \int_{\Hmmu}\big(R\mi \|u\|^n_{H^m}\|u\|^2_{H^{m+1}}\big)\numu(\d u,\d\eta).
\end{align*}
Hence, for all $t>0$, it holds that
\begin{align*}
\int_{\Hmmu}\big[R\mi \|u\|^n_{H^m}\|u\|^2_{H^{m+1}}\big]\numu(\d u,\d\eta)\le \frac{C}{t}+C.
\end{align*}
Upon passing to the limit $t\rightarrow\infty$, we arrive at
\begin{align*}
\int_{\Hmmu}\big[R\mi \|u\|^n_{H^m}\|u\|^2_{H^{m+1}}\big]\numu(\d u,\d\eta)\le C.
\end{align*}
Note that this bound is uniform in $R\ge 1$ and $\mu\in\Mdelta$. The Monotone Convergence Theorem then implies
\begin{align*}
\int_{\Hmmu} \|u\|^n_{H^m}\|u\|^2_{H^{m+1}}\numu(\d u,\d\eta)\le C,
\end{align*}
which produces the desired bound \eqref{ineq:moment-bound:nu(H^m)}.
\end{proof}

{Lastly,} {we now turn our attention to the process $(\uhat(t),\etahat(t))$ defined in \eqref{eqn:react-diff:mu:uhat} and prove Lemma \ref{lem:moment-bound:uhat}; this will finally complete the proof of Theorem \ref{thm:regularity}}.

\begin{proof}[Proof of Lemma \ref{lem:moment-bound:uhat}]
{We must establish \eqref{ineq:moment-bound:|u-uhat|_H<e^-ct} and \eqref{ineq:moment-bound:|uhat|_(H^m)<e^(-ct)+|x|_H^q}.} To establish \eqref{ineq:moment-bound:|u-uhat|_H<e^-ct}, we first set $z=\uhat-u$ and $\zeta=\etahat-\eta$ and observe that
\begin{equation}\label{eqn:react-diff:mu:uhat-u}
\begin{aligned}
\frac{\d}{\d t}z(t)&=- Az(t)-\int_0^\infty\close\mu(s)A\zeta(t;s)\d s+\f(\uhat(t))-\f(u(t))-\kappa\alpha_{\nhat}P_{\nhat}z(t),\\
\frac{\d}{\d t}\zeta(t)&=\Tcal_\mu\zeta(t)+z(t),\\
(z(0),\zeta(0))&=-\x\in\Hzeromu.
\end{aligned}
\end{equation}
Recalling $\Psi_0(u,\eta)=\frac{1}{2}\|u\|^2_H+\frac{1}{2}(1-\kappa)\|\eta\|^2_{\Mzeromu}$ as in \eqref{form:Psi_0}, a routine calculation gives
\begin{align*}
\frac{\d}{\d t}\Psi_0(z(t),\zeta(t))&= -\kappa\|A^{1/2}z(t)\|^2_H+(1-\kappa)\la \Tcal_\mu\zeta(t),\zeta(t)\ra_{\Mzeromu}\\
&\quad\ +\la \f(\uhat(t))-\f(u(t)),z(t)\ra_H-\kappa\alpha_{\nhat}\|P_{\nhat}z(t)\|^2_H.
\end{align*}
We invoke \eqref{ineq:<T.eta,eta>} to obtain
\begin{align*}
\la \Tcal_\mu\zeta(t),\zeta(t)\ra_{\Mzeromu}\le -\frac{1}{2}\delta\|\zeta(t)\|^2_{\Mzeromu}.
\end{align*}
To estimate the nonlinear term, in light of \nameref{cond:phi:3}
\begin{align*}
\la \f(\uhat(t))-\f(u(t)),z(t)\ra_H\le a_\f\|z(t)\|^2_H.
\end{align*}
Recalling the choice $\alpha_{\nhat}$ as in \eqref{cond:alpha_n.hat>a_phi}, we then may bound the nonlinear term as follows:
\begin{align*}
&-\kappa\|A^{1/2}z(t)\|^2_H+\la \f(\uhat(t))-\f(u(t)),z(t)\ra_H-\kappa\alpha_{\nhat}\|P_{\nhat}z(t)\|^2_H\\
&\le -\kappa\|P_{\nhat}A^{1/2}z(t)\|_H^2-\kappa\|(I-P_{\nhat})A^{1/2}z(t)\|^2_H-(\kappa\alpha_{\nhat}-a_\f)\|z(t)\|^2_H\\
&\quad\ +\kappa\alpha_{\nhat}\|(I-P_{\nhat})z(t)\|_H^2\\
&\le -\kappa\|P_{\nhat}A^{1/2}z(t)\|_H^2-(\kappa\alpha_{\nhat}-a_\f)\|z(t)\|^2_H\\
&\le -(\kappa\alpha_{\nhat}-a_\f)\|z(t)\|^2_H,
\end{align*}
where we invoked the Poincar\'e inequality (see \eqref{eqn:Ae_k=-alpha.e_k}) in obtaining the penultimate inequality. It follows that
\begin{align*}
\frac{\d}{\d  t}\Psi_0(z(t),\zeta(t))&\le  -(\kappa\alpha_{\nhat}-a_\f)\|z(t)\|^2_H-\frac{1}{2}(1-\kappa)\delta \|\zeta(t)\|^2_{\Mzeromu}\\
&\le -\min\{2(\kappa\alpha_{\nhat}-a_\f),\delta \}\Psi_0(z(t),\zeta(t)),
\end{align*}
whence
\begin{align*}
\Psi_0(z(t),\zeta(t)) &\le e^{-\min\{2(\kappa\alpha_{\nhat}-a_\f),\delta\}t}\Psi_0(z(0),\zeta(0)).
\end{align*}
As a consequence,
\begin{align*}
\|{(z(t),\zeta(t))}\|^2_{\Hzeromu} &\le\frac{1}{1-\kappa} e^{-\min\{2(\alpha_{\nhat}-a_\f),\delta\}t}\|{(z(0),\zeta(0))}\|^2_{\Hzeromu}\\
&=\frac{1}{1-\kappa} e^{-\min\{2(\alpha_{\nhat}-a_\f),\delta\}t}\|\x\|^2_{\Hzeromu}.
\end{align*}
This produces estimate \eqref{ineq:moment-bound:|u-uhat|_H<e^-ct}.

{Next, we prove \eqref{ineq:moment-bound:|uhat|_(H^m)<e^(-ct)+|x|_H^q}. Let us denote} by $\Lhat$ the generator associated with \eqref{eqn:react-diff:mu:uhat}. {Note} that for any $k=0,\dots,m$, {we have}
\begin{align*}
\Lhat{\Psi_k}(\uhat,\etahat)&=-\L\Psi_k(\uhat,\etahat)-\alpha_{\nhat}\la P_{\nhat}z,\uhat\ra_{H^k},
\end{align*}
which only differs from $\L\Psi_k(\uhat,\etahat)$, (see \eqref{form:L^epsilon}), by the appearance of the term\break $-\alpha_{\nhat}\la P_{\nhat}z,\uhat\ra_{H^k}$. Furthermore, using Young's inequality, it holds that
\begin{align*}
-\alpha_{\nhat}\la P_{\nhat}z,\uhat\ra_{H^k}&=-\alpha_{\nhat}\|P_{\nhat}\uhat\|^2_{H^k}+\alpha_{\nhat}\la P_{\nhat}\uhat,P_{\nhat}u\ra_{H^k}\\&\le \frac{1}{4}\alpha_{\nhat}\|P_{\nhat}u\|^2_{H^k}\\
&\le \frac{1}{4}\alpha_{\nhat}^{k+1}\|u\|^2_{H}.
\end{align*}
{In particular}
\begin{align*}
\Lhat{\Psi_k}(\uhat,\etahat)\le \L\Psi_k(\uhat,\etahat)+\frac{1}{4}\alpha_{\nhat}^{k+1}\|u\|^2_{H}.
\end{align*}
Upon recalling that $(\uhat(0),\etahat(0))=0$, we see that we may therefore employ the same strategy used in establishing \eqref{ineq:moment-bound:H^k} in order to obtain a similar bound for $(\uhat(t),\etahat(t))$, namely
\begin{align*}
\E\|{(\uhat(t),\etahat(t))}\|^2_{\H^m}\le C\int_0^te^{-c(t-r)}\|u(r)\|^q_H\d r+C.
\end{align*}
In the above, we emphasize that $C,c$ and $q$ are positive constants independent of $t$ and $\x$, {but dependent on $\nhat$}. Invoking Lemma \ref{lem:moment-bound:H^0_epsilon} with the fact that $\|u\|^q_H$ is dominated by $\Psi_0(u,\eta)^q$, we obtain
\begin{align*}
\E\|{(\uhat(t),\etahat(t))}\|^2_{\H^m}\le C(\Psi_0(\x)^{q}+1).
\end{align*}
This implies \eqref{ineq:moment-bound:|uhat|_(H^m)<e^(-ct)+|x|_H^q}, thus completing the proof.
\end{proof}

\section*{Acknowledgments}
The authors would like to thank anonymous referees for their providing a thorough review of this work. We appreciate their careful reading and insightful comments, which have improved the manuscript. The work of N. E. Glatt-Holtz was partially supported under the
National Science Foundation grants DMS-1313272, DMS-1816551, DMS-2108790,
and under a Simons Foundation travel support award 515990. The work of V. R. Martinez was partially supported by NSF-DMS 2213363 and NSF-DMS 2206491, PSC-CUNY Award 65187-00 53, which is jointly funded by The Professional Staff Congress and The City University of New York, and the Dolciani Halloran Foundation.

\appendix

\section{Well-posedness} \label{sec:well-posed}
This Appendix is dedicated to addressing the well-posedness of \eqref{eqn:react-diff:mu} via the classical Galerkin approximation. The approach that we employ follows closely the method found in \cite[Section 4]{giorgi1999uniform}, \cite{glatt2009strong}, and~\cite[Chapter 7]{robinson2001infinite} tailored to our settings. Due to the stochastic forcing term, we however will not truncate~\eqref{eqn:react-diff:mu} directly. We consider instead the following ``shifted" system
    \begin{align}
        \bdy_tv&=-\kappa Av-\kappa A\xi-(1-\kappa)\int_0^{\infty}\close \mu(s)A\eta(t,s)\d s+\f(v+\xi),\notag\\
        \bdy_t\eta&=\Tcal_\mu\eta+v+\xi, \label{eqn:shift}\\
        \d \xi&=Q\d w(t),\notag
    \end{align}
    together with the initial condition $v(0)=u_0,\, \eta(0)=\eta_0$ and $\xi(0)=0$. The solution to the original system~\eqref{eqn:react-diff:mu} is then recovered upon setting $u=v+\xi$.

\subsection{A priori estimate}
Recall the basis $\{e_k\}_{k\ge 1}$ as in~\eqref{eqn:Ae_k=-alpha.e_k}, we first look for a pair $(U_n,\eta_N)$ given by
\begin{align*}
V_N(t;x)= \sum_{k=1}^N v_k(t)e_k(x),\quad \eta_N(t;s,x)=\sum_{k=1}^N \eta_k(t;s)e_k(x),
\end{align*}
solving the following finite--dimensional system
\begin{equation} \label{eqn:shift:Galerkin}
\begin{aligned}
\bdy_t V_N(t)&=-\kappa AV_N(t)-\kappa AP_N\xi(t)-\int_0^\infty\close \mu(s)A\eta_N(t;s)\d s+P_N\f(V_N(t)+\xi(t)),\\
\bdy_t \eta_N(t)&= \Tcal_\mu\eta_N(t)+V_N(t)+P_N\xi(t),\qquad V_N(0)=P_N u_0,\,\eta_N(0)=P_N\eta_0.
\end{aligned}
\end{equation}
In what follows, we assert that the above Galerkin system is well-posed and that the solutions are uniformly bounded in $N$.

\begin{lemma} \label{lem:Galerkin}
For all $\x=(u_0,\eta_0)\in\Hzeromu$ and $t\ge 0$, there exists a unique strong solution pair $(V_N,\eta_N)\in C([0,t];H\times\Mzeromu)$. Furthermore, there exists a positive constant $c(\x,t)$ independent of $N$ such that the following holds
\begin{align*}
&\E\sup_{0\le r\le t}\|(V_N(r),\eta_N(r))\|^2_{\Hzeromu}\\
&+\int_0^t\E\big(\|V_N(r)\|^2_{H^1}+\|\eta_N(r)\|^2_{\Mzeromu} +\|V_N(r)\|^{p+1}_{L^{p+1}}+\|\f(V_N(r)+\xi(r))\|^q_{L^{q}}\big)\emph{d} r\\
&\le c(\x,t),
\end{align*}
where $q=(p_0+1)/p_0$ and $p_0$ is as in \nameref{cond:phi:2}.
\end{lemma}

\begin{proof}
From \eqref{eqn:shift:Galerkin}, a routine calculation gives
\begin{align*}
\frac{1}{2}\bdy_t \|V_N(t)\|^2_H&=-\kappa \|A^{1/2}V_N(t)\|^2_H-\kappa\la A^{1/2}V_N(t),A^{1/2}\xi(t)\ra_H\\
&\quad\ - (1-\kappa)\int_0^\infty\close\mu(s)\la A^{1/2}\eta_N(t;s),A^{1/2}V_N(t)\ra_H\d s\\
&\quad\ +\la P_N\f(V_N(t)+\xi(t)),V_N(t)\ra_H.
\end{align*}
Concerning $\eta_N$, from \eqref{eqn:shift:Galerkin}, we see that
\begin{align*}
\frac{1}{2}\bdy_t\|\eta_N(t)\|^2_{\Mzeromu}
&=\la \Tcal_\mu \eta_N(t),\eta_N(t)\ra_{\Mzeromu}+\int_0^\infty\close \mu (s)\la A^{1/2}\eta_N(t;s),A^{1/2}V_N(t)\ra_H\d s\\
&\quad\ +\int_0^\infty\close \mu (s)\la A^{1/2}\eta_N(t;s),A^{1/2}\xi(t)\ra_H\d s.
\end{align*}
So, recalling $\Psi_0(u,\eta)=\frac{1}{2}\|u\|^2_H+\frac{1}{2}(1-\kappa)\|\eta\|^2_{\Mzeromu}$ as in \eqref{form:Psi_0}, we combine the above two identities and cancel the integrals to obtain
\begin{equation} \label{eqn:Psi_0(U_N,eta^n):Galerkin}
\begin{aligned}
&\bdy_t \Psi_0(V_N(t),\eta_N(t)) \\
&=-\kappa \|A^{1/2}V_N(t)\|^2_H+(1-\kappa)\la \Tcal \eta_N(t),\eta_N(t)\ra_{\Mzeromu}-\kappa\la A^{1/2}V_N(t),A^{1/2}\xi(t)\ra_H\\
&\quad\ +\la P_N\f(V_N(t)+\xi(t)),V_N(t)\ra_H+ (1-\kappa)\la \eta_N(t),\xi(t)\ra_{\Mzeromu}.
\end{aligned}
\end{equation}
We invoke inequality \eqref{ineq:<T.eta,eta>} to further bound
\begin{align*}
\la \Tcal_\mu \eta_N(t),\eta_N(t)\ra_{\Mzeromu}\le -\frac{\delta}{2}\|\eta_N(t)\|^2_{\Mzeromu}.
\end{align*}
Next, we employ Cauchy-Schwarz inequality to estimate
\begin{align*}
\la A^{1/2}V_N(t),A^{1/2}\xi(t)\ra_H \le \frac{1}{2}\|A^{1/2}V_N(t)\|^2_H+\frac{1}{2}\|A^{1/2}\xi(t)\|^2_H.
\end{align*}
Likewise,
\begin{align*}
\la \eta_N(t),\xi(t)\ra_{\Mzeromu} &\le \frac{\delta}{4}\|\eta_N(t)\|^2_{\Mzeromu}+\frac{1}{\delta}\|\xi(t)\|^2_{\Mzeromu}\\
&=\frac{\delta}{4}\|\eta_N(t)\|^2_{\Mzeromu}+\frac{\|\mu\|_{L^1(\rbb^+)}}{\delta}\|A^{1/2}\xi(t)\|^2_{H}.
\end{align*}
So that,
\begin{equation}\label{ineq:Psi_0(U_N,eta^n):Galerkin:a}
\begin{aligned}
&-\kappa\|A^{1/2}V_N(t)\|^2_H+(1-\kappa)\la \Tcal_\mu \eta_N(t),\eta_N(t)\ra_{\Mzeromu}-\kappa\la A^{1/2}V_N(t),A^{1/2}\xi(t)\ra_H\\
& +(1-\kappa) \la \eta_N(t),\xi(t)\ra_{\Mzeromu}\\
&\le -\frac{1}{2}\kappa\|A^{1/2}V_N(t)\|^2_H-\frac{\delta}{4}(1-\kappa)\|\eta_N(t)\|^2_{\Mzeromu}\\
&\quad\ +\Big(\frac{1}{2}\kappa+\frac{\|\mu\|_{L^1(\rbb^+)}}{\delta}(1-\kappa)\Big)\|A^{1/2}\xi(t)\|^2_H.
\end{aligned}
\end{equation}
To estimate the non-linear term involving $\f$ in~\eqref{eqn:Psi_0(U_N,eta^n):Galerkin}, we employ \nameref{cond:phi:1}--\nameref{cond:phi:2} to see that
\begin{align*}
&\la P_N\f(V_N(t)+\xi(t)),V_N(t)\ra_H\\
&=\la \f(V_N(t)+\xi(t)),V_N(t)+\xi(t)\ra_H-\la \f(V_N(t)+\xi(t)),\xi(t)\ra_H\\
&\le -a_2\|V_N(t)+\xi(t)\|^{p_0+1}_{L^{p_0+1}}+a_3|\domain|+a_1\la |V_N(t)+\xi(t)|^{p_0},|\xi(t)|\ra_H+a_1\|\xi(t)\|_{L^1(\domain)}.
\end{align*}
We thus deduce the estimate
\begin{align}\label{ineq:Psi_0(U_N,eta^n):Galerkin:b}
&\la P_N\f(V_N(t)+\xi(t)),V_N(t)\ra_H \notag \\
&\le -c\|V_N(t)\|^{p_0+1}_{L^{p_0+1}}+C\|\xi(t)\|^{p_0+1}_{L^{p_0+1}}+C\|\xi(t)\|^2_H+C.
\end{align}
Combining \eqref{eqn:Psi_0(U_N,eta^n):Galerkin}, \eqref{ineq:Psi_0(U_N,eta^n):Galerkin:a}, \eqref{ineq:Psi_0(U_N,eta^n):Galerkin:b}, integrating with respect to time and taking expectations, we arrive at the bound
\begin{align*}
&\E \sup_{0\le r\le t}\Psi_0(V_N(r),\eta_N(r))\\
& +\int_0^t\frac{1}{2}\kappa \E\|A^{1/2}V_N(r)\|^2_H+\frac{\delta}{4}(1-\kappa)\E\|\eta_N(r)\|^2_{\Mzeromu}+c\,\E\|V_N(r)\|^{p_0+1}_{L^{p_0+1}}\d r\\
&\le \Psi_0(P_N u_0,P_N\eta_0)+Ct+C\int_0^t\E\|\xi(r)\|^{p_0+1}_{L^{p_0+1}}+\E\|\xi(r)\|^2_{H^1} \d r,
\end{align*}
for some positive constants $c,C$ independent of time $t$. We note that by condition~\nameref{cond:Q}, $\text{Law}(\xi(t))$ is a Gaussian measure in $H^{1}$, since
\begin{align*}
\E\|\xi(t)\|^2_{H^{1}}=\E\Big\|\sum_{k\ge 1}Qe_k B_k(t)\Big\|^2_{H^{1}}=\Tr(QAQ^*)t<\infty.
\end{align*}
Observe that \nameref{cond:Q} also implies
\begin{align*}
    \sup_{x\in\domain}\E|\xi(t,x)|^2=    \sup_{x\in\domain}\E\Big| \sum_{k\ge 1}Q e_k(x) B_k(t) \Big|^2 =     \sup_{x\in\domain}\sum_{k\ge 1} |Qe_k(x)|^2 t <\infty.
\end{align*}
Since $\xi(x,t)$ is Gaussian, we infer the bound
\begin{align*}
   \E \|\xi(t)\|^{p_0+1}_{L^{p_0+1}}= \int_{\domain}\E |\xi(t,x)|^{p_0+1}\d x \le |\domain|\sup_{x\in\domain} \E|\xi(t,x)|^{p_0+1}\le c t^{\frac{p_0+1}{2}},
\end{align*}
where $c>0$ is independent of $t$. Combining the above estimate, we obtain
\begin{align*}
\int_0^t\E\|\xi(r)\|^{p_0+1}_{L^{p_0+1}}+\E\|\xi(r)\|^2_{H^1} \d r \le C(t,p_0),
\end{align*}
whence
\begin{equation} \label{ineq:Psi_0(U_N,eta^n):Galerkin:c}
\begin{aligned}
&\E \sup_{0\le r\le t}\Psi_0(V_N(r),\eta_N(r))\\
& +\int_0^t\frac{1}{2}\kappa \E\|A^{1/2}V_N(r)\|^2_H+\frac{\delta}{4}(1-\kappa)\E\|\eta_N(r)\|^2_{\Mzeromu}+c\,\E\|V_N(r)\|^{p_0+1}_{L^{p_0+1}}\d r\\
&\le\Psi_0(P_N u_0,P_N\eta_0)+C(t,p_0).
\end{aligned}
\end{equation}
By employing a standard argument for systems of stochastic differential equations (see \cite{albeverio2008spde,glatt2009strong}, for instance), the {truncated} system \eqref{eqn:shift:Galerkin} admits a unique global pathwise solution $(V_N,\eta_N)$ in $C([0,T];H\times\Mzeromu)$, for all $T>0$.

Furthermore, to produce a uniform bound on $\f(V_N(t)+\xi(t))$, we recall \nameref{cond:phi:1} and $q=(p_0+1)/p_0$. So that,
\begin{align}
\int_0^t \E\|\f(V_N(r)+\xi(t))\|^q_{L^{q}}\d r&\le \int_0^t \E\int_\domain\big(a_1(1+|V_N(r)|^{p_0}+|\xi(r)|^{p_0}\big)^q\d x\d r \notag\\
&\le c t+c\int_0^t\E\|V_N(r)\|^{p_0+1}_{L^{p_0+1}}+\|\xi(r)\|^{p_0+1}_{L^{p_0+1}}\d r.    \label{ineq:|f(U_N)+xi|_(L^q)}
\end{align}
The proof is thus complete.
\end{proof}

\subsection{Passage to the limit}

As a consequence of Lemma \ref{lem:Galerkin}, we deduce the following limits (up to a subsequence)
\begin{align*}
V_N &\rightharpoonup^* v \text{ in } L^2(\Omega;L^\infty(0,T;H) ),\\
V_N &\rightharpoonup v \text{ in } L^2(\Omega; L^2(0,T;H^1 )),\\
V_N &\rightharpoonup v\text{ in } L^{p_0+1}(\Omega; L^{p_0+1}(0,T;L^{p_0+1})),\\
\f(V_N+\xi )&\rightharpoonup \chi  \text{ in }  L^{q}(\Omega; L^{q}(0,T;L^{q})),\\
\eta_N &\rightharpoonup^* \eta \text{ in } L^2(\Omega;L^\infty(0,T;\Mzeromu)),\\
\eta_N &\rightharpoonup \eta \text{ in } L^2(\Omega; L^2(0,T;\Mzeromu )).
\end{align*}
Furthermore (see \cite[pg. 224]{robinson2001infinite}), we have
 $$P_N\f(V_N+\xi)\rightharpoonup \chi \text{ in } L^{q}(\Omega; L^{q}(0,T;L^{q})).$$
In particular, for all $\psi\in  H^1  \cap L^{p+1}$, it holds a.s. that
\begin{align*}
\la v(t),\psi\ra_H &= \la u_0,\psi\ra_H+ \int_0^t \la A^{1/2}v(r)+A^{1/2}\xi(r),A^{1/2}\psi\ra_H\d r +\int_0^t  \la \eta(r),\psi\ra_{\Mzeromu}\d r\\
&\quad\ +\int_0^t \la\chi(r),v\ra_H\d r .
\end{align*}
Also, by expression \eqref{form:eta(t)-representation}, observe that for all $\etatilde\in \Mzeromu$
\begin{align*}
\la \eta_N(t),\etatilde\ra_{\Mzeromu}& = \int_0^ t \int_0^s \la V_N(t-r)+P_N\xi(r)\d r, \etatilde(s)\ra_{H^1}\mu(s)\d s\\
&\quad +\int_t^\infty \la P_N\eta_0(s-t),\etatilde(s)\ra_{H^1}\mu(s)\d s\\
&\quad+ \int_t^\infty\close  \int_0^t\la V_N(t-r)+P_N\xi(r),\etatilde(s)\ra_{H^1}\d r\, \mu(s)\d s,
\end{align*}
which implies by taking limit as $n\to\infty$
\begin{align*}
\la \eta(t),\etatilde\ra_{\Mzeromu}& = \int_0^ t \int_0^s \la v(t-r)+\xi(r)\d r, \etatilde(s)\ra_{H^1}\mu(s)\d s\\
&\quad\ +\int_t^\infty \la \eta_0(s-t),\etatilde(s)\ra_{H^1}\mu(s)\d s\\
&\quad\ + \int_t^\infty\close  \int_0^t\la v(t-r)+\xi(r),\etatilde(s)\ra_{H^1}\d r\, \mu(s)\d s,
\end{align*}
by using for example Vitali convergence theorem. Since the above identity holds for all $\etatilde$, $\eta(t)$ indeed satisfies expression \eqref{form:eta(t)-representation}, which is equivalent to equation \eqref{eqn:eta:Cauchy-problem}, i.e.,
\begin{align*}
\bdy_t \eta(t)=\Tcal_\mu \eta(t)+v(t)+\xi(t).
\end{align*}

It remains to prove that a.s., $\chi(t) = \f(v(t)+\xi(t)) $, a.e. $t\in [0,T]$. To this end, fix $m_*$ large such that $H^{m_*}\subset L^{p_0+1}$. For $\psi\in H^{m_*}$, we observe that
\begin{align*}
\la \partial_t V_N(t),\psi\ra_H &= \la P_N u_0,\psi\ra_H+ \int_0^t \la A^{1/2}V_N(r)+A^{1/2}\xi(r),A^{1/2}\psi\ra_H\d r\\
&\quad\ +\int_0^t  \la \eta(r),\psi\ra_{\Mzeromu}\d r+\int_0^t \la P_N\f(V_N(r)+\xi(r)),\psi\ra_H\d r .
\end{align*}
In view of~\eqref{ineq:Psi_0(U_N,eta^n):Galerkin:c}--\eqref{ineq:|f(U_N)+xi|_(L^q)}, we see that $\partial_t V_N$ is uniformly bounded in $L^q(0,T;H^{-m_*})$ where $q=(p_0+1)/p_0$. By \cite[Theorem 8.1]{robinson2001infinite}, we obtain the strong convergence (up to a subsequence)
\begin{align*}
V_N &\to v \text{ in } L^2(0,T;H) .
\end{align*}
It follows that, up to a subsequence, $\f(V_N+\xi)$ converges to $\f(v+\xi)$ a.e. $(x,t)\in \domain\times [0,T]$ since $\f$ is continuous. Together with~\cite[Lemma 8.3]{robinson2001infinite}, we deduce
$$\f(V_N+\xi )\rightharpoonup \f(v+\xi)  \text{ in }   L^{q}(0,T;L^{q}),$$
whence a.s., $\chi=\f(u)$ a.e. $(x,t)\in \domain\times [0,T]$. This finishes the construction of a solution $(v,\eta)$ for~\eqref{eqn:shift}, thereby proving the existence of solutions $(u, \eta)$ for~\eqref{eqn:react-diff:mu}.

The uniqueness of such $(u,\eta)$ as well as their continuity with respect to initial conditions can be derived using an argument similarly to the above energy estimates~\eqref{ineq:Psi_0(U_N,eta^n):Galerkin:c}--\eqref{ineq:|f(U_N)+xi|_(L^q)}, and thus are omitted.

\section{Auxiliary results}\label{sec:auxiliary-result}

\begin{lemma}\label{lem:A^(n)phi(u)}
Suppose for $n\ge 1$, $\f\in C^{2n}(\rbb)$ and $u\in H^{2n}$. Then,
\begin{align} \label{form:A^(n)phi(u)}
\lap^n\f(u)& = \f'(u)\lap^n u+ c_n\f''(u)\grad u\cdot \grad \lap^{n-1}u +\f^{(2n)}(u)|\grad u|^{2n} \notag \\
&\quad\ +\sum_{i=2}^{2n-1}\f^{(i)}(u)\sum_{j=1}^{m_{i}}I_{ij}(u),
\end{align}
where $c_n>0$, $I_{ij}(u)=c_{ij}\prod_{k=1}^{m_{ij}}(D^{\alpha_{ijk}}u)^{n_{ijk}}$ with $c_{ij}>0$, $|\alpha_{ijk}|\le 2n-2$, and $\alpha_{ijk}\in \zbb^+$.
\end{lemma}

\begin{proof}
We proceed with induction on $n$ and start with the base cases $n=1$ and $n=2$. A straightforward calculation yields
\begin{align*}
A\f(u) = \f'(u)Au+\f''(u)|\grad u|^2,
\end{align*}
and
\begin{align*}
A^2\f(u)&=\f'(u) \lap^2 u+4 \f''(u)\grad u\cdot\grad\lap u+\f^{(4)}(u)|\grad u|^4\\
&\quad\ + \f''(u)\big[(\lap u)^2+2\sum_{i,j=1}^d (\partial_{ij}u)^2\big]+\f^{(3)}(u)\big[2|\grad u|^2\lap u+4\grad u\cdot \grad^2 u\cdot\grad u \big],
\end{align*}
which confirms \eqref{form:A^(n)phi(u)} for the base cases.

Suppose \eqref{form:A^(n)phi(u)} holds for up to $n\ge 2$ and consider the case $n+1$. A computation yields
\begin{align*}
&\lap\big(\f'(u)\lap^n u\big)  = \f'(u)\lap^{n+1}u+2\f''(u)\grad u\cdot\grad\lap^n u+\f^{(3)}(u)|\grad u|^2\lap^n u,
\\
&\lap \big(\f''(u)\grad u\cdot \grad\lap^{n-1}u\big)\\
&=\f''(u)\grad u\cdot\grad\lap^n u+ \f''(u)\big[\grad\lap\cdot\grad\lap^{n-1}u+2\sum_{i,j=1}^d\partial_{ij}u\partial_{ij}\lap^{n-1}u\big]\\
&\quad\ \f^{(3)}(u)\big[2\grad u\cdot \grad^2 u\cdot \grad\lap^{n-2}u+2\grad u\cdot \grad^2\lap^{n-1} u\cdot \grad u+\lap u\grad u\cdot\grad\lap^{n-1}u\big],
\end{align*}
and
\begin{align*}
&\lap\big( \f^{(2n)}(u)|\grad u|^{2n} \big)\\
& =\f^{(2n+2)}(u)|\grad u|^{2n+2}+\f^{2n+1}(u)\big[ |\grad u|^{2n}\lap u+|\grad u|^{2n}\grad u\cdot\grad^2 u\cdot\grad u\\
&\quad\ +2n\grad u\cdot\grad^2 u\cdot\grad u\big]\\
&\quad\ +2n\f^{(2n)}(u)\big[|\grad u|^{2n-2}\sum_{i,j=1}^d(\partial_{ij}u)^2 +|\grad u|^{2n-2}\grad u\cdot \grad\lap u\\
&\quad\ +(n-1)|\grad u|^{2n-4}|\grad u\cdot\grad^2 u|^2\big].
\end{align*}
Similarly, concerning the term $\f^{(i)}(u)I_{ij}(u)$ on the right-hand side of \eqref{form:A^(n)phi(u)}, we note that applying $\lap$ to $\f^{(i)}(u)I_{ij}(u)$ will yield a sum of terms having the form
\begin{align*}
\f^{(i')}(u)\prod_{k=1}^{m_{i'j'}}(D^{\alpha'_{ijk}}u)^{n'_{ijk}},
\end{align*}
where $2\le i'\le 2n+1$, $|\alpha_{i'j'k}|\le 2n$, since we are only taking at most two derivatives. We now combine with preceding identities to establish \eqref{form:A^(n)phi(u)} for $n+1$.
\end{proof}

\begin{lemma} \label{lem:phi(x)<x+x^p}
Under \nameref{cond:phi}--\nameref{cond:phi:3}, there exists a constant $C_\f>0$ such that
\begin{align*}
|\f(x)|\le C_\f(|x|+|x|^{p}),
\end{align*}
for all $x\in\rbb$.
\end{lemma}

\begin{proof}
We first pick $c_\f =\max_{-1\le x\le 1}|\f(x)\f''(x)+(\f'(x))^2| $ and consider the function
\begin{align*}
f(x)=2c_\f x^2-\f(x)^2.
\end{align*}
Since $\f(0)=0$, it is straightforward to verify that 0 is a critical point of $f$ defined above. Furthermore, the the choice of $c_\f$ yields
\begin{align*}
f''(x)=4c_\f-2\f(x)\f''(x)-2(\f'(x))^2>0,
\end{align*}
all $|x|\le 1$. As a consequence, $f(x)\ge f(0)=0$, for all $|x|\le 1$. It follows that
\begin{align*}
|\f(x)|\le \sqrt{2c_\f}|x|.
\end{align*}
On the other hand, by \nameref{cond:phi:1}, we infer that
\begin{align*}
|\f(x)|\le a_1(|x|^p+1)\le 2a_1|x|^p,
\end{align*}
for all $|x|\ge 1$. Altogether,  setting $C_\f=\max\{2c_\f,2a_1\}$ produces the desired bound.
\end{proof}

\bibliographystyle{abbrv}
{\footnotesize\bibliography{react-diff}}

\begin{thebibliography}{10}

\bibitem{albeverio2008spde}
S.~Albeverio, F.~Flandoli, and Y.~G. Sinai.
\newblock {\em {SPDE in Hydrodynamics: Recent Progress and Prospects: Lectures
  given at the CIME Summer School held in Cetraro, Italy, August 29-September
  3, 2005}}.
\newblock Springer, 2008.

\bibitem{baeumer2015existence}
B.~Baeumer, M.~Geissert, and M.~Kov{\'a}cs.
\newblock {Existence, uniqueness and regularity for a class of semilinear
  stochastic Volterra equations with multiplicative noise}.
\newblock {\em J. Differ. Equ.}, 258(2):535--554, 2015.

\bibitem{bakhtin2005stationary}
Y.~Bakhtin and J.~C. Mattingly.
\newblock Stationary solutions of stochastic differential equations with memory
  and stochastic partial differential equations.
\newblock {\em Commun. Contemp. Math}, 7(05):553--582, 2005.

\bibitem{barbu1975nonlinear}
V.~Barbu.
\newblock {Nonlinear Volterra equations in a Hilbert space}.
\newblock {\em SIAM J. Math. Anal.}, 6(4):728--741, 1975.

\bibitem{barbu52nonlinear}
V.~Barbu.
\newblock {\em {Nonlinear Semigroups and Differential Equations in Banach
  Spaces}}.
\newblock Noordhoff, Leyden, The Netherlands, 1976.

\bibitem{barbu1979existence}
V.~Barbu.
\newblock {Existence for nonlinear Volterra equations in Hilbert spaces}.
\newblock {\em SIAM J. Math. Anal.}, 10(3):552--569, 1979.

\bibitem{bonaccorsi2012asymptotic}
S.~Bonaccorsi, G.~Da~Prato, and L.~Tubaro.
\newblock Asymptotic behavior of a class of nonlinear stochastic heat equations
  with memory effects.
\newblock {\em SIAM J. Math. Anal.}, 44(3):1562--1587, 2012.

\bibitem{bonaccorsi2004large}
S.~Bonaccorsi and M.~Fantozzi.
\newblock {Large deviation principle for semilinear stochastic Volterra
  equations}.
\newblock {\em Dyn. Syst. Appl.}, 13:203--220, 2004.

\bibitem{bonaccorsi2006infinite}
S.~Bonaccorsi and M.~Fantozzi.
\newblock {Infinite dimensional stochastic Volterra equations with dissipative
  nonlinearity}.
\newblock {\em Dyn. Syst. Appl.}, 15(3/4):465, 2006.

\bibitem{caraballo2007existence}
T.~Caraballo, I.~Chueshov, P.~Mar{\'\i}n-Rubio, and J.~Real.
\newblock Existence and asymptotic behaviour for stochastic heat equations with
  multiplicative noise in materials with memory.
\newblock {\em Discrete Contin. Dyn. Syst. - A}, 18(2\&3):253, 2007.

\bibitem{caraballo2008pullback}
T.~Caraballo, J.~Real, and I.~Chueshov.
\newblock Pullback attractors for stochastic heat equations in materials with
  memory.
\newblock {\em Discrete Contin. Dyn. Syst. - B}, 9(3\&4, May):525, 2008.

\bibitem{cerrai2020convergence}
S.~Cerrai and N.~Glatt-Holtz.
\newblock On the convergence of stationary solutions in the
  smoluchowski-kramers approximation of infinite dimensional systems.
\newblock {\em J. Funct. Anal.}, 278(8):108421, 2020.

\bibitem{clement1996some}
P.~Cl{\'e}ment and G.~Da~Prato.
\newblock {Some results on stochastic convolutions arising in Volterra
  equations perturbed by noise}.
\newblock {\em Atti Accad. Naz. Lincei Cl. Sci. Fis. Mat. Natur. Rend. Lincei
  Mat. Appl.}, 7(3):147--153, 1996.

\bibitem{clement1997white}
P.~Cl{\'e}ment and G.~Da~Prato.
\newblock White noise perturbation of the heat equation in materials with
  memory.
\newblock {\em Dynam. Systems Appl.}, 6:441--460, 1997.

\bibitem{clement1998white}
P.~Cl{\'e}ment, G.~Da~Prato, and J.~Pr{\"u}ss.
\newblock White noise perturbation of the equations of linear parabolic
  viscoelasticity.
\newblock {\em Rend. Istit. Mat. Univ. Trieste}, 29:207--220, 1998.

\bibitem{conti2005singular}
M.~Conti and V.~Pata.
\newblock Singular limit of dissipative hyperbolic equations with memory.
\newblock {\em Conference Publications}, 2005(Special):200--208, 2005.

\bibitem{conti2006singular}
M.~Conti, V.~Pata, and M.~Squassina.
\newblock Singular limit of differential systems with memory.
\newblock {\em Indiana Univ. Math. J.}, pages 169--215, 2006.

\bibitem{da2014stochastic}
G.~Da~Prato and J.~Zabczyk.
\newblock {\em {Stochastic Equations in Infinite Dimensions}}.
\newblock Cambridge University Press, 2014.

\bibitem{dafermos1970asymptotic}
C.~M. Dafermos.
\newblock Asymptotic stability in viscoelasticity.
\newblock {\em Arch. Ration. Mech. Anal.}, 37(4):297--308, 1970.

\bibitem{desch2011p}
W.~Desch and S.-O. Londen.
\newblock {An L p-theory for stochastic integral equations}.
\newblock {\em J. Evol. Equ.}, 11(2):287--317, 2011.

\bibitem{weinan2002gibbsian}
W.~E and D.~Liu.
\newblock {Gibbsian dynamics and invariant measures for stochastic dissipative
  PDEs}.
\newblock {\em J. Stat. Phys.}, 108(5-6):1125--1156, 2002.

\bibitem{weinan2001gibbsian}
W.~E, J.~C. Mattingly, and Y.~Sinai.
\newblock {Gibbsian Dynamics and Ergodicity for the Stochastically Forced
  Navier--Stokes Equation}.
\newblock {\em Commun. Math. Phys.}, 224(1):83--106, 2001.

\bibitem{engel2001one}
K.-J. Engel, R.~Nagel, and S.~Brendle.
\newblock {\em {One-parameter Semigroups for Linear Evolution Equations}},
  volume 194.
\newblock Springer, 2000.

\bibitem{gatti2005memory}
S.~Gatti, M.~Grasselli, A.~Miranville, and V.~Pata.
\newblock Memory relaxation of first order evolution equations.
\newblock {\em Nonlinearity}, 18(4):1859, 2005.

\bibitem{gatti2004exponential}
S.~Gatti, M.~Grasselli, and V.~Pata.
\newblock Exponential attractors for a phase-field model with memory and
  quadratic nonlinearities.
\newblock {\em Indiana Univ. Math. J.}, pages 719--753, 2004.

\bibitem{giorgi1999uniform}
C.~Giorgi, M.~Grasselli, and V.~Pata.
\newblock Uniform attractors for a phase-field model with memory and quadratic
  nonlinearity.
\newblock {\em Indiana Univ. Math. J.}, pages 1395--1445, 1999.

\bibitem{glatt2021long}
N.~Glatt-Holtz, V.~R. Martinez, and G.~H. Richards.
\newblock {On the long-time statistical behavior of smooth solutions of the
  weakly damped, stochastically-driven KdV equation}.
\newblock {\em to appear in Trans. Am. Math. Soc., arXiv preprint
  arXiv:2103.12942}, 2021.

\bibitem{glatt2017unique}
N.~Glatt-Holtz, J.~C. Mattingly, and G.~Richards.
\newblock On unique ergodicity in nonlinear stochastic partial differential
  equations.
\newblock {\em J. Stat. Phys.}, 166(3-4):618--649, 2017.

\bibitem{glatt2008stochastic}
N.~Glatt-Holtz and M.~Ziane.
\newblock The stochastic primitive equations in two space dimensions with
  multiplicative noise.
\newblock {\em Discrete Contin. Dyn. Syst. - B}, 10(4):801, 2008.

\bibitem{glatt2009strong}
N.~Glatt-Holtz and M.~Ziane.
\newblock {Strong pathwise solutions of the stochastic Navier-Stokes system}.
\newblock {\em Adv. Differ. Equ.}, 14(5/6):567--600, 2009.

\bibitem{glatt2020generalized}
N.~E. Glatt-Holtz, D.~P. Herzog, S.~A. McKinley, and H.~D. Nguyen.
\newblock {The generalized Langevin equation with power-law memory in a
  nonlinear potential well}.
\newblock {\em Nonlinearity}, 33(6):2820, 2020.

\bibitem{glatt2022short}
N.~E. Glatt-Holtz, V.~R. Martinez, and H.~D. Nguyen.
\newblock {The short memory limit for long time statistics in a stochastic
  Coleman-Gurtin model of heat conduction}.
\newblock {\em arXiv preprint arXiv:2212.05646}, 2022.

\bibitem{grasselli2002uniform}
M.~Grasselli and V.~Pata.
\newblock Uniform attractors of nonautonomous dynamical systems with memory.
\newblock In {\em {Evolution Equations, Semigroups and Functional Analysis}},
  pages 155--178. Springer, 2002.

\bibitem{hairer2008spectral}
M.~Hairer and J.~C. Mattingly.
\newblock {Spectral gaps in Wasserstein distances and the 2D stochastic
  Navier--Stokes equations}.
\newblock {\em Ann. Prob.}, 36(6):2050--2091, 2008.

\bibitem{hairer2011asymptotic}
M.~Hairer, J.~C. Mattingly, and M.~Scheutzow.
\newblock {Asymptotic coupling and a general form of Harris' theorem with
  applications to stochastic delay equations}.
\newblock {\em Probab. Theory Relat. Fields}, 149(1):223--259, 2011.

\bibitem{ito1964stationary}
K.~It{\^o} and M.~Nisio.
\newblock On stationary solutions of a stochastic differential equation.
\newblock {\em J. Math. Kyoto Univ.}, 4(3):1--75, 1964.

\bibitem{joseph1989heat}
D.~D. Joseph and L.~Preziosi.
\newblock {Heat waves}.
\newblock {\em Rev. Mod. Phys.}, 61(1):41, 1989.

\bibitem{joseph1990heat}
D.~D. Joseph and L.~Preziosi.
\newblock {Addendum to the paper: ``Heat waves" [Rev. Modern Phys. 61 (1989),
  no. 1, 41-73]}.
\newblock {\em Rev. Mod. Phys.}, 62:375--391, 1990.

\bibitem{karatzas2012brownian}
I.~Karatzas and S.~Shreve.
\newblock {\em {Brownian Motion and Stochastic Calculus}}, volume 113.
\newblock Springer Science \& Business Media, 2012.

\bibitem{li2019asymptotic}
L.~Li, J.~Shu, Q.~Bai, and H.~Li.
\newblock Asymptotic behavior of fractional stochastic heat equations in
  materials with memory.
\newblock {\em Appl. Anal.}, pages 1--22, 2019.

\bibitem{liu2017well}
L.~Liu and T.~Caraballo.
\newblock Well-posedness and dynamics of a fractional stochastic
  integro-differential equation.
\newblock {\em Physica D}, 355:45--57, 2017.

\bibitem{liu2019asymptotic}
Y.~Liu, W.~Liu, X.-G. Yang, and Y.~Zheng.
\newblock {Asymptotic behavior for 2D stochastic Navier-Stokes equations with
  memory in unbounded domains}.
\newblock {\em arXiv preprint arXiv:1903.07251}, 2019.

\bibitem{miller1974linear}
R.~Miller.
\newblock {Linear Volterra integrodifferential equations as semigroups}.
\newblock {\em Funkcial. Ekvac}, 17:39--55, 1974.

\bibitem{nguyen2018small}
H.~D. Nguyen.
\newblock {The small-mass limit and white-noise limit of an infinite
  dimensional generalized Langevin equation}.
\newblock {\em J. Stat. Phys.}, 173(2):411--437, 2018.

\bibitem{nguyen2022ergodicity}
H.~D. Nguyen.
\newblock Ergodicity of a nonlinear stochastic reaction--diffusion equation
  with memory.
\newblock {\em Stoch. Process. Their Appl.}, 155:147--179, 2023.

\bibitem{ottobre2011asymptotic}
M.~Ottobre and G.~A. Pavliotis.
\newblock {Asymptotic analysis for the generalized Langevin equation}.
\newblock {\em Nonlinearity}, 24(5):1629, 2011.

\bibitem{pata2001attractors}
V.~Pata and A.~Zucchi.
\newblock Attractors for a damped hyperbolic equation with linear memory.
\newblock {\em Adv. Math. Sci. Appl.}, 11:505--529, 2001.

\bibitem{pavliotis2014stochastic}
G.~A. Pavliotis.
\newblock {\em {Stochastic Processes and Applications: Diffusion Processes, the
  Fokker-Planck and Langevin Equations}}, volume~60.
\newblock Springer, 2014.

\bibitem{robinson2001infinite}
J.~C. Robinson.
\newblock {\em {Infinite-dimensional Dynamical Systems: an Introduction to
  Dissipative Parabolic PDEs and the Theory of Global Attractors}}, volume~28.
\newblock Cambridge University Press, 2001.

\bibitem{shangguan2024geometric}
D.~Shangguan, Q.~Zhang, J.~Hu, and X.~Li.
\newblock Geometric ergodicity of a stochastic reaction--diffusion tuberculosis
  model with varying immunity period.
\newblock {\em J. Nonlinear Sci.}, 34(6):114, 2024.

\bibitem{shu2020asymptotic}
J.~Shu, H.~Li, X.~Huang, and J.~Zhang.
\newblock Asymptotic behaviour of stochastic heat equations in materials with
  memory on thin domains.
\newblock {\em Dyn. Syst.}, pages 1--25, 2020.

\end{thebibliography}

\end{document}